\documentclass[nonblindrev,copyedit]{informs3} 

\UseRawInputEncoding

\usepackage{graphicx}
\usepackage{amsfonts}
\usepackage{subfigure}
\usepackage{longtable}
\usepackage{rotating}
\usepackage{epstopdf}
\usepackage{color}
\usepackage{multirow}
\usepackage{bbm}
\usepackage{setspace}
\usepackage{enumerate}
\usepackage{diagbox}
\usepackage{booktabs}
\usepackage{makecell}

\usepackage{natbib}
\usepackage{hyperref}
\hypersetup{
    colorlinks=true,
    linkcolor=black,
    filecolor=black,
    urlcolor=black,
    citecolor=blue,
}

\usepackage{endnotes}
\makeatletter
\def\enoteheading{\section*{\notesname
		\@mkboth{\MakeUppercase{\notesname}}{\MakeUppercase{\notesname}}}%
	\noindent
}
\makeatother

\usepackage{lipsum}
\makeatletter
\renewcommand\footnoterule{%
  \kern-3\p@
  \hrule\@width 6.5in
  \kern2.6\p@}
\makeatother

 \bibpunct[, ]{(}{)}{,}{a}{}{,}%
 \def\bibfont{\fontsize{9}{11}\selectfont} 
 \def\bibsep{\smallskipamount} 
 \def\bibhang{24pt}%

 \def\argmax{{\rm argmax}\,}
 \def\argmin{{\rm argmin}\,}

 \def\do{{\rm dom }\,}

 \def\today{\ifcase\month\or January\or February\or  March\or  April\or
 May\or June\or July\or August\or  September\or October\or November\or
 December\fi  \space\number\day, \number\year}

 \def\qed{\hspace*{10pt}\hfill{$\square$}\hfilneg\par}
 \def\sqr#1#2{{\vcenter{\hrule height .#2pt
       \hbox{\vrule width .#2pt height#1pt \kern#1pt\vrule width.#2pt}
                        \hrule height.#2pt}}}
 \def\square{\mathchoice\sqr54\sqr54\sqr{2.1}3\sqr{1.5}3}

 \theoremstyle{TH}
 \newtheorem{thm}{Theorem}
 \newtheorem{lem}{Lemma}
 \newtheorem{prop}{Proposition}
 
 \newtheorem{exap}{Example}

 \newcommand{\enu}{\begin{enumerate}}
 \newcommand{\eenu}{\end{enumerate}}
 \newcommand{\des}{\begin{description}}
 \newcommand{\edes}{\end{description}}
 \newcommand{\iit}{\begin{itemize}}
 \newcommand{\eiit}{\end{itemize}}
 \newcommand{\tab}{\begin{tabular}}
 \newcommand{\etab}{\end{tabular}}
 \newcommand{\ben}{\begin{equation}}   \newcommand{\een}{\end{equation}}
 \newcommand{\bea}{\begin{eqnarray}}   \newcommand{\eea}{\end{eqnarray}}
 \newcommand{\beq}{\begin{eqnarray*}}  \newcommand{\eeq}{\end{eqnarray*}}
\newcommand{\PaperTitle}{From Optimization to Satisficing: Robust Screening under Distributional Ambiguity}

\TheoremsNumberedThrough     

\EquationsNumberedThrough    

\begin{document}



\RUNTITLE{\PaperTitle}

\TITLE{\PaperTitle}

\ARTICLEAUTHORS{%

\AUTHOR{Shumin Ma}
\AFF{Guangdong Provincial/Zhuhai Key Laboratory of IRADS, and Department of Mathematical Sciences, Beijing 
Normal-Hong Kong Baptist University, Zhuhai, 519087, China.
\EMAIL{shuminma@bnbu.edu.cn}}

\AUTHOR{Lijian Lu}
\AFF{HKUST Business School, Hong Kong University of Science and Technology. \EMAIL{lijianlu@ust.hk}}

\AUTHOR{Daniel Zhuoyu Long}
\AFF{Department of Systems Engineering and Engineering Management, The Chinese University of Hong Kong, Hong Kong, China.
\EMAIL{zylong@se.cuhk.edu.hk}}

} 

\ABSTRACT{%
\textbf{Problem definition:} This study investigates a robust screening problem under distributional ambiguity, where a seller is uncertain about a buyer's true valuation distribution, knowing only that it lies near a reference distribution measured by the Wasserstein metric. Traditional robust optimization (RO) approaches prioritize maximizing worst-case revenue within predefined ambiguity sets, often yielding seller-centric outcomes and reliance on precise set specifications. \textbf{Methodology/results:} We propose a robust satisficing (RS) framework aimed at attaining a specified revenue target by minimizing the worst-case shortfall across all potential distributions. Our approach offers a tractable formulation and detailed characterization of optimal mechanisms using randomized pricing strategies. We also assess the out-of-sample efficacy of a simple posted pricing mechanism, finding it particularly effective with lower targets and positively skewed valuations, where smaller valuations have high probability mass. Comparing RO with RS, we find that RS consistently enhances buyer surplus when the reference distribution has an increasing hazard rate and increases out-of-sample seller revenue with positively skewed true valuations. \textbf{Managerial implications:} Our analysis indicates that a target-driven RS framework enhances buyer surplus and fairness by offering more opportunities to lower-valuation buyers, potentially boosting overall revenue in scenarios with demand skewed toward these valuations. This approach offers a practical and viable modeling alternative to conventional RO methods, effectively overcoming the challenges of ambiguity set calibration while ensuring broader equitable access for diverse buyers.
}%


\HISTORY{\today}

\maketitle

%


\vspace{-0.2in}

\section{Introduction}

In the classic monopoly mechanism design problem, sellers aim to maximize expected revenue from selling a product to buyers with private, unknown valuations. When valuation distributions are completely known, deterministic posted pricing mechanisms are optimal \citep{myerson1981optimal,riley1981optimal}. However, in practice, acquiring precise distributional information is challenging and costly. Sellers often lack reliable historical data, and valuations can fluctuate due to market dynamics, competition, or external shocks. Relying on potentially inaccurate distributional assumptions can lead to significant revenue losses. An ubiquitous approach to this issue is robust optimization (RO), where sellers plan for the worst-case distribution within an ambiguity set. While RO provides a viable way to address the challenges of unknown distribution, its success hinges on the specification of the ambiguity set. A too-narrow set remains vulnerable to misspecification, while a too-broad set becomes overly conservative. Furthermore, calibrating the size of this ambiguity set is data-intensive, especially in small-sample or nonstationary environments.

To address these challenges, this paper introduces the robust satisficing (RS) paradigm as an alternative approach for designing effective mechanisms in the face of distributional ambiguity. Originating from recent advances in decision theory \citep{long2023robust,Sim2025RS}, RS shifts focus from ambiguity-set calibration in RO to target-driven robustness. In this framework, decision-makers set a revenue target and minimize the worst-case shortfall from it across all potential distributions, aligning the design problem with managerial goals centered on target attainment and removing the need to specify an ambiguity set. To illustrate the key difference between these approaches, consider a market where many buyers have low valuations, and fewer have high valuations. An RO approach concentrates on high-valuation buyers to safeguard revenue against worst-case scenarios. In contrast, RS may allocate more frequently to low-valuation buyers to ensure the revenue target is met. This approach not only broadens access for low-valuation buyers but also increases buyer surplus through distinct allocation patterns.

Despite their differing philosophies in addressing model uncertainties, the theoretical connections between RO and RS in mechanism design have been limited. This paper seeks to uncover their theoretical implications and deepen understanding of their practical performance. Our analysis shows that both models result in the same family of mechanisms, implementable through randomized pricing strategies with a piecewise linear payment function of the valuation. Although both paradigms belong to the same class of selling mechanisms, they differently influence allocation decisions and the distribution of surplus across buyers. The RO approach prioritizes protection against worst-case revenue loss, typically concentrating allocation on higher-valuation buyers. In contrast, the RS framework focuses on meeting a target and allows for greater deviations in adverse scenarios, expanding access for lower-valuation buyers and enhancing buyer surplus and fairness. When true valuation distributions concentrate on low valuations or fall outside typical RO ambiguity sets, the target-driven RS framework can improve revenue performance compared to RO. Thus, these robustness paradigms present different trade-offs between protection, fairness, and revenue.

Our first contribution is a comprehensive characterization of the optimal mechanism within the RS framework. We show that the infinite-dimensional satisficing problem can be reformulated into a tractable closed-form solution. The optimal mechanism features a continuum of randomized prices with allocation probabilities that follow a logarithmic pattern. This structure mirrors the optimal mechanism derived in the RO framework \citep{chen2024screening,wang2024minimax}. Notably, we identify a specific condition involving two exogenous hyperparameters -- ambiguity size in the RO framework and revenue target in the RS framework -- under which these two robust paradigms yield the same optimal mechanism. To the best of our knowledge, this is the first equivalence result between these two robust approaches in the mechanism design literature with model uncertainties.

Our second contribution is a detailed analysis of the effectiveness of simple deterministic posted pricing strategies within the RS framework. We explicitly characterize the optimal posted price for both continuous and discrete reference distributions, including empirical distributions, and evaluate its performance against the optimal randomized mechanism. Although posted pricing is often favored for its simplicity and is optimal when model uncertainties are absent, our findings indicate that it offers weaker robustness guarantees compared to the optimal RS mechanism when model uncertainties are present. Specifically, for heavy-tailed distributions such as power distributions, the optimal mechanism -- employing randomized pricing strategies -- consistently leads to higher buyer surplus across all potential valuations, resulting in a Pareto improvement over the simple deterministic posted pricing strategy.

Our third contribution is the first systematic comparison of buyer surplus across two different robustness paradigms in mechanism design. While previous research has primarily focused on revenue performance, we demonstrate that RS and RO can lead to significantly different buyer outcomes. We identify specific conditions under which RS surpasses RO in terms of buyer surplus and show that these differences in allocation can also lead to higher revenue when demand is concentrated among low-valuation buyers. This underscores a fundamental trade-off between protection and access: RO emphasizes worst-case revenue, whereas RS focuses on reliable target attainment and broader allocation. To our knowledge, this is the first comparison in the mechanism design literature between RO and RS approaches that highlights buyer surplus.

Finally, we conduct numerical experiments to assess the out-of-sample seller revenue for the optimal mechanisms derived from the two robust approaches. Our analysis indicates that the RS framework is particularly effective when true valuations are positively skewed, concentrating probability mass on lower valuations, or when the ambiguity size is small due to abundant historical data. In contrast, the RO framework performs better in other scenarios. From a managerial perspective, the target-driven RS framework is especially attractive when (i) information ambiguity is minimal, (ii) true valuations are positively skewed, and (iii) equitable access or fairness is a priority. By eliminating the need for ambiguity-set calibration and providing straightforward, target-driven decision rules, the RS framework offers a practical and implementable alternative to traditional robust optimization methods.

The remainder of the paper is organized as follows: Section \ref{sec-literature} reviews the related literature. Section \ref{Sec_Setup} introduces the model and the RS framework. Section \ref{sec-rs-mechanism} characterizes the optimal mechanism and explores the efficacy of simple deterministic posted pricing strategies. Section \ref{sec-RS-vs-RO} compares RS and RO, emphasizing buyer surplus and allocation patterns. Section \ref{sec-conclude} concludes the paper with discussions on future research directions. All proofs are provided in the appendix.

\section{Literature Review}\label{sec-literature}
Our work connects with three key areas of literature: robust mechanism design under demand uncertainty, the analysis of simple versus optimal mechanisms under ambiguity, and the emerging paradigm of robust satisficing. We position our contribution at the intersection of these streams by introducing a target-driven robustness paradigm into mechanism design and exploring its implications for both revenue and allocation outcomes.

First, our work builds on the literature surrounding robust mechanism design where the seller operates with incomplete demand information. This literature typically employs a \emph{maximin revenue} framework, where an adversary selects the worst-case distribution from a predefined ambiguity set, and the seller then determines the optimal selling mechanism to maximize worst-case revenue. In a single-item mechanism design setting (which is also known as the \emph{screening} problem, see \citealt{borgers2015introduction}), \citet{bergemann2011robust} consider the ambiguity set with the Prohorov metric and found deterministic posted pricing optimal. Other notable contributions include \citet{carrasco2018optimal} and \citet{pinar2017robust}, who examine robust revenue maximization with moment information, and \citet{li2019revenue}, who use the Wasserstein metric to define the ambiguity set. \citet{chen2024screening} further explore optimal mechanisms across a broad range of ambiguity sets. This robust mechanism design framework has also been extended to multiproduct or multi-buyer scenarios \citep{Carroll2017, koc2020distributionally, Che2025}. However, this body of work primarily focuses on the robustness of the seller's revenue, providing limited insights into how robustness considerations impact buyer surplus or the distribution of gains.

An alternative stream of research uses the \emph{minimax regret} criterion, as explored by \cite{bergemann2008pricing, bergemann2011robust,caldentey2017intertemporal,anunrojwong2025robustness}, and \cite{koccyiugit2024regret}. In this framework, an adversary first selects the worst-case distribution from a predefined ambiguity set, and the seller chooses the optimal selling mechanism to minimize worst-case regret. This regret measures the difference between the revenue achievable with complete knowledge of the true valuation distribution and the revenue earned without this information. Other research, such as \cite{eren2010monopoly, wang2024minimax,wang2025multi,wang2025power}, and \cite{anunrojwong2023robust}, uses the \emph{maximin competitive ratio} criterion to derive robust selling mechanisms. This approach is similar to minimax regret but measures the revenue achieved under distributional ambiguity relative to the revenue obtained with complete distributional information. \cite{anunrojwong2024best} offers a unified analysis of three key robust criteria across different information structures, assessing how mechanisms optimized for each criterion perform relative to each other. These studies highlight a fundamental prerequisite: the precise calibration of the ambiguity set's size or structure, which is highly sensitive to specification. Our approach diverges from these robust approaches by adopting an RS objective that bypasses ambiguity set specification. Instead, we minimize the system's fragility relative to a revenue target across all distributions, thus avoiding the complex task of set calibration inherent in classical RO approaches.

Second, our analysis relates to research on deterministic posted pricing mechanisms, valued for simplicity and optimality under complete distributional information \citep{riley1981optimal}. Under uncertainty, much work justifies its near-optimality or derives robust price guarantees with limited information, as seen in \citet{cohen2021simple,elmachtoub2021value,chen2023model}. Other contributions include \citet{chen2022distribution}, who derive the closed-form price and performance bound, and \citet{allouah2023optimal} who characterize the optimal pricing with quantile information, and \citet{chen2023robust} who study the robust pricing model with asymmetric valuation distribution. In contrast, we characterize the optimal mechanism under the RS framework and assess deterministic posted pricing, allowing for direct comparison in terms of both revenue and buyer surplus, and quantifying the value of randomization under a satisficing objective.

Third, we contribute to the literature on the RS framework, which aims to minimize a model's fragility in achieving prescribed targets \citep{long2023robust}. RS has been applied in diverse operational contexts, including assortment optimization \citep{jin2022distributionally}, advance scheduling \citep{zhou2022advance}, resource pooling \citep{cui2023target}, supervised learning models \citep{sim2021new}, portfolio optimization \citep{xue2025robust} and prescriptive analytics \citep{Sim2025RS}. Recently, \cite{wang2025equivalence} established connections between the RO and RS frameworks in data-driven robust optimization, such as network lot-sizing and portfolio construction, with the Wasserstein metric. The work most closely related to ours is \cite{rujeerapaiboon2023target}, which uses a RS framework to explore robust monopoly pricing. We extend this research into a mechanism design setting focused on revenue maximization, offering the first comprehensive comparison between RS and RO mechanisms within robust mechanism design. Our study emphasizes the implications for buyer surplus, fairness, and out-of-sample performance.

In summary, our paper integrates the RS paradigm into mechanism design, offering a unified analysis of its implications for both revenue and buyer outcomes. By comparing RS and RO within the same framework, we demonstrate that the choice of robustness paradigm affects allocation patterns and surplus distribution among buyers. Our work not only circumvents the need for ambiguity set specification but also highlights a new dimension of robustness, its impact on fairness and buyer surplus, largely overlooked in existing literature.

{\bf Notation conventions:} In this paper, we use $[N]$ to represent the set $\{1,2,\ldots,N\}$, $x^+ = \max\{x,0\}$ and $|x| = \max\{x,-x\}$ for any real value $x\in\mathbb{R}$. The indicator function $\mathbbm{1}(A)=1$ if event $A$ occurs, and $0$ otherwise. We denote the set of all distributions supported on $[0,1]$ as $\mathcal{P}$, using $P$ for the cumulative distribution function (CDF) and $\bar{P}$ for the complementary cumulative distribution function (CCDF). Additionally, we denote $\bar{P}_{-}(x) = \lim_{v \uparrow x}\bar{P}(v)$ to represent the left-limit, indicating the right-tail probability. The expectation with respect to a CDF $P$ is written as $\mathbb{E}_P[\cdot]$. The terms `decreasing' or `increasing' refer to weaker monotonicity, representing non-increasing or non-decreasing, respectively.

\section{Model Setup: A Distributionally Free Satisficing Framework}\label{Sec_Setup}

Consider the classical selling mechanism design in which a seller sells a single product to a buyer. For clarity, we use masculine pronouns for the buyer and feminine pronouns for the seller. The buyer is risk-neutral and seeks to maximize his expected utility. His valuation, or willingness to pay, for the product is \emph{private} information, represented by a random variable $\tilde{v}\in[0,1]$, drawn from a distribution $P\in \mathcal{P}$, where $P(x)={\mathbb P}(\tilde{v}\le x)$ represents the cumulative distribution function (CDF) and $\mathcal{P}$ denotes the set of all distributions supported on $[0,1]$. We denote a realization of the random variable $\tilde{v}$ as $v$ (without the tilde). 

The seller knows only the support of the buyer's valuation but does not have information about the true distribution $P$. To maximize expected revenue, the seller aims to design a robustly optimal mechanism. According to the revelation principle \citep{myerson1981optimal}, this can be restricted to the following set of direct mechanisms:
\beq
\mathcal{M} = \left\{ (q,m): [0,1] \rightarrow \left[0,1\right] \times \mathbb{R} \bigg| 
\begin{array}{ll}
    q(v)v-m(v)\geq 0,  &\forall v\in[0,1],   \\
    q(v)v-m(v)\geq q(\omega)v-m(\omega), &\forall v,\omega \in [0,1], 
\end{array}
\right\},
\eeq
where $q(\cdot)\in[0,1]$ denotes an allocation rule, and $m(\cdot)\in \mathbb{R}$ denotes a monetary payment rule. In a direct mechanism, the seller allocates the product to the buyer with probability $q(v)$ and requests a payment $m(v)$ based on the buyer's valuation $v$. A feasible mechanism must satisfy two key constraints: \emph{individual rationality (IR)}, ensuring voluntary participation of the buyer, and \emph{incentive compatibility (IC)}, ensuring truthfully reporting of the valuation.

One simple but elegant mechanism is the \emph{posted price (PP) mechanism}, in which the seller offers a take-it-or-leave-it deterministic price $p\in [0,1]$, taking the following form:
\be \label{Prob_PP}
q(v|p) = \mathbbm{1}(v\geq p) \ \text{and} \  m(v|p) = p\mathbbm{1}(v\geq p).
\tag{PP}
\ee
A more general mechanism is the \emph{randomized price (RP)} mechanism, which uses a random allocation mechanism with $q(v)\in[0,1]$. In this approach, the seller selects a probability distribution $Q(\cdot)$ for the randomized price $\tilde{p}\sim Q$. This results in specific allocation and payment rules as follows:
\be \label{Prob_RP}
q(v|Q) = \mathbb{E}_{\tilde{p}\sim Q}[\mathbbm{1}(v\geq \tilde{p})] \ \text{and} \ m(v|Q) =\mathbb{E}_{\tilde{p}\sim Q}[ \tilde{p}\mathbbm{1}(v\geq \tilde{p})].
\tag{RP}
\ee
Both PP and RP mechanisms clearly satisfy the IC and IR constraints. While it is well established that a PP mechanism maximizes revenue when the seller knows the buyer's valuation distribution, the uncertainty surrounding this distribution in practice can lead to substantial revenue losses.

\textbf{Wasserstein Ambiguity Set.} 
In this work, we focus on the \emph{Wasserstein ambiguity set} defined with the Wasserstein distance:
\beq
\mathcal{W}(P_0,r):=\{P\in \mathcal{P}: \; d(P,P_0)\leq r\},
\eeq
where $r\ge 0$ is the \emph{radius} of the Wasserstein balls, indicating the size of the ambiguity set, and $d(P,P_0)$ denotes the type-1 Wasserstein distance between distributions $P$ and $P_0$, defined as \endnote{Note that the type-1 Wasserstein distance has following two equivalent definitions. First, it can be equivalently defined as $d(P,P_0) = \int_0^1 \vert \bar{P}^{-1}(x)-\bar{P}^{-1}_0(x)\vert \mathrm{d}x$, see \cite{abdellatif2025wasserstein} Corollary 3.3 and Remark 3.4. Second, it is equivalent to the classical optimal transportation formulation: $d(P,P_0) = \mathop{\textup{inf}} \limits_{Q \in \mathcal{Q}(P,P_0)} \mathbb{E}_{(\tilde{v},\tilde{u}) \sim Q}[|\tilde{v}-\tilde{u}|]$, where $\mathcal{Q}(P,P_0)$ represents the set of joint distributions with marginals $P$ and $P_0$ (as shown by \citealt{chen2024screening}).}:
\beq
d(P,P_0) : = \int_0^1 \vert \bar{P}(x)-\bar{P}_0(x)\vert \mathrm{d}x,
\eeq
where $\bar{P}(x)$ represents the complementary cumulative distribution function (CCDF). Intuitively, the Wasserstein ambiguity set consists of all probability distributions $P(\cdot)$ that are within a Wasserstein distance $r$ from the \emph{reference distribution} $P_0(\cdot)$, with mean $\mu_0$. For example, in data-driven contexts, the reference distribution can be estimated from historical data, such as an empirical distribution, when available. Conversely, when historical data is unavailable or limited, the reference distribution can be set to a uniform distribution to represent a lack of information.

\textbf{RO Framework.} The literature primarily explores a robust optimization framework, which entails solving the following optimization problem:
\begin{equation}\label{Prob_RO}
    \Pi_{RO}^*(r) = \sup\limits_{(q,m)\in \mathcal{M}} \inf\limits_{P \in \mathcal{W}(P_0,r)}  \mathbb{E}_P[m(\tilde{v})].\tag{OPTIMIZATION}
\end{equation}
This robustly optimal mechanism aims to maximize the worst-case expected revenue within the ambiguity set \citep{carrasco2018optimal, chen2024screening, wang2024minimax}. While the RO framework in \eqref{Prob_RO} effectively addresses distribution ambiguity, it has a notable drawback: it is highly sensitive to the radius parameter $r$, which defines the allowable distance between the true and reference distributions. A small radius offers inadequate protection against ambiguity in the valuation distribution, while a larger radius provides excessive protection, leading to a conservative mechanism. Furthermore, calibrating the radius parameter is often unclear, and the optimal mechanism derived from \eqref{Prob_RO} may perform poorly if the true distribution falls outside the ambiguity set $\mathcal{W}(P_0,r)$. 

\textbf{RS Framework.} To address these challenges, we propose the RS framework introduced by \cite{long2023robust} to develop a distributionally free robust mechanism as follows:
\be
\label{Prob_RS}
\begin{array}{ll}
    \inf\limits_{(q,m)\in \mathcal{M},k > 0} & k\\
    s.t.  & \tau - \mathbb{E}_P[m(\tilde{v})] \leq kd(P,P_0), \quad \forall P \in \mathcal{P},
\end{array} \tag{SATISFICING}
\ee
where the parameter $\tau > 0$ is a pre-determined constant representing the \emph{target} revenue that the decision maker is willing to accept, relative to the reference distribution (such as the empirical distribution). One important difference from the RO framework is the specification of target parameter $\tau$ in \eqref{Prob_RS}, rather than the size of the ambiguity set. The satisficing constraint in \eqref{Prob_RS} must hold when the distribution $P$ matches the reference distribution $P_0$. Therefore, a necessary condition for the feasibility of Problem \eqref{Prob_RS} is $\tau \leq \max_{(q,m)\in \mathcal{M}}\mathbb{E}_{P_0}[m(\tilde{v})]$, ensuring that the expected revenue over the reference distribution meets the target. 

The decision variable $k$ quantifies the \emph{fragility} of underperforming the revenue target by measuring the \emph{worst-case expected revenue loss from the target} across all possible distributions $P\in \mathcal{P}$, normalized by the Wasserstein distance of this distribution from the reference distribution $P_0$. Intuitively, a larger target violation is acceptable when a distribution is more distant from the reference distribution (such as the empirical distribution). When $P$ represents the true valuation distribution, the satisficing constraint implies that $\mathbb{E}_P[m(\tilde{v})]\geq \tau -k d(P,P_0)$, indicating better performance guarantees with smaller values of $k$.

Notably, the framework in \eqref{Prob_RS} offers a distinct modeling philosophy compared to \eqref{Prob_RO}. While RO aims to maximize the worst-case expected revenue, the RS approach seeks to minimize fragility, measured by the worst-case expected revenue loss from the target. Consequently, RO is more appealing to revenue-driven decision-makers, whereas RS is better suited for target-driven decision-makers. Throughout this paper, we use subscripts `$\cdot_{RO}$' and `$\cdot_{RS}$' to denote the RO and RS frameworks, respectively. We use superscripts `$\cdot^{*}$' and `$\cdot^{PP}$' to indicate optimal and PP mechanisms, respectively. For instance, `$\Pi_{RS}^*$' denotes the revenue under the optimal mechanism in a RS framework.


We conclude this section by presenting an equivalent representation of feasible direct mechanisms. According to \cite{myerson1981optimal}, a direct mechanism $(q,m)$ is incentive compatible if and only if (i) $q(v)$ is non-decreasing and (ii) $m(v)=m(0)+q(v)v-\int_0^vq(x)\mathrm{d}x$ for all $v \in [0,1]$. This implies that the payment function $m(\cdot)$ is uniquely defined by the allocation function $q(\cdot)$. By applying integration by parts, we can express the payment function $m$ as $m(v) = m(0)+\int_0^v x \mathrm{d}q(x)$. Individual rationality requires that $0\leq q(0)\cdot 0-m(0)=-m(0)$, leading to $m(0)\le 0$. Therefore, in the optimization problems \eqref{Prob_RO} and \eqref{Prob_RS}, it is optimal to set the net payment for the lowest-valuation buyer to zero, i.e., $m(0) = 0$. Consequently, the set of direct mechanisms can be equivalently expressed as:
\beq
\mathcal{M} = \left\{ (q,m): [0,1] \rightarrow \left[0,1\right] \times \mathbb{R} \bigg| 
\begin{array}{ll}
    q(v) \geq q(\omega),  &\forall v,\omega\in[0,1], v\geq \omega,  \\
    m(v) = \int_0^v x \mathrm{d}q(x), &\forall v \in [0,1], 
\end{array}
\right\}.
\eeq
This indicates that a feasible direct mechanism  $(q,m)$ can be implemented using a randomized pricing mechanism described in \eqref{Prob_RP}, with $q(\cdot)$ serving as the CDF for the randomized prices. 

\section{Robust Satisficing Mechanisms}\label{sec-rs-mechanism}
In this section, we analyze optimal mechanisms within the RS framework. We begin by fully characterizing the optimal RS mechanism by solving \eqref{Prob_RS} in Subsection \ref{sec-rs-opt}. Next, we derive the optimal deterministic posted price in Subsection \ref{sec-rs-pp}. Finally, we assess the effectiveness of the posted pricing mechanism by comparing it to the optimal mechanism in Subsection \ref{sec-rs-pp-effectiveness}.
\subsection{Optimal RS Mechanism}\label{sec-rs-opt}
In this section, we examine the robust revenue satisficing problem \eqref{Prob_RS} to derive the optimal RS mechanism. We begin by analyzing the satisficing constraint, which entails an infinite number of constraints. For any feasible $(q,m) \in \mathcal{M}$ and $k > 0$, the satisficing constraint in \eqref{Prob_RS} can be expressed as $\tau \leq   \mathbb{E}_P[m(\tilde{v})] + kd(P,P_0), \forall P \in \mathcal{P}$, which is equivalent to:
\beq
      \tau \leq \inf_{P \in \mathcal{P}}\left[\mathbb{E}_P\left[m(\tilde{v})\right] + kd(P,P_0)\right] = \inf_{P \in \mathcal{P}}\left[ \mathbb{E}_P\left[ \int_0^{\tilde{v}} x\mathrm{d}q(x)\right]+ kd(P,P_0)\right].
\eeq
Thus, the problem \eqref{Prob_RS} is equivalent to the following optimization:
\be
\begin{array}{ll}
    \inf\limits_{k > 0} & k\\
    s.t.  & \tau \leq \sup_{(q,m)\in \mathcal{M}}\inf_{P \in \mathcal{P}}\left[ \mathbb{E}_P\left[ \int_0^{\tilde{v}} x\mathrm{d}q(x)\right]+ kd(P,P_0)\right],
\end{array} \label{Eq_constraint}
\ee
We demonstrate in the following lemma that, for any given $k > 0$, the maximin problem in the constraint of \eqref{Eq_constraint} can be equivalently represented as the following minimax pricing problem, regularized by the Wasserstein distance.

\begin{prop}\label{Lm_pricing}
For any $k > 0$, we have 
\bea
 \sup_{(q,m)\in \mathcal{M}}\inf_{P \in \mathcal{P}}\left[ \mathbb{E}_P\left[ \int_0^{\tilde{v}} x\mathrm{d}q(x)\right]+ kd(P,P_0)\right]
 & = & \inf_{P \in \mathcal{P}}\left[\max_{x\in[0,1]}\mathcal{R}_{P}(x)+ k d(P,P_0)\right] \label{Prob_pricing} \\
 & = & 
\inf_{r \geq 0} \left\{\inf_{P \in \mathcal{W}(P_0,r)} \max_{x\in [0,1]} \mathcal{R}_{P}(x) + kr \right\}, \label{Eq_robustpricing}
\eea where $\mathcal{R}_{P}(x) :=x\bar{P}_{-}(x)$, with $\bar{P}_{-}(x):={\mathbb P}(\tilde{v}\ge x) = \lim_{v \uparrow x}\bar{P}(v)$, represents the expected selling revenue at a posted price $x$ when the valuation follows distribution $P$.
\end{prop}

Proposition \ref{Lm_pricing} significantly simplifies the optimization problem by reducing the decision variables from three functions $(q, m, P)$ to just one function $P$ and one scalar $x\in [0,1]$ in Problem \eqref{Prob_pricing}. This simplification arises from strong duality, leveraging the functional version of von Neumann's minimax theorem. The inner maximization problem in Problem \eqref{Prob_pricing}, specifically $\max_{x\in [0,1]}\mathcal{R}_{P}(x)$, represents a classical \emph{monopoly pricing} problem that seeks to find the optimal deterministic posted price based on a \emph{known} valuation distribution $P$. It reflects the seller's optimal price against the worst-case distribution selected by the adversary.

The minimax problem on the right-hand side of \eqref{Prob_pricing} is still challenging to solve due to the curse of dimensionality, stemming from the infinite number of valuation distributions generated by regularizing the Wasserstein distance between two distributions. In Problem \eqref{Eq_robustpricing}, we tackle this challenge by reformulating it to focus on a single decision variable: the size of the Wasserstein ambiguity set $r$. This simplification leads to an objective function that depends solely on $r$, comprising two terms: the classic \emph{minimax monopoly pricing} problem using posted pricing mechanisms, specifically $\inf_{P \in \mathcal{W}(P_0,r)} \max_{x\in [0,1]} \mathcal{R}_{P}(x)$, with an \emph{unknown} valuation distribution within a Wasserstein ball of size $r$, coupled with a linear fragility term $kr$. 
\cite{chen2024screening} studied the minimax monopoly pricing problem with a Wasserstein ambiguity set and provided a full characterization of its solution (in Theorem 4), which is restated below for the sake of self-completeness. 

\begin{lem} \label{Lm_chen}
{\sf (Minimax Monopoly Pricing with a Wasserstein Ambiguity Set)}
In a Wasserstein ambiguity set of size $r< \mu_0$, the worst-case distribution to the minimax monopoly pricing problem, $\inf_{P \in \mathcal{W}(P_0,r)} \max_{x\in [0,1]} \mathcal{R}_{P}(x)$, is given by:
\beq 
\bar{P}_{RO}^{PP}(v|r) = \min\left\{\bar{P}_0(v),\frac{\Pi_{RO}^{PP}(r)}{v}\right\}, \quad \forall~v \in [0,1],
\eeq
and the corresponding optimal revenue, denoted as $\Pi_{RO}^{PP}(r)$, is the unique solution in $[0,1]$ to:
\bea
d\left(\Pi_{RO}^{PP}(r)\right) = r, \text{ where } d\left(\pi\right) :=\int_0^1 \left(\bar{P}_0(x)-\pi/x\right)^+ \mathrm{d}x. \label{equ-rev-ro-pp}
\eea
\end{lem}

The worst-case distribution for a minimax monopoly pricing problem can be constructed using the reference distribution and an inverse function. Specifically, when the valuation distribution adopts the inverse function, the selling revenue remains constant regardless of the buyers' valuations, expressed as $\mathcal{R}_{P}(v) = \pi$ when $\bar{P}(x) = \pi/x$. Thus, the inverse function $\pi/x$ represents an \emph{iso-revenue} curve that generates the same revenue $\pi$ for any buyer valuation. This iso-revenue curve is then truncated using the reference distribution to ensure that it forms a valid distribution function, i.e., $\bar{P}_{\pi}(x) = \min\left\{\pi/x, \bar{P}_0(x)\right\}$. The Wasserstein distance between this truncated iso-revenue distribution and the reference distribution simplifies to $d\left(\pi\right)$. To minimize revenue, the adversary shifts this iso-revenue curve downward as far as possible within the Wasserstein ball of size $r$. Consequently, the worst-case distribution lies on the boundary of this Wasserstein ball, maintaining a Wasserstein distance $r$ from the reference distribution, i.e., $d(P_{\pi},P_0) (=d(\pi)) = r$, as illustrated in the shaded area of Figure \ref{Fig_illustration} for a uniform reference distribution $P_0$. 

\begin{figure}[!htb]
    \vspace{-0.1in}
    \FIGURE
    {
        \subfigure[~The worst-case distribution $\bar{P}_{RO}^{PP}$]{
            \includegraphics[width=0.45\textwidth]{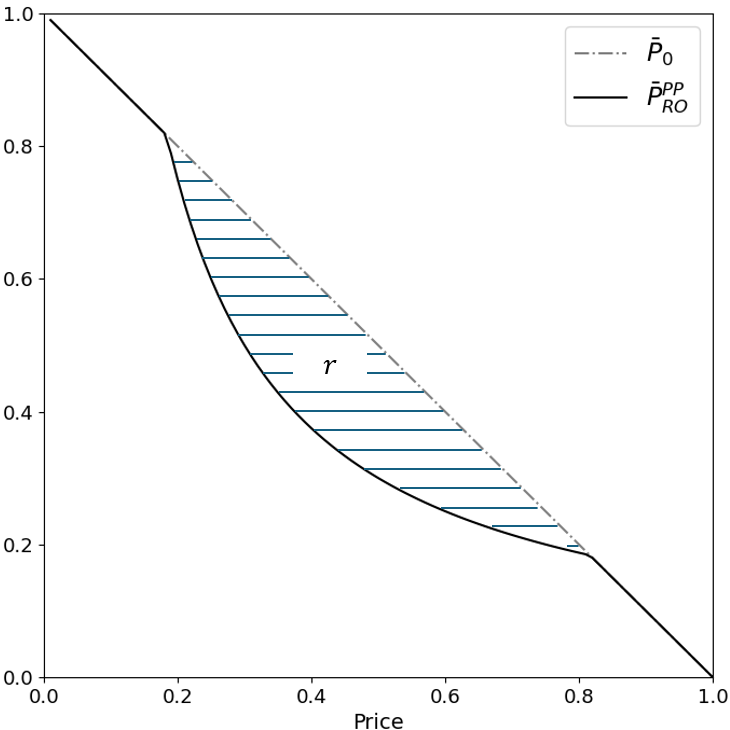}}   
        \subfigure[~Empirical Distribution $\bar{P}_0$ vs $0.25/x$]{
            \includegraphics[width=0.45\textwidth]{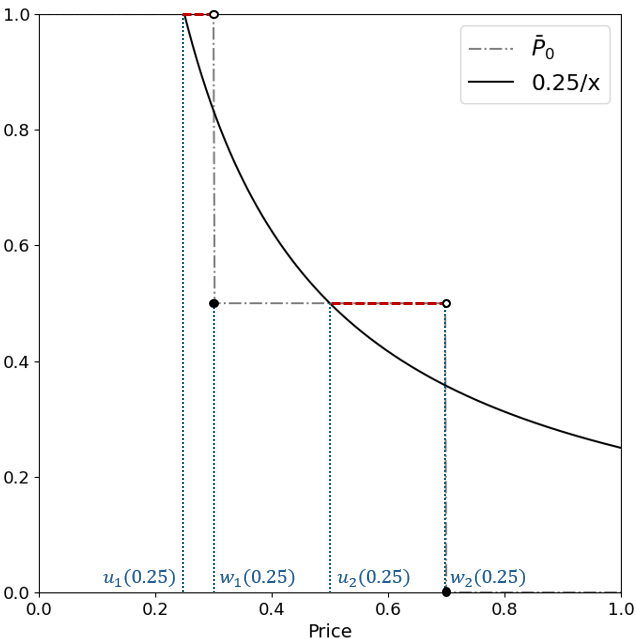}}
    }
    {Illustration of the Worst-case Distribution for the Minimax Monopoly Pricing Problem \label{Fig_illustration}}
    {The shaded area in the left panel represents the Wasserstein distance from the reference distribution. The red-highlighted intervals in the right panel represent disjoint intervals  $\bigcup_{j\in[J]} [u_j, w_j)$ where $\bar{P}_0(x)\ge 0.25/x$.}
    \vspace{-0.2in}
\end{figure}

Notably, the minimax monopoly pricing problem reduces to a classical posted pricing problem when the ambiguity set collapses to a point. Specifically, when the ambiguity size $r=0$, the worst-case distribution aligns with the reference distribution, leading the optimal revenue to match that of the optimal posted pricing mechanism:  $\bar{P}_{RO}^{PP}(v|0) = \bar{P}_0(v)$ and $\Pi_{RO}^{PP}(0)= \max_{x\in [0,1]} \mathcal{R}_{P_0}(x)$. A key observation from Lemma \ref{Lm_chen} is that the Wasserstein distance between the worst-case distribution and the reference distribution is precisely the size of the ambiguity set $r$. This means the worst-case distribution lies on the boundary of the Wasserstein ball  $\mathcal{W}(P_0,r)$. This insight allows us to reformulate the RS Problem \eqref{Eq_constraint} as an optimization problem in which the ambiguity size $r$ -- which can be transformed into the optimal worst-case revenue from the minimax monopoly pricing problem -- serves as the decision variable. This is formally presented in Proposition \ref{Prop_robustpricing}.

\begin{prop}\label{Prop_robustpricing}{\sf (Equivalent Representation)}
The RS Problem \eqref{Prob_RS} is equivalent to:
\be
\begin{array}{ll}
    \inf\limits_{k > 0} & k\\
    s.t.  & \tau \leq \min_{\pi \in [0,\Pi_0]} \left[ \pi + k\cdot d(\pi) \right],
\end{array} \label{Eq_optimizationPi}
\ee
where $\Pi_0:=\max_{x\in [0,1]} \mathcal{R}_{P_0}(x)$ denotes the maximum expected revenue generated under a PP mechanism with the reference distribution, and $d(\pi)$ is given in \eqref{equ-rev-ro-pp}.
\end{prop}


The key idea behind Proposition \ref{Prop_robustpricing} is that for any Wasserstein distance $r$, we can solve the robust monopolist pricing problem to identify the worst-case distribution and optimal revenue, as shown in Lemma \ref{Lm_chen}. We demonstrate that the optimal Wasserstein distance must be less than the mean of the reference distribution $\mu_0$, which simplifies the solution process by limiting the search for the Wasserstein distance to the interval $[0,\mu_0]$. Finally, we transform the decision variable in the minimization from Wasserstein distance to worst-case revenue, leveraging their monotonic relationship as described in \eqref{equ-rev-ro-pp} and illustrated in Figure \ref{Fig_fragility_revenue}. This greatly simplifies the original problem \eqref{Prob_RS} by reducing the decision variables from three functions $(q,m,P)$ to a single scalar variable $\pi$ within a finite interval, enabling us to analytically characterize the optimal solution structure.


Next, we derive the solution for the satisficing problem \eqref{Eq_optimizationPi}. To obtain a closed-form expression for the Wasserstein distance $d(\pi)$, we assume there are finitely many disjoint intervals where the reference distribution exceeds the truncated iso-revenue distribution. Specifically, for any $\pi\in[0,\Pi_0]$, there exist a non-negative integer $J(\pi)$ and intervals $\{\left(u_i\left(\pi\right),w_i\left(\pi\right)\right)\}_{i=1}^{J(\pi)}$ such that $0\le u_1(\pi)<w_1(\pi)<u_2(\pi)<w_2(\pi)<\cdots<u_J(\pi)<w_J(\pi)\le  1$. For convenience, we denote $w_0=u_0=0$ and $u_{J+1} =1$. Within these intervals, we have $\bar{P}_0(x)\geq \pi/x$ for all $x\in \bigcup_{j=1}^{J(\pi)} [u_j(\pi), w_j(\pi))$. 
This assumption holds, for instance, when the reference distribution is either an empirical or a uniform distribution, as illustrated in Figure \ref{Fig_illustration}, with $J=1$ for a uniform distribution in panel (a) and $J=2$ for an empirical distribution in panel (b). Later in Proposition \ref{Prop_ContinuousSol}, we show $J(\pi)=1$ for a general class of reference distribution for any $\pi\in[0,\Pi_0]$. 


We will first solve the minimization problem under the constraint of \eqref{Eq_optimizationPi}. Let:
\bea
\rho(\pi,k) :=\pi + k\cdot d(\pi), \  \rho^*(k) := \min_{\pi\in[0,\Pi_0]} \rho(\pi,k), \  \mbox{and} \  \pi^*(k) := \argmin_{\pi\in[0,\Pi_0]} \rho(\pi,k). \label{equ-fragility-adjusted-rev}
\eea Intuitively, $\rho(\pi,k)$ represents the \emph{fragility-adjusted revenue} and $\pi^*(k)$ denotes the worst-case revenue for the minimax monopoly pricing problem given a specific fragility $k$. The following proposition offers a comprehensive characterization of the minimization problem \eqref{equ-fragility-adjusted-rev}.

\begin{thm}\label{Prop_optimalPi}{\sf (Solution to RS Problem \eqref{Eq_optimizationPi})}
For any given fragility $k > 0$, $\rho(\pi,k)$ is a convex function of $\pi$ with a unique minimum $\pi^*(k)\in [0,\Pi_0]$, which is the unique solution to the equation $\sum_{j \in [J(\pi)] } \ln \frac{w_j(\pi)}{u_j(\pi)} = \frac{1}{k}$. The minimum value $\rho^*(k)$ is given by:
\bea
\rho^*(k) = k \cdot \sum_{j\in[J(\pi^*(k))]}\int_{u_j\left(\pi^*(k)\right)}^{w_j\left(\pi^*(k)\right)} \bar{P}_0(x)\mathrm{d}x. \label{equ-min-fragility-revenue}
\eea
Additionally, $\rho^*(k)$ is a monotonically increasing function of $k$ (see Figure \ref{Fig_fragility_revenue}). Therefore, the optimal fragility for \eqref{Eq_optimizationPi} is given by $k_{RS}^* := (\rho^*)^{-1}(\tau)$.
\end{thm}


The above theorem provides a complete characterization of the solution to the satisficing problem \eqref{Eq_optimizationPi}. We first establish that fragility-adjusted revenue is convex in worst-case revenue, leading to a unique minimum. We then show that the minimum fragility-adjusted revenue increases monotonically with fragility $k$, as illustrated in Figure \ref{Fig_fragility_revenue}. Consequently, the satisficing problem is trivially achieved at the boundary, where the target level is exactly met, i.e.,  $\rho^*(k_{RS}^*) = \tau$. As discussed above, since $\Pi_0 = \Pi_{RO}^{PP}(0)$, the Wasserstein distance from the truncated iso-revenue distribution at $\Pi_0$ to the reference distribution is zero, i.e., $d(\Pi_0) = 0$ (see \eqref{equ-rev-ro-pp}). This implies $\rho^*(k) \le \rho(\Pi_0,k) = \Pi_0$ (as per \eqref{equ-min-fragility-revenue}). Thus, a necessary condition for the feasibility of $\rho^*(k_{RS}^*) = \tau$ is $\tau \le \Pi_0$.

We are now ready to translate the optimal solution to the RS problem \eqref{Eq_optimizationPi} back into the optimal mechanism for the original RS problem \eqref{Prob_RS}. The optimal RS mechanism for \eqref{Prob_RS} is formally presented in the following Theorem \ref{Thm_optimalmechanism}.

\begin{thm}\label{Thm_optimalmechanism} {\sf (Optimal RS Mechanism)}
The optimal RS mechanism $(q_{RS}^*,m_{RS}^*)$ to Problem \eqref{Prob_RS} is given by: For each $j=0,1,\ldots,J^*$,
\bea
\left(q_{RS}^*\left(v\right),m_{RS}^*\left(v\right)\right) = \left\{\begin{array}{ll}
    k_{RS}^* \cdot\left(\sum_{i\in[j-1]} \ln \left(\frac{w_i^*}{u_i^*}\right) + \ln (\frac{v}{u_j^*}), \; \sum_{i\in[j-1]}(w_i^*-u_i^*) +(v-u_j^*) \right), & v \in [u_j^*,w_j^*)\\
    k_{RS}^* \cdot \left(\sum_{i\in[j]} \ln \left(\frac{w_i^*}{u_i^*}\right),\; \sum_{i\in[j]}(w_i^*-u_i^*) \right), & v \in[w_j^*,u_{j+1}^*),   
\end{array}\right.\label{equ-opt-rs-mechanism}
\eea
where $k_{RS}^* = (\rho^*)^{-1}(\tau)$, $\pi^*=\pi^*(k_{RS}^*)$, $J^* = J(\pi^*)$, $u_j^* = u_j(\pi^*)$, $w_j^* = w_j(\pi^*)$ for all $j\in[J^*]$, and $\left(q_{RS}^*\left(1\right),m_{RS}^*\left(1\right)\right) = \left(1,\; k_{RS}^*\cdot\sum_{i\in[J^*]}(w_i^*-u_i^*) \right)$.
\end{thm}

Notably, the optimal RS allocation described in the theorem involves offering a continuum of lotteries, each associated with a logarithmic winning probability and a linear price. Specifically, the optimal allocation rule $q_{RS}^*(\cdot)$ for the robust satisficing problem \eqref{Prob_RS} is continuous and non-decreasing, with $q_{RS}^*(0)=0$ and $q_{RS}^*(1) = 1$. The optimal RS mechanism $(q_{RS}^*, m_{RS}^*)$ can be reinterpreted as a randomized pricing mechanism, as described in \eqref{Prob_RP}, where the randomized price $\tilde{p}$ is drawn from the probability distribution $q_{RS}^*\in\mathcal{P}$ with the following density function:
\beq
\frac{\mathrm{d} q_{RS}^*(v)}{\mathrm{d} v} = \frac{\mathrm{d} q_{RS}^*(v)}{\mathrm{d} v} 
   =  \frac{k_{RS}^*}{v} \cdot \mathbbm{1}\left(v\in \bigcup_{j\in[J^*]} [u_j^*,w_j^*)\right),
\eeq where $\mathbbm{1}\left(A\right) = 1$ if event $A$ is true, and $\mathbbm{1}\left(A\right) = 0$ otherwise.
In this context, the allocation probability $q_{RS}^*(\cdot)$ serves as the cumulative distribution function for the randomized price $\tilde{p}$. Consequently, $\mathrm{d} m_{RS}^*(v) / \mathrm{d} v = k_{RS}^* \cdot \mathbbm{1}\left(v\in \bigcup_{j\in[J^*]} [u_j^*,w_j^*)\right)$. Thus, the optimal fragility $k_{RS}^*$ not only reflects the extent to which expected revenue falls short of the target, as discussed after \eqref{Prob_RS}, but also indicates payment sensitivity to buyer valuations.

We conclude this section by providing two examples to illustrate how the worst-case revenue $\pi^*(k)$ and the corresponding fragility-adjusted revenue $\rho^*(k)$, as derived in Theorem \ref{Prop_optimalPi}, along with the optimal RS mechanism $\left(q_{RS}^*\left(v\right),m_{RS}^*\left(v\right)\right)$ described in Theorem \ref{Thm_optimalmechanism}, apply when the reference distribution is a uniform or empirical distribution.

\begin{exap}\label{exap-uniform-RS-opt}
{\sf (Uniform Reference Distribution)} Given a uniform reference distribution $\bar{P}_0(x)=1-x$ for all $x\in [0,1]$, we have $\mu_0 = 1/2$, $\Pi_0 = \max_{p\in[0,1]}p (1-p) = 0.25$, $J=1$, $u_1(\pi) = \frac{1-\sqrt{1-4\pi}}{2}$, and $w_1(\pi)=\frac{1+\sqrt{1-4\pi}}{2}$. Solving the equation $\sum_{j \in [J(\pi)] } \ln \frac{w_j(\pi)}{u_j(\pi)} = \frac{1}{k}$ yields $\pi^*(k) =\frac{1}{4}\left[1-\left(\frac{e^{1/k}-1}{e^{1/k}+1}\right)^2\right]$. Substituting this expression into \eqref{equ-min-fragility-revenue}, we obtain:
\beq
\rho^*(k) = k\int_{u_1\left(\pi^*(k)\right)}^{w_1\left(\pi^*(k)\right)} \bar{P}_0(x)\mathrm{d}x = k\frac{\left(e^{1/k}-1\right)}{2\left(e^{1/k}+1\right)},
\eeq
which can be verified to be a monotonically increasing function of $k$. Solving $\rho^*(k) = \tau$ for each $\tau\in\{0.1,0.2\}$, we derive the optimal fragility and the corresponding optimal RS mechanism, as detailed in Table \ref{table-RS-PP-opt} and illustrated in Figure \ref{Fig_opt}(a).
\end{exap}

\begin{figure}[!htb]
	\vspace{-0.1in}
	\FIGURE
	{
            \subfigure[~Uniform reference distribution]
            {
    		\begin{minipage}{0.45\textwidth}
    				\includegraphics[width=0.96\textwidth]{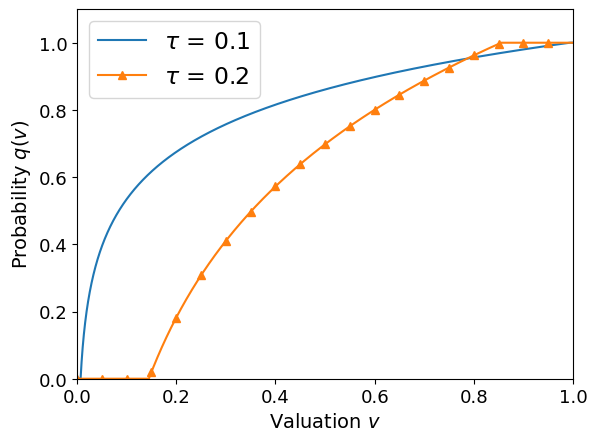}
                \vfill
    				\includegraphics[width=0.96\textwidth]{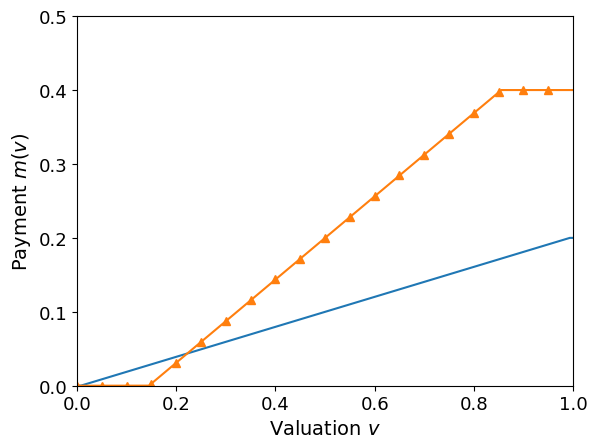}
    		\end{minipage}
            }
            \subfigure[~Empirical reference distribution]
            {
    		\begin{minipage}{0.45\textwidth}
    				\includegraphics[width=0.96\textwidth]{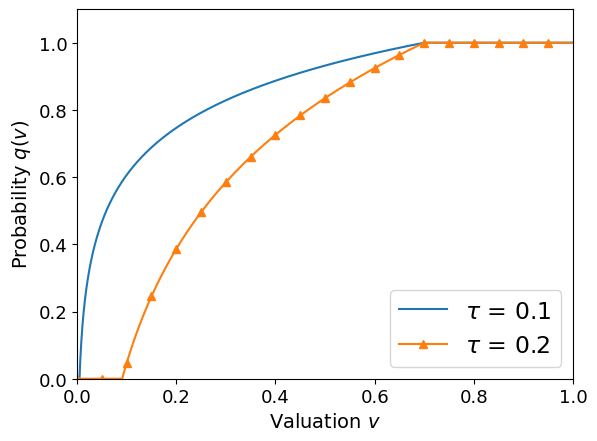}
                \vfill
    				\includegraphics[width=0.96\textwidth]{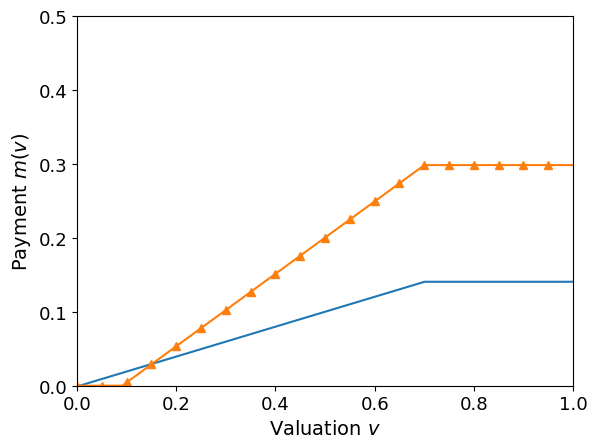}
    		\end{minipage}
            }
	}
	{The Optimal RS Mechanism under Uniform and Empirical Reference Distributions\label{Fig_opt}}
	{}
    \vspace{-0.2in}
\end{figure}


\begin{table}
\vspace{-0.2in}
\TABLE
{Mechanisms under RS Framework: Uniform and Empirical Reference Distributions\label{table-RS-PP-opt}}
{\begin{tabular}{c cccc  ccc}
\toprule
 & \multicolumn{4}{c}{Optimal RS mechanism} & \multicolumn{3}{c}{Optimal PP mechanism} \\
\cmidrule(lr){2-5} \cmidrule(lr){6-8} 
\makecell{target \\ $\tau$} & \makecell{fragility\\ $k_{RS}^*$} & \makecell{revenue\\ $\Pi_{RS}^*$} & \makecell{intervals \\ $\bigcup_{j\in[J^*]} [u_j^*,w_j^*)$ } & \makecell{mechanism \\ $(q_{RS}^*,m_{RS}^*)$} & \makecell{fragility\\ $k_{RS}^{PP}$} & \makecell{revenue\\ $\Pi_{RS}^{PP}$} & \makecell{price \\ $p_{RS}^{PP}$ } \\
\hline
& \multicolumn{7}{l}{Uniform reference distribution} \\ 
0.1 & 0.20 & 0.007 & [0.007,\,0.99] & \multirow{2}{*}{Figure \ref{Fig_opt}(a)}  & 0.33 & 0.04  & 0.2  \\
0.2 & 0.56 & 0.12 & [0.14,\quad 0.86] & & 2 & 0.16  & 0.4 \\ \hline 
& \multicolumn{7}{l}{Empirical reference distribution} \\ 
0.1 & 0.20 & 0.005  & [0.005,\,0.70] & \multirow{2}{*}{Figure \ref{Fig_opt}(b)}  & 0.29 & 0.08 & 0.16\\
0.2 & 0.49 & 0.09  & [0.09,\quad 0.70]& & 1.33 & 0.2 & 0.4 \\
\bottomrule
\end{tabular}}
{}
\vspace{-0.2in}
\end{table}

\begin{exap}\label{exap-empirical-RS-opt}
{\sf (Empirical Reference Distribution)} Given an empirical reference distribution with support $\{0.3,0.7\}$ and equal probability $0.5$ for each value, we have $\mu_0 = \frac{1}{2}$, $\Pi_0 = \max_{p\in [0,1]} p\bar{P}_{0-}(p)=0.35$, and
\beq
\left(J\left(\pi \right), \; \left\{\left(u_i\left(\pi\right),w_i\left(\pi\right)\right)\right\}_{i\in[J(\pi)]} \right)= \left\{\begin{array}{ll}
    \left(1,\;\left\{\left(\pi, 0.7\right)\right\}\right), & \pi \in [0,0.15) \\
    \left(2, \;\left\{\left(\pi, 0.3 \right), \left(2\pi, 0.7\right)\right\}\right), & \pi \in[0.15,0.3)\\
    \left(1,\;\left\{\left(2\pi, 0.7\right)\right\} \right), & \pi \in [0.3,0.35).
\end{array}\right.
\eeq
Solving $\sum_{j \in [J(\pi)] } \ln \frac{w_j(\pi)}{u_j(\pi)} = \frac{1}{k}$ yields:
\beq
\left(\pi^*\left(k\right),\;\rho^*\left(k\right)\right)= \left\{\begin{array}{ll}
    \left(0.7e^{-1/k}, \;k(0.65-1.4e^{-1/k})\right), & k \in (0,\frac{1}{\ln(14/3)}) \\
    \left(\sqrt{0.105e^{-1/k}}, \;k(0.65-2\sqrt{0.105e^{-1/k}})\right), & k \in[\frac{1}{\ln(14/3)},\frac{1}{\ln(7/6)})\\
    \left(0.35e^{-1/k}, \;0.35k(1-e^{-1/k})\right), & k\in [\frac{1}{\ln(7/6)},\infty).
\end{array}\right.
\eeq
It can be verified that $\rho^*\left(k\right)$ is a monotonically increasing function of $k$, with $\lim_{k\rightarrow \infty}\rho^*(k) =0.35$. Solving $\rho^*(k) = \tau$ for each $\tau\in\{0.1,0.2\}$, we derive the optimal fragility and the corresponding RS mechanisms, as detailed in Table \ref{table-RS-PP-opt} and illustrated in Figure \ref{Fig_opt}(b).
\end{exap}


\subsection{Posted Pricing Mechanism}\label{sec-rs-pp}
Although Theorem \ref{Thm_optimalmechanism} explicitly characterizes the optimal mechanism for minimizing the fragility of underperforming a specific revenue target, implementing such a complex mechanism involving a continuum lottery menu is practically challenging \citep{wang2025power}. As a result, simpler mechanisms, like PP mechanisms, are often more appealing to practitioners despite their suboptimal performance \citep{jin2020tight}. In this section, we focus on the widely used deterministic posted pricing mechanism, as described in \eqref{Prob_PP}, aiming to characterize its optimal pricing and assess its effectiveness compared to the optimal mechanism within the RS framework. 


\subsubsection{The Optimal PP Mechanism}\label{sec-opt-PP-rs} 
In this section, we characterize the optimal PP mechanism within the RS framework. By substituting the PP mechanisms, as described in \eqref{Prob_PP}, into problem \eqref{Prob_RS}, we can derive the optimal RS PP mechanism by solving the following: 
\be
\label{Prob_PP_RS}
\begin{array}{ll}
    \inf\limits_{p\in[0,1],k > 0} & k\\
    s.t.  & \tau - \mathbb{E}_P[p\mathbbm{1}(\tilde{v}\geq p)] \leq kd(P,P_0), \quad \forall P \in \mathcal{P},
\end{array} \tag{PP-SATISFICING}
\ee

We begin with a continuous reference distribution, such as a uniform distribution, and present the results in the following proposition.


\begin{prop}\label{Prop_ContinuousPP} {\sf (Equivalent RS Formulation under a Continuous Reference Distribution)}
For a continuous reference distribution, Problem \eqref{Prob_PP_RS} is equivalent to
\be\label{Prob_ContinuousOpt} 
\begin{array}{ll}
\inf_{k > 0} \quad & k\\
s.t. \quad & \tau \leq \max_{p\in[0,1]}\left\{\int_p^{(1+\frac{1}{k})p} k(v-p) \mathrm{d}P_0(v) + \int_{(1+\frac{1}{k})p} ^1 p \mathrm{d}P_0(v)\right\}.
\end{array}
\ee
\end{prop}

The above proposition, similar to Proposition \ref{Prop_robustpricing}, simplifies the satisficing problem by reducing decision variables from a continuous distribution within a Wasserstein ball in \eqref{Prob_PP_RS} to a single scalar posted price in \eqref{Prob_ContinuousOpt}. We establish this by applying a reformulation of the Wasserstein distance. Specifically, the robust satisficing constraint in \eqref{Prob_PP_RS} is equivalent to $\tau \le \max_{p \in [0,1]}\min_{ P \in \mathcal{P}}\left\{ \mathbb{E}_P[p\mathbbm{1}(\tilde{v}\geq p)] + kd(P,P_0) \right\}$. Applying a similar technique as Theorem 4.2 of \cite{mohajerin2018data}, we show that the worst-case distribution can be found within Dirac distributions on $[0,1]$, leading to the following equivalent reformulation of the satisficing constraint:
\beq
\tau \le \max_{p \in [0,1]}\int_{0}^1 \min_{ x\in[0,1] } \left[p\mathbbm{1}(x\geq p) + k |x-v|\right] dP_0(v),
\eeq 
where the Wasserstein distance is reformulated via the classical optimal transportation formulation, as shown by \cite{chen2024screening}, $d(D_x,P_0) = \int_{0}^1|x-v| dP_0(v)$, with $D_x$ denoting the Dirac distribution concentrated at $x$. This is equivalent to the satisficing constraint in \eqref{Prob_ContinuousOpt} because the inner minimization is given by:
\beq
\min_{x \in [0,1]} [p\mathbbm{1}(x \ge p) + k|x - v|] =
\left\{\begin{array}{ll}
0, & \mbox{if $v<p$,} \\
k(v-p), & \mbox{if $v\in[p,\left(1+\frac{1}{k}\right)p)$,} \\
p, &  \mbox{if $v\ge \left(1+\frac{1}{k}\right)p$.}
\end{array}
\right.
\eeq

To derive the optimal posted price for \eqref{Prob_ContinuousOpt}, akin to \eqref{equ-fragility-adjusted-rev}, let
\bea
\begin{array}{l}
\rho^{PP}(p,k) :=\int_{(1+\frac{1}{k})p} ^1 p \mathrm{d}P_0(v) + k\cdot\int_p^{(1+\frac{1}{k})p} (v-p) \mathrm{d}P_0(v) = \mathbb{E}_{P_0}\left[\min\left\{p,k(\tilde{v}-p)^+\right\}\right], \\
p^{PP*}(k) := \argmax_{p\in[0,1]} \rho^{PP}(p,k), \quad \rho^{PP*}(k) := \max_{p\in[0,1]} \rho^{PP}(p,k).
\end{array}\label{equ-fragility-adjusted-rev-pp}
\eea

Similar to $\rho(\pi,k)$ in Theorem \ref{Prop_optimalPi}, $\rho^{PP}(p,k)$ denotes the fragility-adjusted revenue at a posted price $p$, which is concave in $p$ for any given $k$ due to Jensen's inequality. Thus, there exists a unique posted price, $p^{PP*}(k)$, that maximizes this revenue. This maximum, $\max_{p} \rho^{PP}(p,k)$, increases with fragility $k$. 
Consequently, the optimal fragility for \eqref{Prob_ContinuousOpt}, denoted as $k_{RS}^{PP}$, is achieved at the boundary where the target is exactly met. The optimal posted price can then be determined by maximizing the fragility-adjusted revenue at this optimal fragility level. That is,
\bea
p_{RS}^{PP} = p^{PP*}\left(k_{RS}^{PP}\right), \mbox{ where } k_{RS}^{PP} = \left(\rho^{PP*}\right)^{-1}(\tau). \label{equ-opt-RS-price-fragility}
\eea

The following proposition characterizes the optimal posted price for a given fragility under the \emph{regularity} condition of the reference distribution \citep{myerson1981optimal}. Formally, a distribution $F$ is considered regular if its \emph{virtual valuation}, defined as $ x - \frac{1-F(x)}{f(x)}$, is strictly increasing in $x$, where $f$ and $F$ denote the density and cumulative distribution function, respectively. Regularity is ensured when the \emph{hazard rate}, defined as $\frac{f(x)}{1-F(x)}$, is increasing. As noted by \cite{Bagnoli2005} and \cite{Ewerhart2013}, regularity holds for uniform, normal, exponential, logistic, extreme-value, Laplace, Maxwell, and Rayleigh distributions; with certain parameter restrictions, it also applies to power, Weibull, Gamma, Chi-squared, Chi, and Beta distributions.

\begin{prop}\label{Prop_ContinuousSol}{\sf (Optimal Posted Price under RS Framework)}
    If the reference distribution $P_0$ is regular, then for each $c\in[0,\Pi_0]$, there are at most two posted prices, denoted by $u(c)$ and $w(c)$ (with $u(c)\le w(c)$), that yield the same revenue $c$ (i.e., $\mathcal{R}_{P_0}(p)=c$). Furthermore, $\frac{w\left(c\right)}{u\left(c\right)}$ decreases with $c$, which ensures a unique $c^*(k)$ satisfies $\frac{w\left(c^*\left(k\right)\right)}{u\left(c^*\left(k\right)\right)} = \frac{k+1}{k}$, and for any $k>0$, $p^{PP*}(k) = u\left(c^*\left(k\right)\right)$.
\end{prop}

We begin by illustrating the solution procedure using a uniform reference distribution. A necessary condition for the feasibility of the constraint is $\tau \leq \max_{p\in[0,1]}\mathcal{R}_{P_0}(p) (=\Pi_0) = 1/4$. Below, we demonstrate that this condition is also sufficient. 
The iso-revenue $\mathcal{R}_{P_0}(p)=c$ yields two solutions: $u(c) = \frac{1-\sqrt{1-4c}}{2}$ and $w(c) = \frac{1+\sqrt{1-4c}}{2}$. Setting $\frac{w(c)}{u(c)}=\frac{k+1}{k}$ yields $c^*(k)=\frac{k^2+k}{(2k+1)^2}$, thus: 
$$ p^{PP*}(k) = u\left(c^*\left(k\right)\right) = \frac{k}{2k+1}, \mbox{ and } \rho^{PP*}(k) = \frac{k}{4k+2} \in\left[0,\frac{1}{4}\right],$$ 
which clearly increases with $k$ for any $k > 0$. Therefore, the equation $\rho^{PP*}(k)=\tau$ has a solution if and only if $\tau\le 1/4$, given by (as illustrated in Table \ref{table-RS-PP-opt}):
\be \label{equ-opt-RS-price-fragility-uniform}
k_{RS}^{PP} = \frac{2\tau}{1-4\tau}, \text{ and } p_{RS}^{PP} = 2\tau. \tag{PP-RS-UNIFORM}
\ee


Next, we consider a discrete reference distribution, such as an empirical distribution derived from historical data. Specifically, the reference distribution is a discrete distribution $P_0=\{(\hat{v}_n,\alpha_n)\}_{n\in[N]}$, where the support consists of discrete values $\{\hat{v}_n\}_{n\in [N]}$ with $\hat{v}_1<\hat{v}_2 < \ldots <\hat{v}_N$, and the estimated probability density is $\alpha_i =\mathbb{P}(\tilde{v} = \hat{v}_i)\ge 0$ for each $i\in[N]$ satisfying $\sum_{i\in[N]}\alpha_i=1$. 
In this case, the CDF can be expressed as $P_0(x) = \sum_{i\in[N]}\alpha_i\cdot\mathbbm{1}\left(\hat{v}_i\le x\right)$, representing the probability that a buyer's valuation is no larger than $x$. Additionally, $\bar{P}_{0-}(x) = \sum_{i\in[N]}\alpha_i\cdot\mathbbm{1}\left(\hat{v}_i\ge x\right)$ indicates the probability that a buyer's valuation is no smaller than $x$, which represents the purchasing probability at the posted price $x$. Analogous to Proposition \ref{Prop_ContinuousPP}, the following proposition simplifies the RS problem \eqref{Prob_PP_RS} under a discrete reference distribution into a minimization involving a single scalar variable of the posted price.



\begin{prop}\label{Prop_DiscretePP} {\sf (Equivalent RS Formulation under an Empirical Reference Distribution)}
For an empirical reference distribution $\{(\hat{v}_n,\alpha_n)\}_{n\in[N]}$, Problem \eqref{Prob_PP_RS} is equivalent to
\be\label{Prob_RS2}
\begin{array}{ll}
\min_{k > 0} \quad & k\\
s.t. \quad & \tau \leq \max_{p\in[0,1]} \left\{ \sum_{n\in[N]} \alpha_n \inf_{v \in [0,1]} \left[ p\mathbbm{1}(v\geq p)+k\vert \hat{v}_n-v\vert \right] \right\}.
\end{array}
\ee
\end{prop}


The fragility-adjusted revenue with a discrete reference distribution can be defined similarly to \eqref{equ-fragility-adjusted-rev-pp} as follows:
\beq
\begin{array}{l}
\rho^{PP}(p,k) := \sum_{n\in[N]} \alpha_n \rho_n^{PP}(p,k), \  \rho_n^{PP}(p,k) := \inf_{v \in [0,1]} \left[ p\mathbbm{1}(v\geq p)+k\vert \hat{v}_n-v\vert \right] = \min\left\{k(\hat{v}_n-p)^+,p\right\}, \\
p^{PP*}(k) := \argmax_{p\in[0,1]} \rho^{PP}(p,k), \  \rho^{PP*}(k) := \max_{p\in[0,1]} \rho^{PP}(p,k).
\end{array}\label{equ-fragility-adjusted-rev-pp-disc}
\eeq 
It is clear that $\rho_n^{PP}(p,k)$ is concave in $p$ for any given $k$ and increases with $k$. Therefore, the fragility-adjusted revenue $\rho^{PP}(p,k)$ is also concave in $p$ and increases in $k$. This implies the existence of a unique posted price, $p^{PP*}(k)$, that maximizes this revenue, with the maximum revenue $\rho^{PP*}(k)$ increasing with fragility $k$ (see Figure \ref{fig-fragility-adj-rev-pp-empirical}). Consequently, the optimal fragility for \eqref{Prob_RS2} is achieved at the boundary where the target is precisely met, allowing the optimal posted price to be determined by maximizing the fragility-adjusted revenue at this optimal fragility level, as shown in \eqref{equ-opt-RS-price-fragility}.


Notably, since $\rho_n^{PP}(p,k) \le p\mathbbm{1}(\hat{v}_n\geq p)$, a necessary condition for the feasibility of the constraint is $\tau \leq \max_{p\in[0,1]} \sum_{n\in[N]} \alpha_n p\mathbbm{1}(\hat{v}_n\geq p) = \max_{p\in[0,1]} \mathcal{R}_{P_0}(p) (=\Pi_0)$, which is expressed as:
\beq
\mathcal{R}_{P_0}(p) = \sum_{n\in[N]} \alpha_n p\mathbbm{1}(\hat{v}_n\geq p) =  \left\{\begin{array}{ll}
    p, \quad &\text{if } p\leq \hat{v}_1\\
    (1-\sum_{j\in[i]}\alpha_j) p, \quad &\text{if } \hat{v}_i < p \leq \hat{v}_{i+1}, i \in [N-1],  \\
    0, \quad &\text{if } p> \hat{v}_N.
\end{array}\right. \\
\Pi_0 = \max_{p\in[0,1]} \mathcal{R}_{P_0}(p)  = \max\left\{\hat{v}_1,(1-\alpha_1)\hat{v}_2,\ldots, (1-\sum_{j\in[N-1]}\alpha_j)\hat{v}_N\right\} = \max_{i\in[N]}\left\{ \left(1-\sum_{j\in[i-1]}\alpha_j\right)\hat{v}_{i}\right\}.
\eeq While equation \eqref{equ-opt-RS-price-fragility} can be solved numerically for a general empirical reference distribution, obtaining an analytical solution is more challenging due to the complexities in explicitly characterizing $\rho^{PP*}(k)$. In the following proposition, we analytically derive the optimal posted price and its corresponding fragility using a two-point empirical distribution.


\begin{figure}[!ht]
    \vspace{-0.1in}
    \centering
    \caption{The Optimal Posted Price and Fragility under an Empirical Reference Distribution $\{(0.3,0.5),(\hat{v}_2,0.5)\}$ 
    }
    \label{Fig_Twosample}
    \includegraphics[width=0.45\columnwidth]{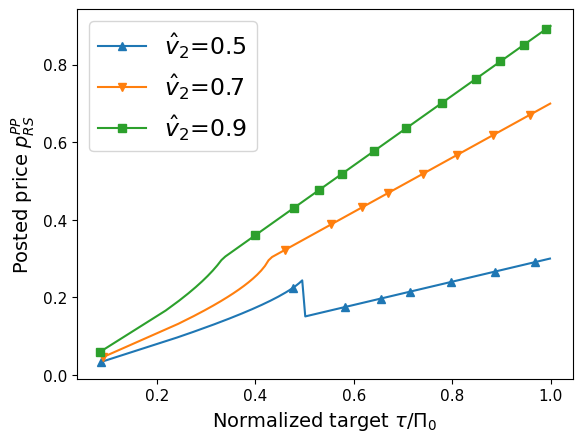}
    \includegraphics[width=0.45\columnwidth]{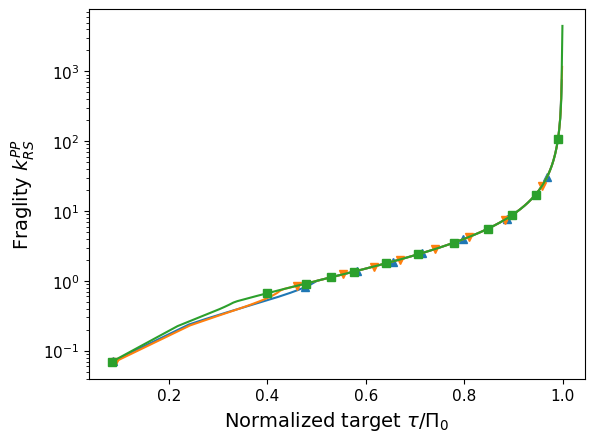}
\vspace{-0.2in}
\end{figure}


\begin{prop}\label{Prop_DiscreteSol} {\sf (Optimal RS-PP Mechanism under an Empirical Reference Distribution)} For a two-point empirical reference distribution $\{(\hat{v}_i,\alpha_i)\}_{i\in[2]}$, $\mu_0=\alpha_1\hat{v}_1 + \alpha_2\hat{v}_2$, $\Pi_0 = \max(\hat{v}_1, (1-\alpha_1)\hat{v}_2)$ and satisficing Problem \eqref{Prob_RS2} is feasible if and only if $\tau \leq \Pi_0 $, with the optimal solution given by:
\begin{enumerate}
    \item[1)] If $\hat{v}_1 \leq (1-\alpha_1)\hat{v}_2$, then $\Pi_0 =(1-\alpha_1)\hat{v}_2$, and
    \beq
    p_{RS}^{PP} = \frac{k_{RS}^{PP}}{k_{RS}^{PP}+1}\hat{v}_2, \text{ and } k_{RS}^{PP} = \left\{\begin{array}{ll}
        \frac{\mu_0 - \tau - \sqrt{(\mu_0  - \tau)^2-4\alpha_1\tau(\hat{v}_2-\hat{v}_1)}}{2\alpha_1(\hat{v}_2-\hat{v}_1)}, \quad &\text{if } \tau \leq  (1-\alpha_1)\hat{v}_1\\
        \frac{\tau}{(1-\alpha_1)\hat{v}_2-\tau}, \quad &\text{if } (1-\alpha_1)\hat{v}_1 < \tau \leq (1-\alpha_1)\hat{v}_2; 
    \end{array}\right.
    \eeq
    \item[2)] Otherwise, if $\hat{v}_1 > (1-\alpha_1)\hat{v}_2$, then $\Pi_0 =\hat{v}_1$, and
    \beq
    \left(p_{RS}^{PP}, \; k_{RS}^{PP}\right) = \left\{\begin{array}{ll}
        \left(\frac{k_{RS}^{PP}}{k_{RS}^{PP}+1}\hat{v}_2, \; \frac{\mu_0 - \tau - \sqrt{(\mu_0  - \tau)^2-4\alpha_1\tau(\hat{v}_2-\hat{v}_1)}}{2\alpha_1(\hat{v}_2-\hat{v}_1)}\right), \quad &\text{if } \tau \leq (1-\alpha_1)\hat{v}_1\\
        \left(\frac{k_{RS}^{PP}}{k_{RS}^{PP}+1}\hat{v}_1, \; \frac{\tau}{\hat{v}_1-\tau})\right), \quad &\text{if } (1-\alpha_1)\hat{v}_1 < \tau \leq \hat{v}_1. 
    \end{array}\right.
    \eeq
\end{enumerate}
\end{prop}



Next, we illustrate the above proposition using the empirical reference distribution $\{(0.3,0.5),(0.7,0.5)\}$ examined in Example \ref{exap-empirical-RS-opt}. In this context, based on \eqref{equ-fragility-adjusted-rev-pp}, we have:
\beq
p^{PP*}(k)=0.7\frac{k}{1+k} \mbox{ and } \rho^{PP*}\left(k\right)= \left\{\begin{array}{ll}
    0.15k+0.35\frac{k(1-k)}{1+k}, & k \in (0, 0.75] \\
    0.35\frac{k}{1+k}, & k \in(0.75,\infty)
\end{array}\right.,
\eeq as illustrated in Figure \ref{fig-fragility-adj-rev-pp-empirical}. Solving $\rho^{PP*}(k) = \tau$ for each $\tau\in\{0.1,0.2\}$, we derive the optimal posted price and its corresponding fragility, as shown in Table \ref{table-RS-PP-opt}. Additionally, we illustrate the roles of the revenue target $\tau$ on the optimal posted price and fragility in Figure \ref{Fig_Twosample}, which features a fixed $\hat{v}_1 = 0.3$ and varying $\hat{v}_2$. Consistent with our theoretical prediction, these numerical results reveal that feasible targets depend on the empirical distribution through the maximum revenue achievable from a deterministic posted pricing mechanism, i.e., $\Pi_0 = \max(\hat{v}_1, 0.5\hat{v}_2)$. Second, the optimal fragility of a PP mechanism $k_{RS}^{PP}$ increases monotonically with the revenue target $\tau$, which aligns with the intuition that an ambitious revenue target leads to a higher fragility. Third, the optimal posted price $p_{RS}^{PP}$ may fail to be monotonic in the revenue target $\tau$ when the differentiation between valuations is relatively small. This is because the seller switches to a lower price to achieve the revenue target when this target is becoming more aggressive. For example, when $\hat{v}_1>0.5\hat{v}_2$ or $\hat{v}_2<0.6$, the seller switches to a lower price as the revenue target increases and crosses the threshold $0.15$ (i.e., when $\tau>(1-\alpha_1)\hat{v}_1)$). Interestingly, the optimal deterministic posted price is linear in the revenue target when this target is large, as illustrated in Figure \ref{Fig_Twosample}. This numerical observation aligns with Proposition \ref{Prop_DiscreteSol}, which implies that, when the revenue target is large ($\tau> (1-\alpha_1)\hat{v}_1$), the optimal price equals to $\frac{1}{1-\alpha_1}\tau$ if valuation differentiation is larger (i.e., $\hat{v}_1\le (1-\alpha_1)\hat{v}_2$) or $\tau$ if valuation differentiation is small. This suggests that, to minimize worst-case fragility, an ambitious seller can set the price as a simple linear function of the chosen revenue target.


\begin{figure}[!htb]
	\vspace{-0.1in}
	\FIGURE
	{
            \subfigure[~Low Valuation ($v=0.25$)]
            {
    		\begin{minipage}{0.33\textwidth}
    				\includegraphics[width=0.96\textwidth]{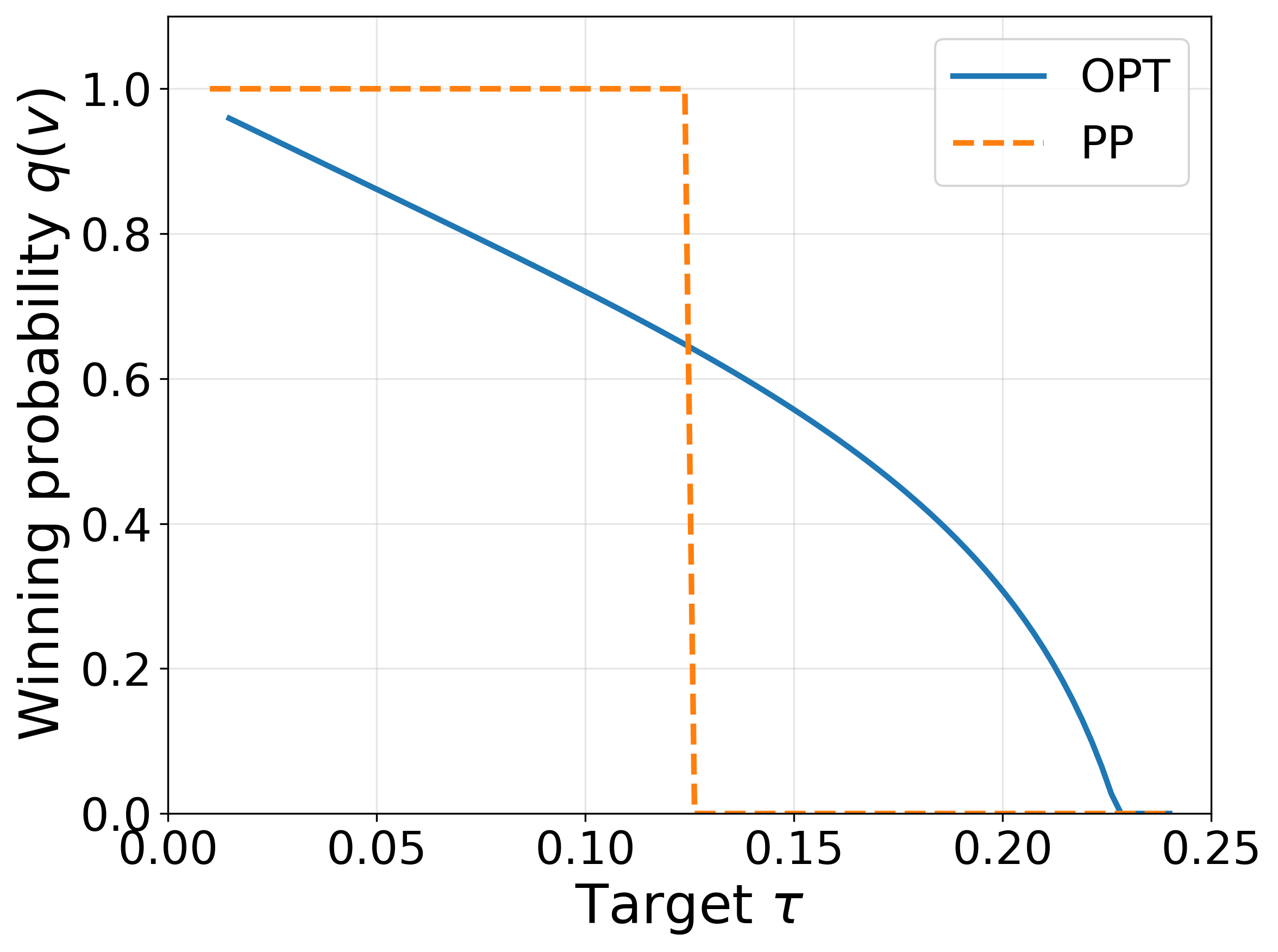}
                \vfill
    				\includegraphics[width=0.96\textwidth]{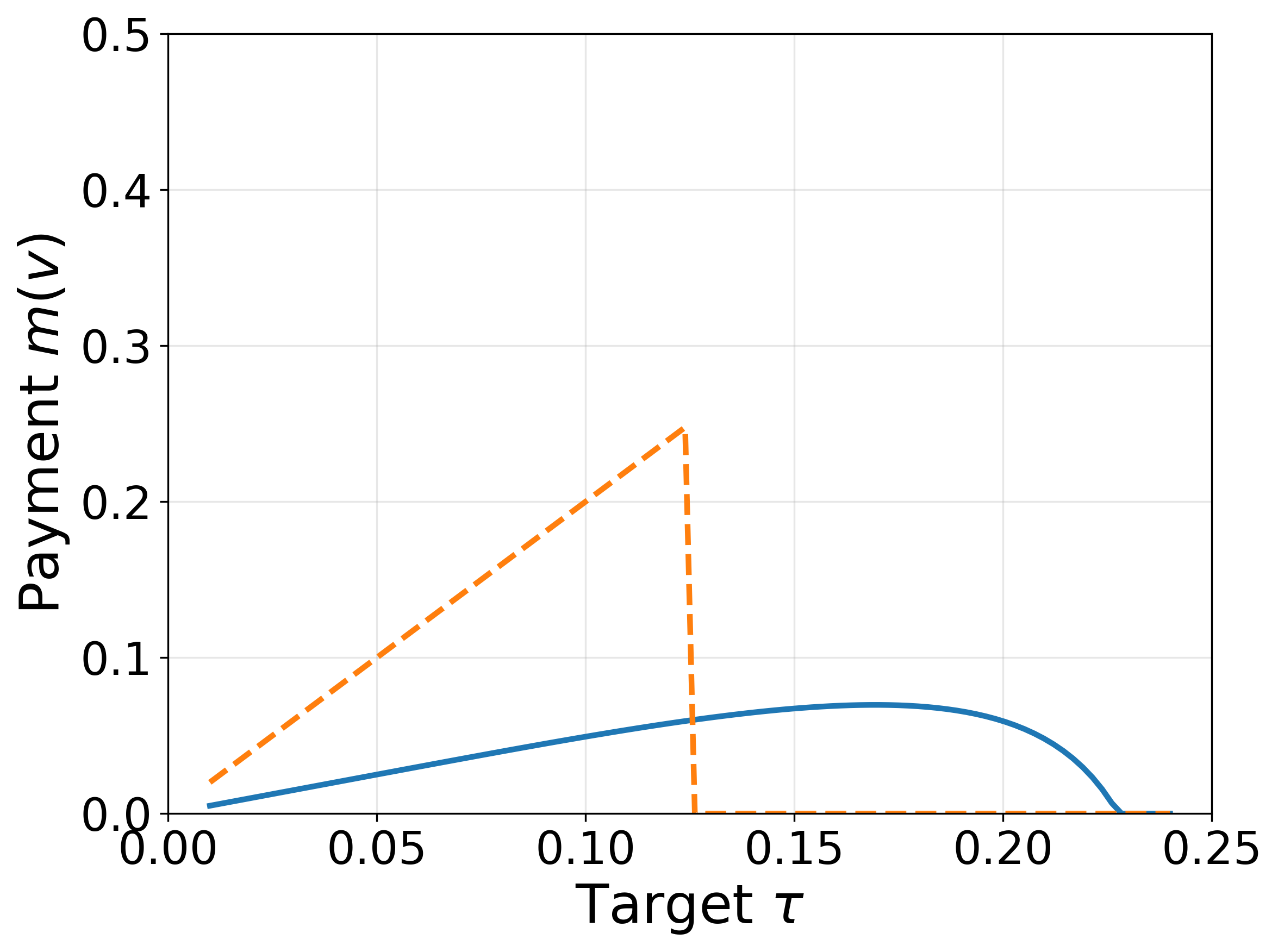}
                \vfill
                    \includegraphics[width=0.96\textwidth]{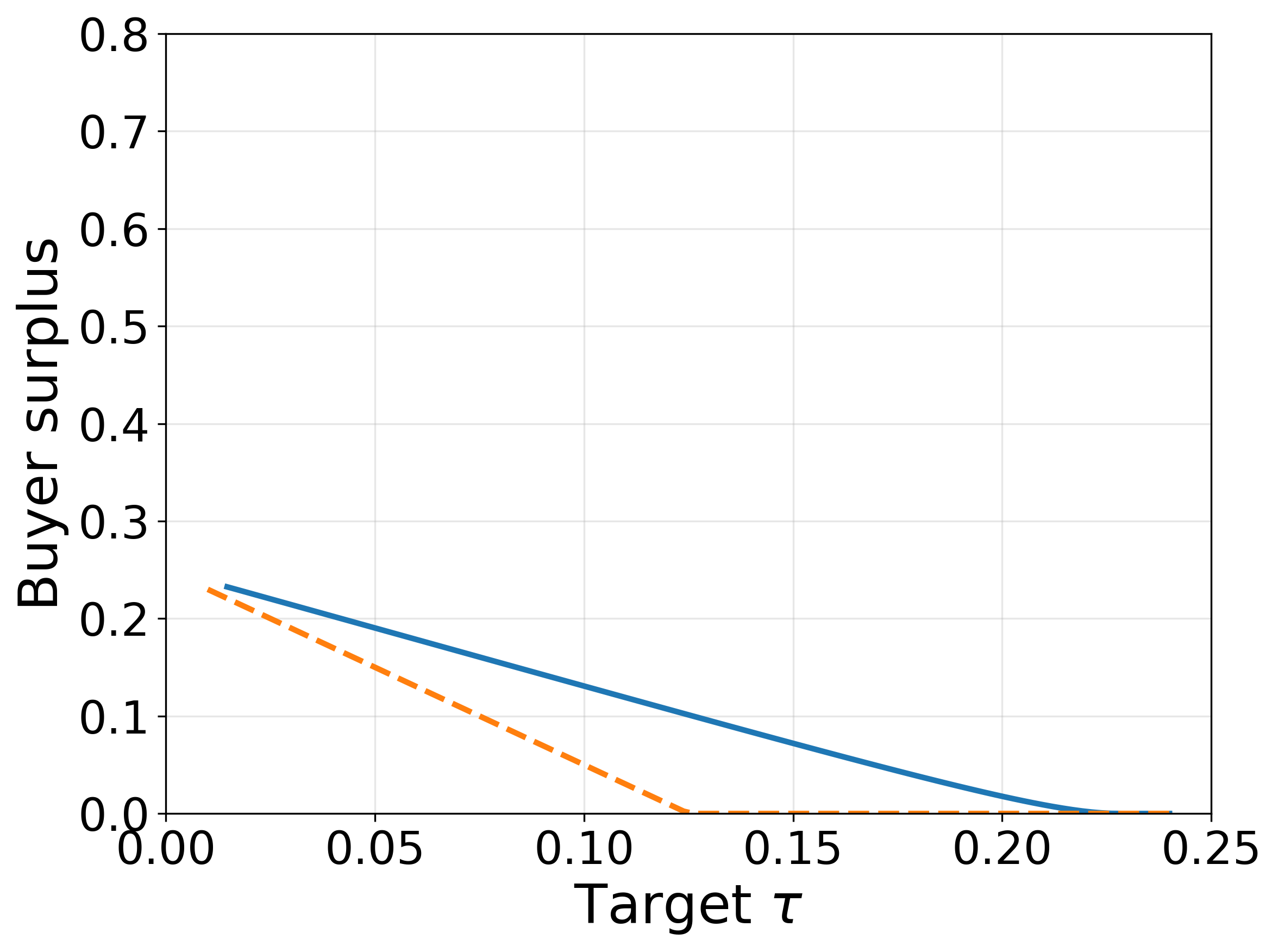}
    		\end{minipage}
            }
            \subfigure[~Medium Valuation ($v=0.5$)]
            {
    		\begin{minipage}{0.33\textwidth}
    				\includegraphics[width=0.96\textwidth]{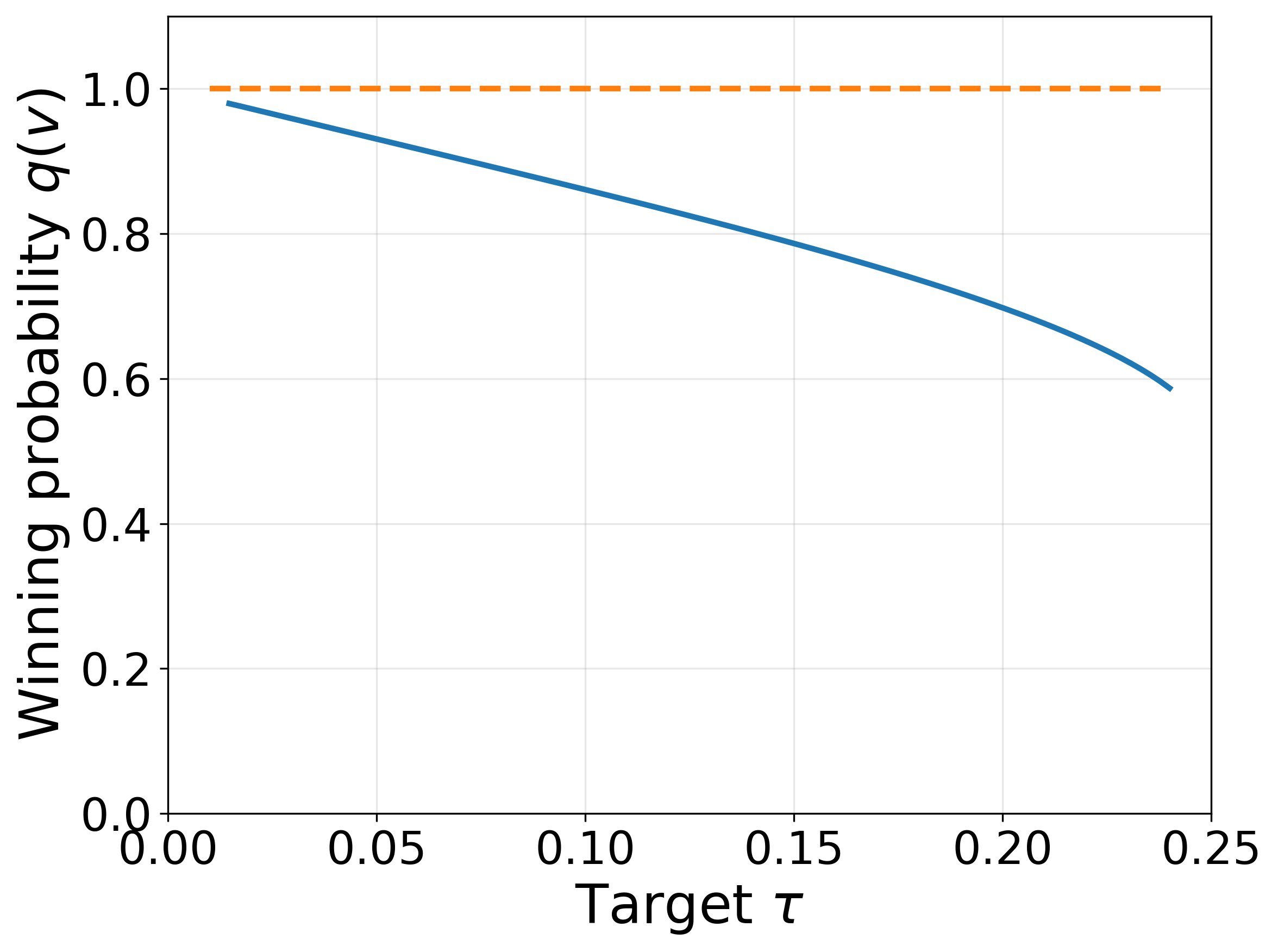}
                \vfill
    				\includegraphics[width=0.96\textwidth]{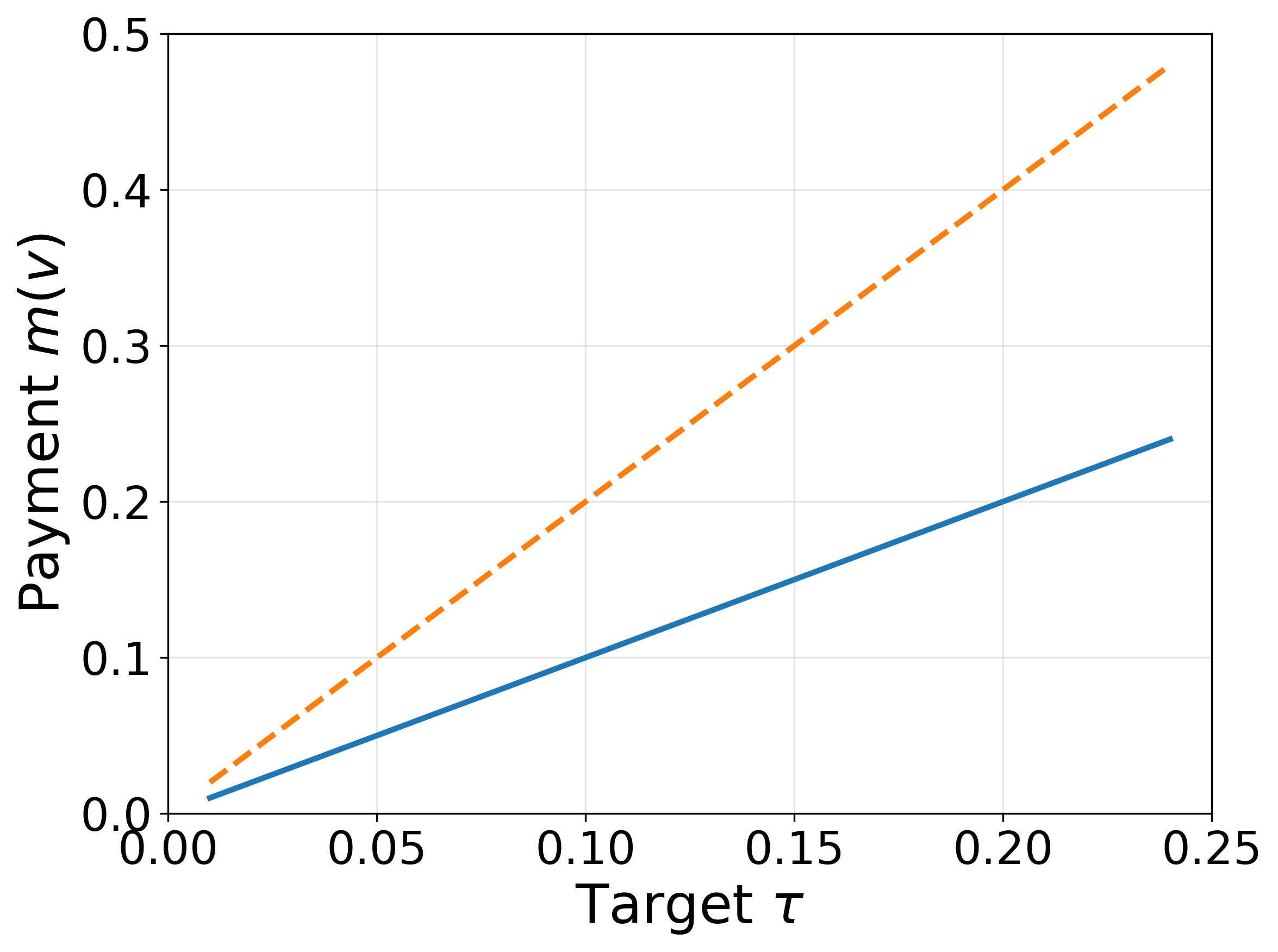}
                \vfill
                    \includegraphics[width=0.96\textwidth]{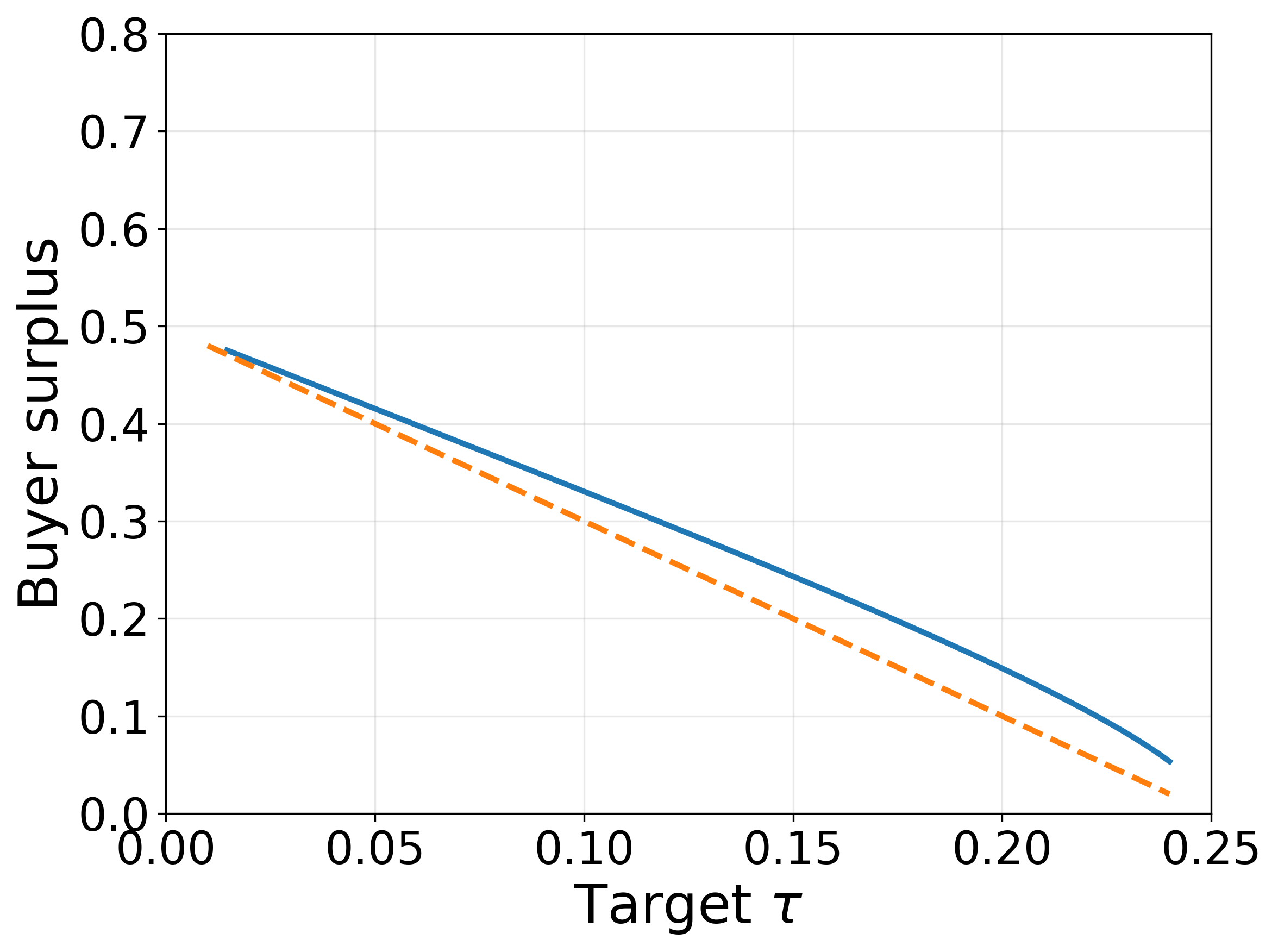}
    		\end{minipage}
            }
            \subfigure[~High Valuation ($v=0.75$)]
            {
    		\begin{minipage}{0.33\textwidth}
    				\includegraphics[width=0.96\textwidth]{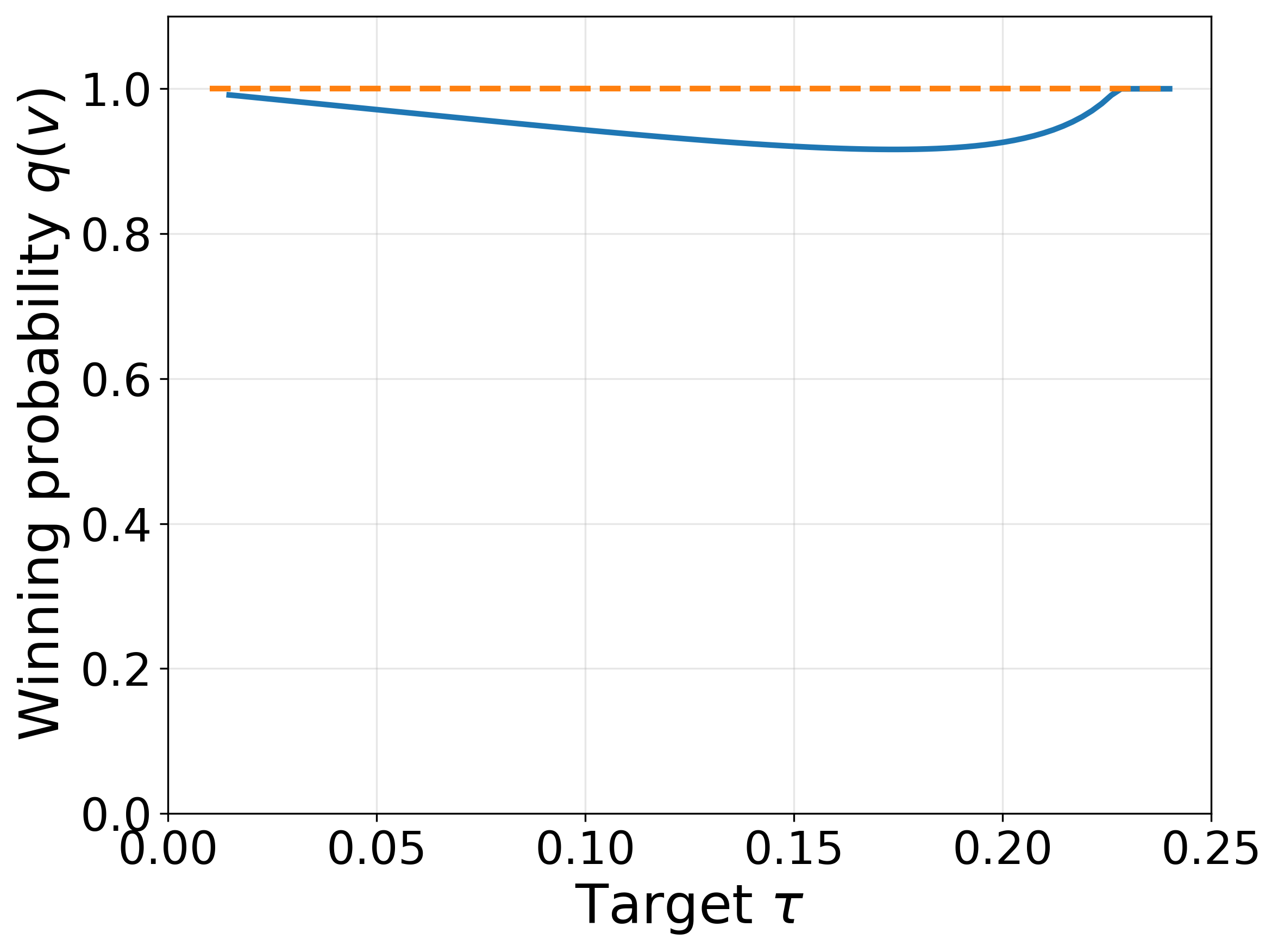}
                \vfill
    				\includegraphics[width=0.96\textwidth]{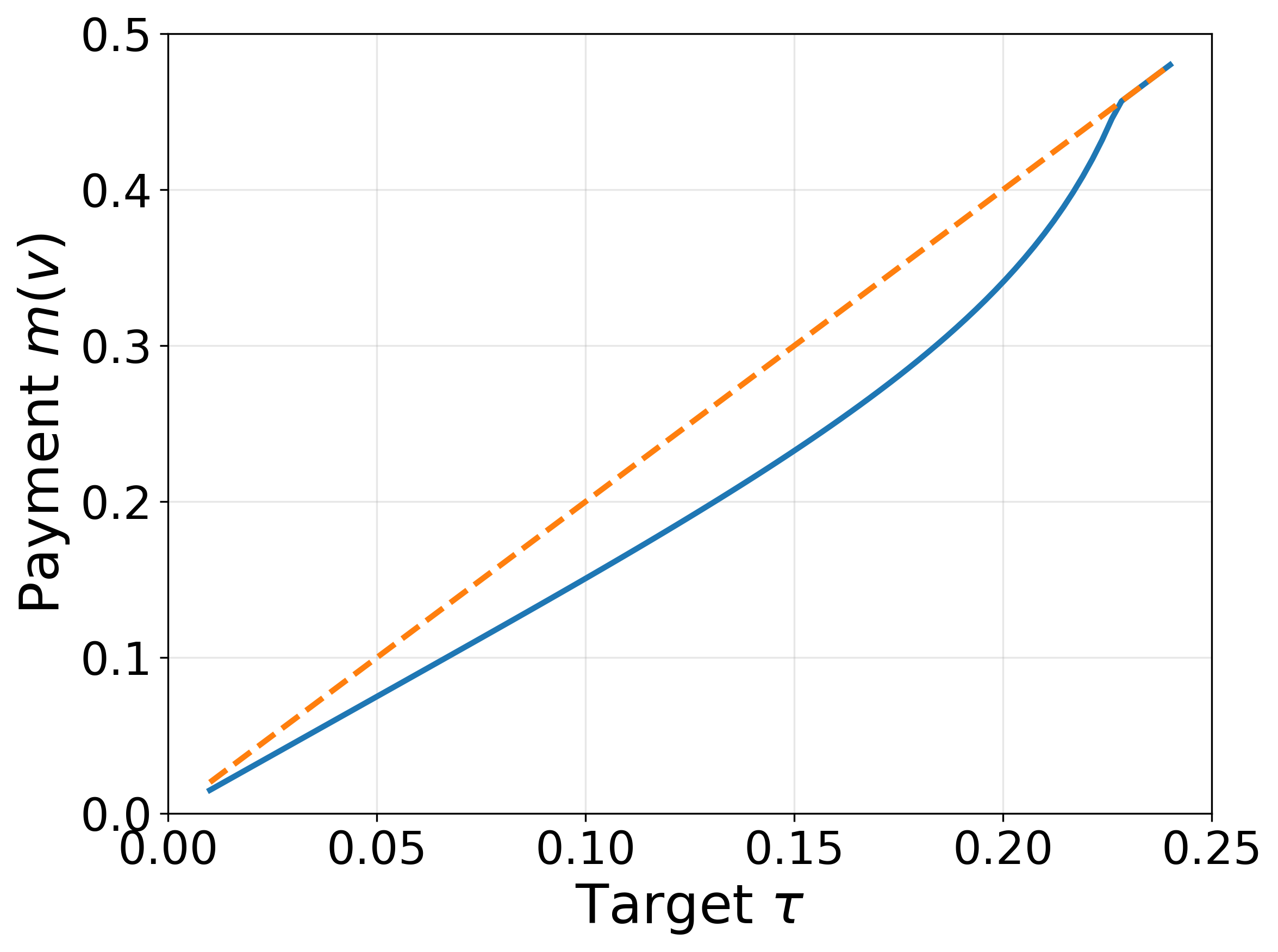}
                \vfill
                    \includegraphics[width=0.96\textwidth]{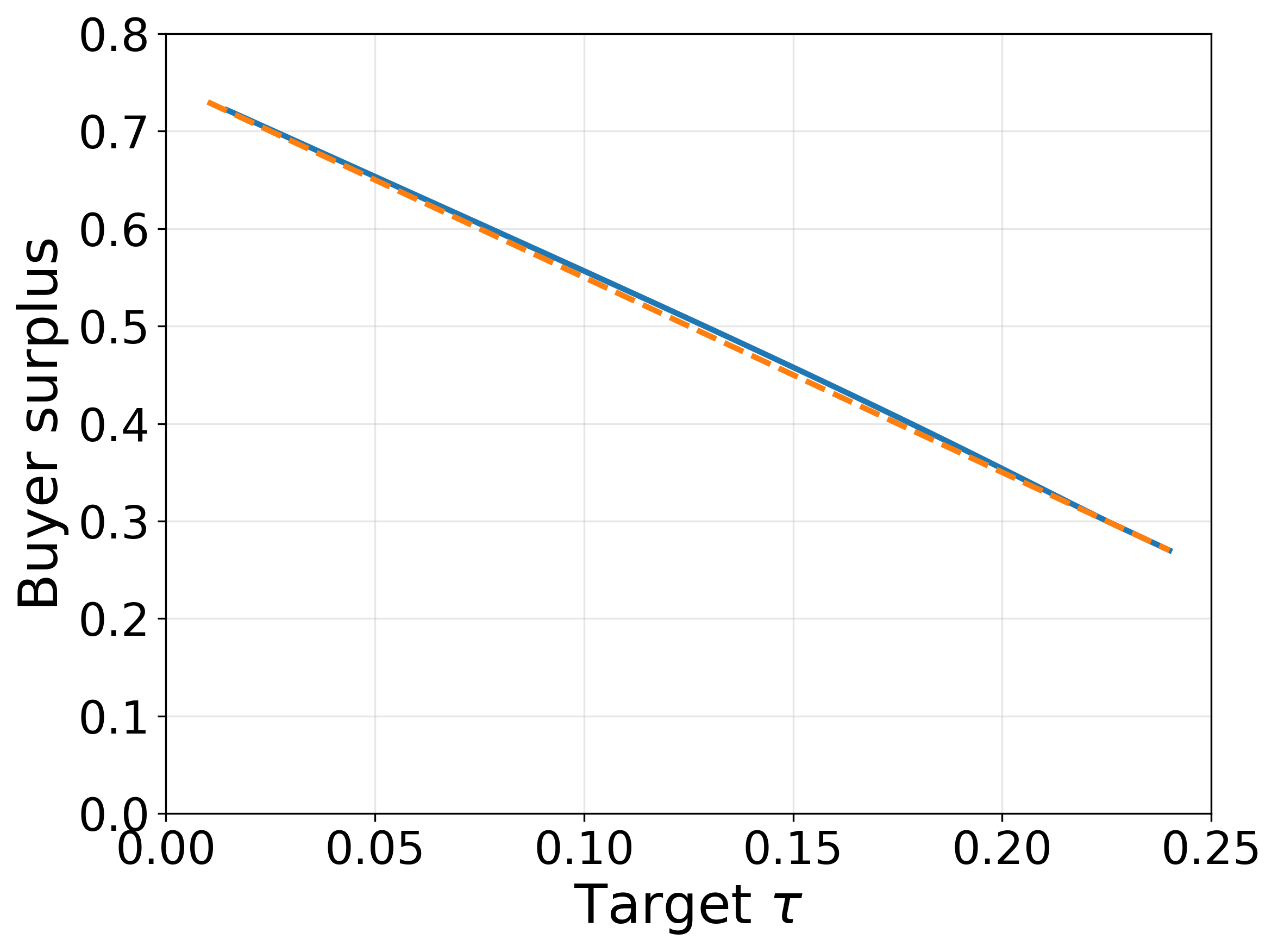}
    		\end{minipage}
            }
	}
	{Optimal vs PP Mechanisms under RS Framework: Winning Probabilities, Payments and Buyer Surplus\label{fig-rs-optVSpp}}
	{}
    \vspace{-0.2in}
\end{figure}
\begin{exap}\label{Exap_2}{\sf(Optimal vs PP Mechanisms under RS Framework)}
Given a buyer valuation distribution of $\{$\$0.25 (50\%), \$0.5 (30\%), \$0.75 (20\%)$\}$, and a seller who, with no prior information, adopts a uniform reference distribution with a revenue target of \$0.2. According to \eqref{equ-opt-RS-price-fragility-uniform}, the seller charges a fixed price of \$0.4 under the PP mechanism, effectively excluding the 50\% of low-valuation buyers valuing at \$0.25. In contrast, the optimal randomized pricing mechanism, as per Theorem \ref{Thm_optimalmechanism}, offers these low-valuation buyers a 30.7\% chance of purchase. Figure \ref{fig-rs-optVSpp} compares the PP and optimal mechanisms for these three buyer groups under the RS framework, detailing differences in winning probabilities, payments, and buyer surplus across various revenue targets.
\end{exap}

\subsubsection{Effectiveness of Posted Pricing Mechanism}\label{sec-rs-pp-effectiveness}
As the optimal mechanism can be implemented as a randomized posted pricing mechanism (Theorem \ref{Thm_optimalmechanism}), this section explores the effectiveness of simple deterministic posted pricing by comparing it to the optimal randomized mechanism. We first analyze the statistical properties of the optimal randomized pricing mechanism in the following proposition.

\begin{prop}{\sf (Statistical Properties of The Optimal Randomized Prices)}\label{prop-statistics-optRS}
    For a uniform reference distribution, the optimal randomized prices under the RS framework, as characterized in Theorem \ref{Thm_optimalmechanism}, exhibit the following statistical properties:
    \beq
        \mathbb{E}_{q_{RS}^*}[\tilde{p}] = p_{RS}^{PP} = 2\tau, \quad  \text{Var}(\tilde{p}) =  \tau(1-4\tau), 
        \quad \text{Skew}(\tilde{p}) =\frac{1}{\sqrt{\tau(1-4\tau)^3}} \left(\frac{2\tau^2}{3 k_{RS}^{*2}}+\frac{1}{2}-6\tau+16\tau^2 \right).
    \eeq
\end{prop} 

\begin{figure}[!htb]
    \vspace{-0.2in}

    \FIGURE
    {
        \begin{minipage}{\textwidth}
            \centering
        \subfigure[~Uniform reference distribution]
        {
    		\includegraphics[width=0.45\textwidth]{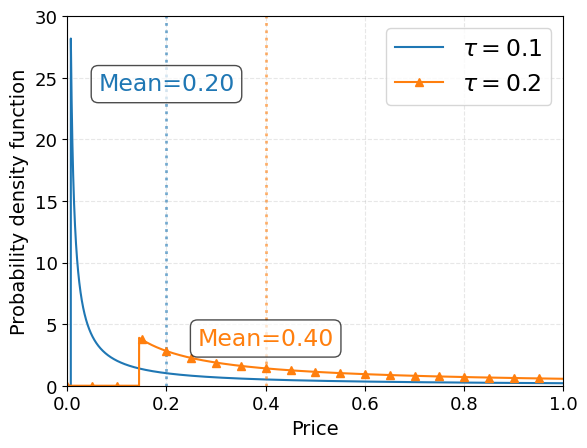}
    		\includegraphics[width=0.45\textwidth]{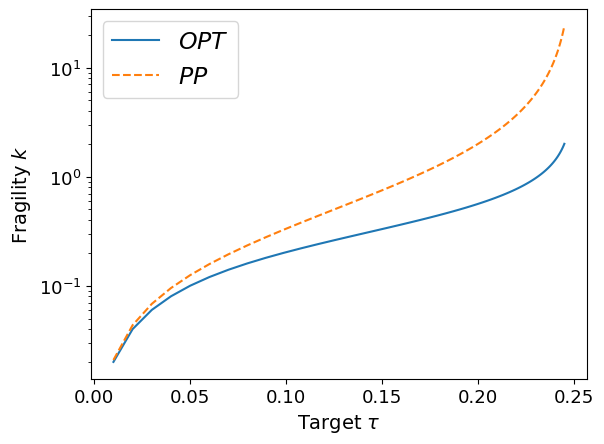}
        }
        \vfill \vspace{-0.1in}
        \subfigure[~Empirical reference distribution $\{(0.3,0.5),(\hat{v}_2,0.5)\}$]
        {
    		\includegraphics[width=0.45\textwidth]{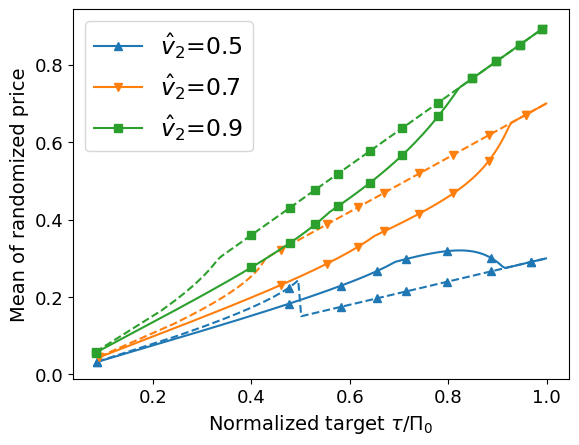}
    		\includegraphics[width=0.45\textwidth]{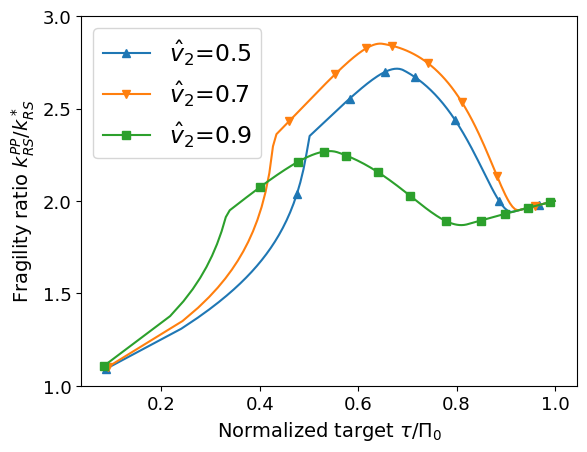}
        }
        \end{minipage}
    }
    {Optimal vs PP Mechanisms under RS Framework: Price and Fragility\label{fig-comparison-k}
    }
    {}
    \vspace{-0.2in}
\end{figure}



Interestingly, these results suggest that the optimal randomized prices serve as an unbiased estimator of the deterministic posted price when the seller adopts a uniform reference distribution. The revenue target parameter $\tau$ crucially controls both the mean and variability of the randomized posted price. As $\tau$ increases, the average randomized price rises, and the variance may decrease for high revenue targets (exceeding threshold of 0.125), indicating a reduced variability in randomized prices. This implies that, for higher revenue targets, randomized prices become more tightly centered around their mean, as illustrated in Figure \ref{fig-comparison-k}(a). This pattern occurs because a more aggressive seller, aiming for a higher revenue target, offers a stochastically higher randomized price (first stochastic order), with greater weight on higher prices toward the upper bound of $\$1$. As a result, price dispersion shrinks and the distribution becomes more concentrated.
Additionally, the randomized price exhibits positive skewness, signifying a fatter right tail. This feature arises from the underlying logarithmic function used in its derivation, which naturally generates a longer right tail. For a general reference distribution, randomized prices may not be unbiased estimators of the deterministic posted price. Figure \ref{fig-comparison-k}(b), for example, shows that average randomized prices (solid curves) can be larger or smaller than the deterministic posted price (dashed curves) under an empirical reference distribution, depending on the revenue target. They coincide only when the revenue target $\tau$ is sufficiently high.


Importantly, the optimal mechanism's aim to minimize fragility results in lower fragility than the PP mechanism, especially at high revenue targets (Figure \ref{fig-comparison-k}). Interestingly, as observed in Figure \ref{fig-rs-optVSpp}, the optimal mechanism can consistently yield a higher buyer surplus than the posted pricing mechanism. This numerical finding is formally established in the following theorem. 

\begin{thm}\label{Thm_comparison_opt_pp_bs}{\sf (Buyer Surplus Dominance)}
Assuming a regular reference distribution, for any given target $\tau$, there exists a threshold $\underline{v}_{RS}\in[q_{RS}^{PP},1]$ such that, compared to the posted pricing mechanism, the optimal mechanism yields a higher surplus for a buyer with valuation $v$:
\beq
q_{RS}^*(v)v-m_{RS}^*(v)\geq q_{RS}^{PP}(v)v-m_{RS}^{PP}(v) & \text{~if and only if~}& v\leq \underline{v}_{RS}. 
\eeq
Furthermore, for a power reference distribution (where $P_0(x)=x^{\alpha}$ for some $\alpha\geq 1$), the optimal mechanism yields a higher surplus than the PP mechanism for all buyers.
\end{thm}

The above theorem demonstrates that, compared to the posted pricing mechanism, the optimal mechanism can achieve a Pareto `win-win' outcome, resulting in lower fragility for the seller and higher surplus for buyers, particularly those with lower valuations. This Pareto outcome occurs when the reference distribution satisfies Myerson's regularity condition. Specifically, when the reference is a power distribution, the optimal mechanism provides a higher surplus to all buyers, regardless of their valuations, compared to the posted pricing mechanism. This dominance in buyer surplus holds irrespective of the revenue target $\tau$. Conversely, for a general regular reference distribution, the optimal mechanism tends to benefit buyers with low valuations, while the posted pricing mechanism favors those with high valuations, as illustrated in Example \ref{exap-exponential-opt-PP}.

\begin{exap}\label{exap-exponential-opt-PP}{\sf (When a PP Mechanism Yields a Higher Buyer Surplus)}
    Assuming the reference distribution is a truncated exponential with $P_0(x) = \frac{1-e^{-x}}{1-e^{-1}}$, and given $\tau = 0.1$, we find that for the highest valuation buyer with $v=1$, $q_{RS}^*(1) = q_{RS}^{PP}(1) = 1$, $m_{RS}^*(1) = 0.238$, $m_{RS}^{PP}(1) = 0.228$. The surplus for this buyer under the optimal mechanism is $0.762 (= q_{RS}^*(1) - m_{RS}^*(1))$, while under the PP mechanism, it is $0.772 (=q_{RS}^{PP}(1) - m_{RS}^{PP}(1))$. Thus, the PP mechanism yields a higher surplus for this buyer than the optimal mechanism. Figure \ref{Fig_opt_comparison_exp} illustrates that the PP mechanism generates a higher surplus for high-valuation buyers, whereas the optimal mechanism yields a greater surplus for low-valuation buyers.  
\end{exap}
\begin{figure}[!htb]
    \vspace{-0.1in}
    \FIGURE
    {
        \includegraphics[width=0.33\textwidth]{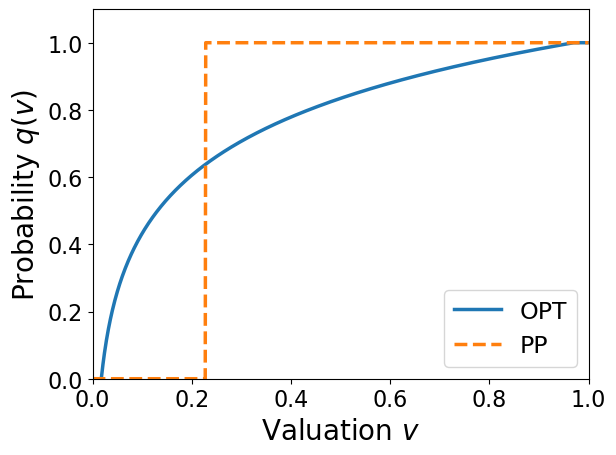}
        \includegraphics[width=0.33\textwidth]{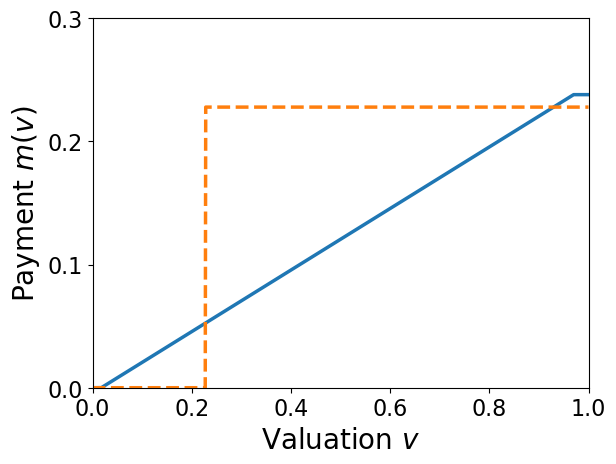}  
        \includegraphics[width=0.33\textwidth]{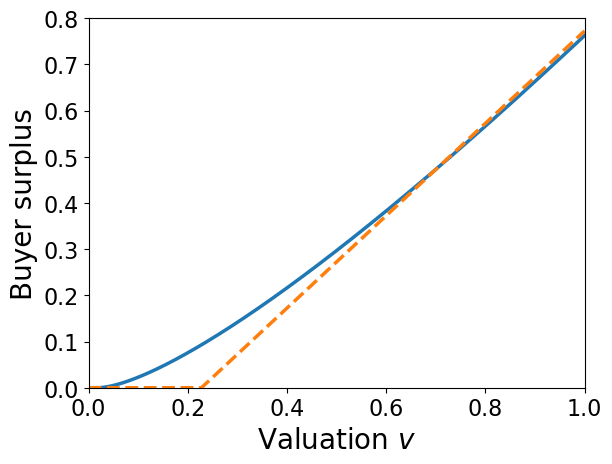}
    }
    {OPT vs PP Mechanisms under an Exponential Reference Distribution  \label{Fig_opt_comparison_exp}}
    {}
    \vspace{-0.2in}
\end{figure}

Finally, we conduct extensive numerical analysis to evaluate the \emph{out-of-sample} performance of the deterministic PP mechanism on seller revenue. Using the true valuation distribution $P$, we compare its expected revenue to that of the optimal randomized mechanism, with the out-of-sample performance ratio formally defined as:
$$\eta_{RS} := \frac{\mathbb{E}_{P}[p_{RS}^{PP}\mathbbm{1}(\tilde{v}\geq p_{RS}^{PP})]}{\mathbb{E}_P[m_{RS}^*(\tilde{v})]}.$$ This quantifies the expected seller revenue of the posted price relative to the fragility-optimal mechanism, both derived under the RS framework and evaluated at the true valuation distribution. Our numerical experiments, using a uniform reference distribution to reflect typical seller information scarcity, evaluate the out-of-sample performance ratio of the PP mechanism across various Beta-distributed true valuations. 

\begin{figure}[!ht]
\vspace{-0.1in}
    \centering
    \FIGURE
    {
        \includegraphics[width=  0.45\linewidth]{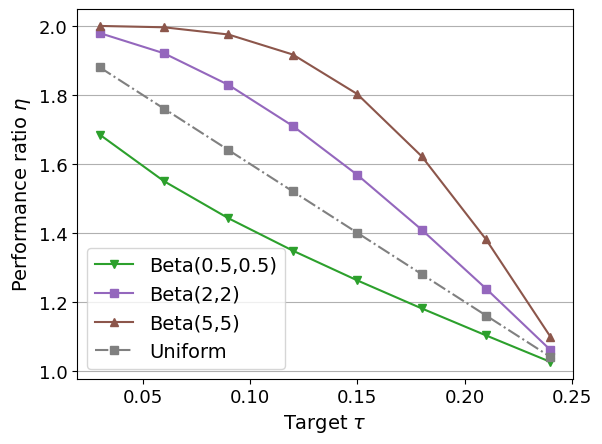}
        \includegraphics[width=  0.44\linewidth]{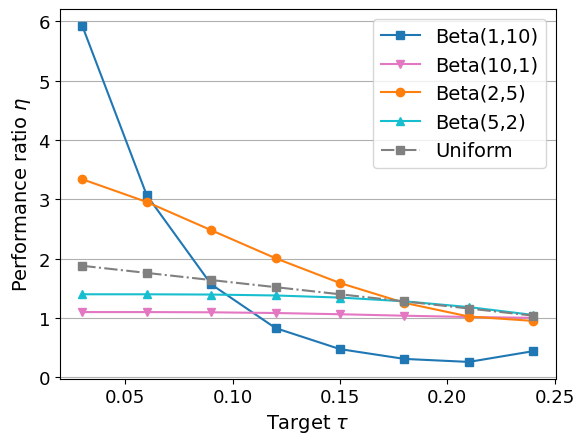}
    }
    {Out-of-Sample Performance Ratio $(\eta_{RS})$ with Uniform Reference Distribution \label{Fig_PRatio_GUniform} 
    }
    {This figure illustrates the out-of-sample performance ratio of a PP mechanism. The left panel shows symmetric true valuation distributions with identical means, while the right panel depicts asymmetric ones.}
\vspace{-0.1in}
\end{figure}

Our numerical analysis reveals that while the optimal randomized mechanism minimizes fragility, it does not consistently outperform the deterministic PP mechanism in terms of revenue, especially at lower revenue targets (see Figure \ref{Fig_PRatio_GUniform}). The performance of the deterministic PP mechanism is significantly influenced by the revenue target, variability, and skewness of the underlying valuations. For symmetric valuation distributions (left panel of Figure \ref{Fig_PRatio_GUniform}, where Beta$(\alpha,\alpha)$ has a mean of 0.5 and variance $1/(4(2\alpha+1))$), the deterministic PP mechanism performs best with lower valuation variability. In contrast, for skewed distributions (right panel of Figure \ref{Fig_PRatio_GUniform}, where the skewness of Beta$(\alpha,\beta)$ is proportional to $\beta-\alpha$), the deterministic PP mechanism can significantly outperform the randomized approach, particularly at low revenue targets when skewness is positive (indicating a concentration of small values with a fatter right tail, as illustrated in Figure \ref{Fig_Beta_PDF}). The randomized mechanism tends to be overly conservative under low revenue targets, sacrificing potential revenue to manage fragility against a fatter right tail. As sellers pursue higher revenue targets, the effectiveness of the deterministic PP mechanism declines in both expected revenue and the risk of underperforming the target, particularly when valuations exhibit high positive skewness. More numerical results can be found in Figure \ref{Fig_optvspp_rs_comparison} to verify that the deterministic PP mechanism is more effective for low revenue targets or positive skewness in valuations.

\section{Robust Mechanisms: RS vs RO Frameworks}\label{sec-RS-vs-RO}
As outlined in Section \ref{Sec_Setup}, in the context of distributional ambiguity, the robust optimal mechanism \eqref{Prob_RO} seeks to maximize worst-case revenue, whereas the robust satisficing mechanism \eqref{Prob_RS} aims to minimize the fragility of falling short of a specific target. In this section, we compare these two robust frameworks to evaluate their effects on the seller's expected revenue, buyer surplus, allocation probabilities, and the effectiveness of deterministic posted pricing mechanisms. For self-completeness, we first restate the robust optimal mechanism in the following lemma, as examined in Proposition 6 of \cite{chen2024screening}.

\begin{lem}\label{Lm_optimalmechanism} {\sf (Optimal RO Mechanism)}
Given the ambiguity size $r$, the maximum worst-case revenue $\Pi_{RO}^*(r)$ is the unique solution of $\pi \in [0,1]$ to $\int_0^1[\bar{P}_0(x)-\pi/x]^+\mathrm{d}x = r$. The robust optimal mechanism to \eqref{Prob_RO} is given by: For each $j=0,1,\dots,J$,
\beq
\left(q_{RO}^*\left(v\right),m_{RO}^*\left(v\right)\right) = \left\{\begin{array}{ll}
    \alpha \big(\ln (\frac{v}{u_j})+\sum_{i=1}^{j-1} \ln (\frac{w_i}{u_i}),\, (v-u_j)+\sum_{i=1}^{j-1}(w_i-u_i) \big), \quad v \in [u_j,w_j)\\
    \alpha \big(\sum_{i=1}^{j} \ln (\frac{w_i}{u_i}),\, \sum_{i=1}^{j}(w_i-u_i) \big), \quad v \in[w_j,u_{j+1}),   
\end{array}\right.
\eeq
where $J$, $u_j$, and $w_j$ denote $J(\Pi_{RO}^*(r))$, $u_j(\Pi_{RO}^*(r))$, and $w_j(\Pi_{RO}^*(r))$, respectively, and $\alpha = 1/\sum_{j \in [J]}\ln(w_j/u_j)$. Additionally, for a uniform reference distribution (i.e., $P_0(x)=x$), $\Pi_{RO}^*(r)$ is the unique solution of $\pi\in[0,1]$ to the equation $r = \frac{\sqrt{1-4\pi}}{2}-\pi \ln \frac{1+\sqrt{1-4\pi}}{1-\sqrt{1-4\pi}}$. In this scenario, $J=1$, $u_1 = \frac{1-\sqrt{1-4\Pi_{RO}^*(r)}}{2}$, $w_1 = \frac{1+\sqrt{1-4\Pi_{RO}^*(r)}}{2}$, and $\alpha = \frac{1}{\ln(1+\sqrt{1-4\Pi_{RO}^*(r)})-\ln (1-\sqrt{1-4\Pi_{RO}^*(r)})}$. The optimal deterministic posted price is $p_{RO}^{PP} = \frac{1-\sqrt{2r}}{2}$.
\end{lem}

Interestingly, the optimal mechanism under the RO framework shares the same structure with the one under the RS framework. Specifically, as discussed following Theorem \ref{Thm_optimalmechanism}, the robust optimal mechanism can be implemented using a randomized pricing mechanism, described in \eqref{Prob_RP}, where the randomized price $\tilde{p}$ is drawn from the probability distribution $q_{RO}^*\in\mathcal{P}$. The sole difference between these two frameworks is the choice of revenue level used to select the intersections $\{u_j,w_j\}_{j\in[J]}$ between the reference distribution $\bar{P}_0(x)$ and the iso-revenue curve $\pi/x$. In the RO framework, the revenue level is set to the maximum worst-case revenue $\Pi_{RO}^*(r)$, whereas in the RS framework, it corresponds to the worst-case revenue $\pi^*(k_{RS}^*(\tau))$ such that the fragility-adjusted revenue meets the revenue target $\tau$. In the next proposition, we demonstrate that these two robust frameworks are equivalent, producing the exact same mechanism when the two exogenous parameters -- the ambiguity size $r$ in the RO framework and the revenue target $\tau$ in the RS framework -- satisfy the equivalence condition specified in \eqref{equ-equivalence-condition}.

\begin{thm}\label{Prop_Connection}{\sf (Equivalence Between RO and RS Mechanisms)}
    Assume that the ambiguity size $r$ and the revenue target $\tau$ satisfy:
    \bea
    \tau = \tau^{\sf Equiv}(r), \quad \mbox{where} \quad \tau^{\sf Equiv}(r):= \Pi_{RO}^*(r) + \frac{r}{\sum_{j \in [J]} \ln \left[w_j\left(\Pi_{RO}^*(r)\right)/u_j\left(\Pi_{RO}^*(r)\right)\right]}. \label{equ-equivalence-condition}
    \eea
    Then, the frameworks of distributionally robust optimization \eqref{Prob_RO} and robust satisficing \eqref{Prob_RS} are equivalent, resulting in the same optimal selling mechanism.
\end{thm}

Notably, Theorem \ref{Prop_Connection} establishes the equivalence between the two robust frameworks in the context of mechanism design under distributional ambiguity. Specifically, for any given revenue target, the radius parameter in the RO problem can be selected such that the family of the RO mechanisms belongs to the family of the RS mechanisms. Conversely, for any given radius of ambiguity set, there is a corresponding revenue target in the RS problem such that the family of the RS mechanisms belongs to the family of the RO mechanisms. In different contexts such as network lot-sizing and portfolio optimization, \cite{wang2025equivalence} identified the conditions under which a similar equivalence holds between the RO and RS frameworks. 

\begin{figure}[!htb]
    \vspace{-0.1in}
    \FIGURE
    {
        \includegraphics[width=0.33\textwidth]{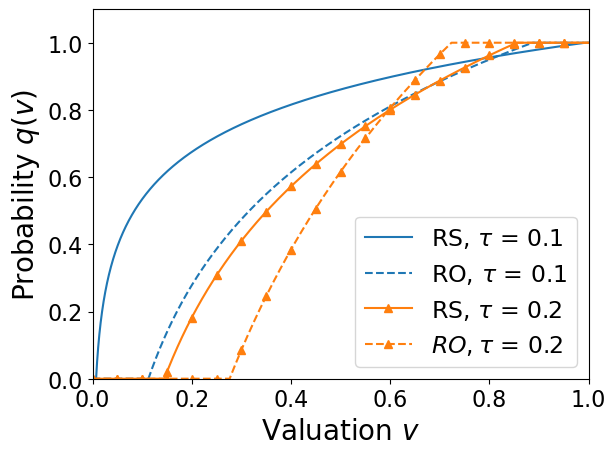}
        \includegraphics[width=0.33\textwidth]{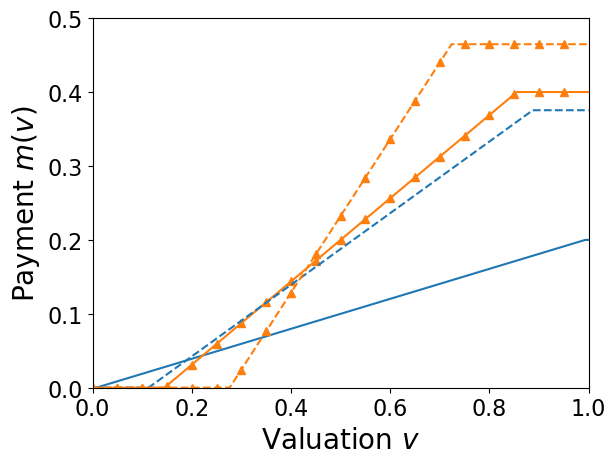} 
        \includegraphics[width=0.33\textwidth]{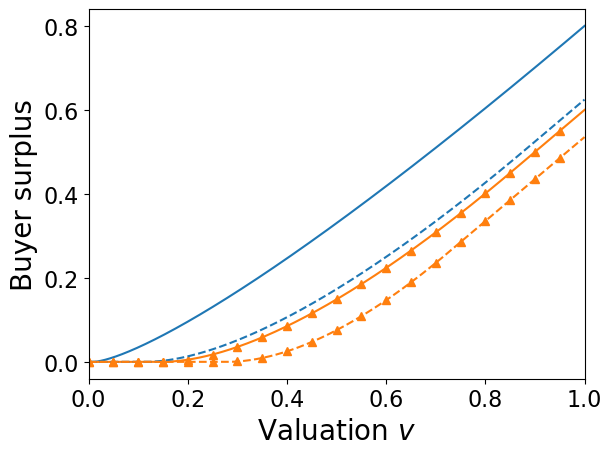}
    }
    {RS vs RO Mechanisms under a Uniform Reference Distribution\label{Fig_opt_comparison}}
    {The in-sample worst-case expected revenue is $0.1$ for RO and $0.007$ for RS when $\tau=0.1$, and $0.2$ for RO and $0.12$ for RS when $\tau=0.2$.}
    \vspace{-0.1in}
\end{figure}

Although the equivalence condition can be analytically described, finding the equivalent target for a given radius is challenging due to the difficulty in identifying all intersections $\{u_j,w_j\}_{j\in[J]}$. In the remainder of this section, we compare the two mechanisms by focusing on a common practical choice for the target: setting it equal to the maximum worst-case revenue associated with a given radius, i.e., $\tau = \Pi_{RO}^*(r)$. Notably, compared to RO, the RS framework tends to favor low-valuation buyers by allocating a higher winning probability, as formally stated in Proposition \ref{Prop_ComparisonU}.


\begin{prop}\label{Prop_ComparisonU}
Assuming $\tau = \Pi_{RO}^*(r)$, the RS mechanism allocates a higher selling probability to low-valuation buyers compared to the RO mechanism, i.e., $q_{RS}^*(v) \ge q_{RO}^*(v)$ for any $v\le u_1\left(\Pi_{RO}^*\left(r\right)\right)$.
\end{prop}


Importantly, our analysis reveals that the RS framework consistently provides a higher selling chance to low-valuation buyers than the RO framework, regardless of the reference distribution chosen. As an illustrative example, we compare these two mechanisms under a uniform reference distribution across various dimensions -- winning probability, payments, and buyer surplus -- in Figure \ref{Fig_opt_comparison}. We observe that, compared to the RO mechanism, the RS mechanism allocates a higher selling probability and charges higher payments to low-valuation buyers, while offering a lower selling probability and lower payments to high-valuation buyers. Interestingly, the RS mechanism results in a higher surplus for all buyers than the RO mechanism, as shown in the following theorem.

\begin{thm}\label{thm_universal-dominance-rs-ro}{\sf (Single-Crossing and Buyer Surplus Dominance)}
Assume that the revenue target satisfies $\tau=\Pi^*_{RO}(r)$ and that the reference distribution $P_0$ has an increasing hazard rate. Then, the RS mechanism provides a higher surplus for all buyers compared to the RO mechanism. Furthermore, there exist thresholds $\underline{v}_q,\underline{v}_m \in [0,1]$ such that:
\beq
 q_{RS}^*(v) \geq q_{RO}^*(v)  & \text{~if and only if~}& v \leq \underline{v}_q, \\
m_{RS}^*(v) \geq m_{RO}^*(v)  & \text{~if and only if~}& v \leq \underline{v}_m.
\eeq
\end{thm}

Theorem \ref{thm_universal-dominance-rs-ro} provides a unified buyer-side comparison between the RS and RO mechanisms. When the reference distribution has an increasing hazard rate -- characteristic of regular distributions like uniform, normal, or exponential -- we show that single-crossing properties hold for both allocation and payment. Consequently, the RS mechanism yields a higher surplus for all buyers compared to the RO mechanism. This result highlights that different robustness paradigms can lead to systematically different buyer outcomes, even when the mechanisms belong to the same class.

To understand this result, it is useful to examine how the two mechanisms differ across valuation levels. The RS mechanism reallocates probability mass toward lower valuations, whereas the RO mechanism concentrates allocation on higher valuations. Consequently, low-valuation buyers face a higher probability of receiving the good under RS, albeit at a higher expected payment, while high-valuation buyers experience both a lower allocation probability and a lower payment. Conditional on receiving the good under both mechanisms, payments are consistently lower under RS, as shown in Figure \ref{Fig_comparison_r}. These effects generate opposing forces on buyer surplus: for low-valuation buyers, the gain from a higher allocation probability outweighs the increase in payment, leading to higher surplus; for high-valuation buyers, the reduction in allocation probability may offset the benefit of lower payments. The single-crossing condition ensures that, overall, the net effect favors the RS mechanism, particularly for buyers with lower valuations.

Notably, the increasing hazard rate condition is sufficient but not necessary for the surplus dominance result. As shown in Lemma \ref{lem-ifr-sensitivity}, the key requirement can be relaxed to a more general condition requiring that the revenue function is more sensitive to changes in higher posted prices than in lower ones along iso-revenue curves. This condition admits a clear economic interpretation. It implies that marginal adjustments to higher prices have a disproportionately larger impact on revenue than comparable adjustments at lower prices. As a result, the RS mechanism that shift allocation toward lower valuations can improve buyer surplus without significantly weakening revenue robustness. Importantly, this insight extends well beyond the class of distributions with increasing hazard rates. Figure \ref{Fig_opt_comparison_pareto} in the Appendix shows that even for regular distributions with non-monotone hazard rates (e.g., truncated Pareto), the single-crossing structure and the resulting surplus dominance continue to hold. Moreover, Figure \ref{Fig_opt_comparison_irregular} demonstrates that this dominance persists even under irregular reference distributions, where standard regularity conditions fail. Taken together, these results indicate that the surplus advantage of the RS mechanism is driven by a more fundamental structural feature-its systematic reallocation toward lower valuations-rather than by specific distributional assumptions. Consequently, the RS framework yields higher buyer surplus across a broad class of economically relevant environments, particularly benefiting low-valuation buyers. The impact of these mechanisms on individual buyers is further illustrated in Example \ref{Exap_5}.

\begin{exap}\label{Exap_5}{\sf(RS vs RO Mechanisms)}
Continuing with Example \ref{Exap_2}, we derive the optimal RO mechanism corresponding to the ambiguity size that achieves the same target $\tau = \Pi_{RO}^*(r)$. Table \ref{table-example-rs-vs-ro} summarizes the impacts of three mechanisms: the PP mechanism under the RS framework, the optimal RS mechanism, and the optimal RO mechanism, on buyers with various valuations. The analysis reveals a distinct advantage of the optimal RS mechanism. Firstly, it sells the item to the low-valuation buyer with a valuation of \$0.25, offering a 30.7\% chance of acquisition, whereas the PP and optimal RO mechanisms offer no such chance. Secondly, for buyers with a positive chance of acquisition across all mechanisms, the optimal RS mechanism charges the lowest expected payment -- 50\% (or 15\%) less than the PP mechanism and 13\% (or 26\%) less than the optimal RO mechanism for the medium-valuation (or high-valuation) buyers. Thirdly, the optimal RS mechanism provides the highest surplus for all buyers.
\begin{table}
\vspace{-0.2in}
\TABLE
{An Illustration Example for RS vs RO Mechanisms: Winning Probability, Payment, Buyer Surplus\label{table-example-rs-vs-ro}}
{
\begin{tabular}{c ccc  ccc   ccc}
\toprule
& \multicolumn{3}{c}{PP mechanism (RS framework)} & \multicolumn{3}{c}{Optimal RS mechanism} & \multicolumn{3}{c}{Optimal RO mechanism}\\
\cmidrule(lr){2-4} \cmidrule(lr){5-7} \cmidrule(lr){8-10}
\makecell{Buyer\\ valuation} & \makecell{Winning\\ prob.} & \makecell{Expected\\ payment} & \makecell{Buyer\\ surplus} &\makecell{Winning\\ prob.} & \makecell{Expected\\ payment} & \makecell{Buyer\\ surplus} & \makecell{Winning\\ prob.} & \makecell{Expected\\ payment} & \makecell{Buyer\\ surplus}\\
\hline
\$ 0.25 & 0 & \$ 0.0 & \$ 0.0 & 0.307 & \$ 0.06 & \$ 0.02  & 0 & \$ 0.0 & \$ 0.0\\
\$ 0.5 & 1 & \$ 0.4 & \$ 0.1 & 0.698 & \$ 0.20 & \$ 0.15 & 0.616 & \$ 0.23 & \$ 0.08 \\
\$ 0.75  & 1 & \$ 0.4 & \$ 0.35 & 0.926 & \$ 0.34 & \$ 0.35 & 1 & \$ 0.46 & \$ 0.29\\
\bottomrule
\end{tabular}}
\noindent{\footnotesize\textit{Note.} In this table, with a uniform reference distribution and $\tau=0.2$, the in-sample worst-case expected revenues are $0.12$ for the optimal RS mechanism, $0.16$ for the PP mechanism under the RS framework, and $0.2$ for the optimal RO mechanism.}
\vspace{-0.1in}
\end{table}
\end{exap}


Next, we investigate how the effectiveness of the PP mechanism changes from the RS framework to the RO framework. As demonstrated in Section \ref{sec-rs-pp-effectiveness}, under the RS framework, the PP mechanism with a deterministic posted price can yield a higher \emph{out-of-sample} revenue than the optimal RS mechanism with randomized prices, particularly at low revenue targets when valuations are positively skewed. We seek to explore how these insights regarding the out-of-sample performance of PP mechanism changes under the RO framework. Analogous to the RS framework, we define the out-of-sample performance ratio under the RO framework as $\eta_{RO}:= \frac{\mathbb{E}_{P}[p_{RO}^{PP}\mathbbm{1}(\tilde{v}\geq p_{RO}^{PP})]}{\mathbb{E}_P[m_{RO}^*(\tilde{v})]}$, which represents revenue achieved under the PP mechanism relative to the optimal RO mechanism, both derived using a uniform reference distribution and evaluated at the true valuation distribution $P$.

\begin{prop}\label{Prop_ComparisonStatistics}{\sf (Statistical Properties of Randomized Prices)}
Assuming $\tau = \Pi_{RO}^*(r)$ and a uniform reference distribution, it follows that $\mathbb{E}_{q_{RO}^*}[\tilde{p}] >  \mathbb{E}_{q_{RS}^*}[\tilde{p}] = p_{RS}^{PP}$ and $\mathbb{E}_{q_{RO}^*}[\tilde{p}] > p_{RO}^{PP}$.
\end{prop}


Importantly, our analysis reveals the following key differences between the RO and RS frameworks. First, we compare randomized prices associated with the optimal mechanism with the deterministic price for the PP mechanism in both robust frameworks. We show that, while the randomization does not change the (mean) posted price in the RS framework, it increases the posted price in the RO framework, as formally stated in Proposition \ref{Prop_ComparisonStatistics} and numerically depicted in the left panel of Figure \ref{Fig_ComparisonStatistics}. Second, we examine how the out-of-sample performance ratio of the PP mechanism changes from the RS to RO framework, focusing on Beta distributions as true valuation distribution, as outlined in Figure \ref{Fig_ComparisonStatistics}. While the PP mechanism tends to be more effective when valuations exhibit a lower variability or higher skewness in the RS framework, per discussed in Section \ref{sec-rs-pp-effectiveness}, its effectiveness is marginally affected by these factors in the RO framework. 

\begin{figure}[!htb]
    \vspace{-0.1in}
    \FIGURE
    {
        \includegraphics[width=0.33\textwidth]{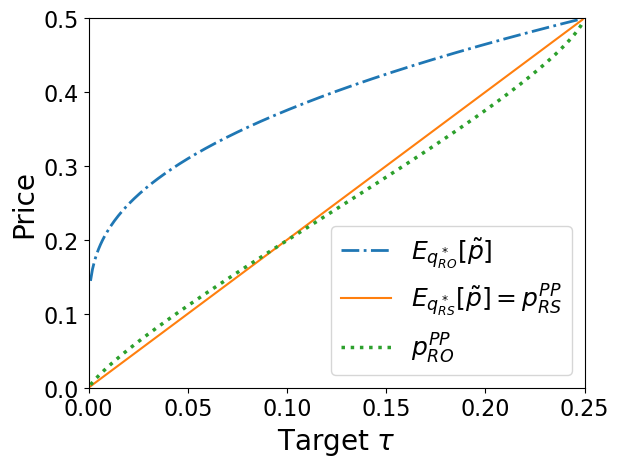}
        \includegraphics[width=0.33\textwidth]{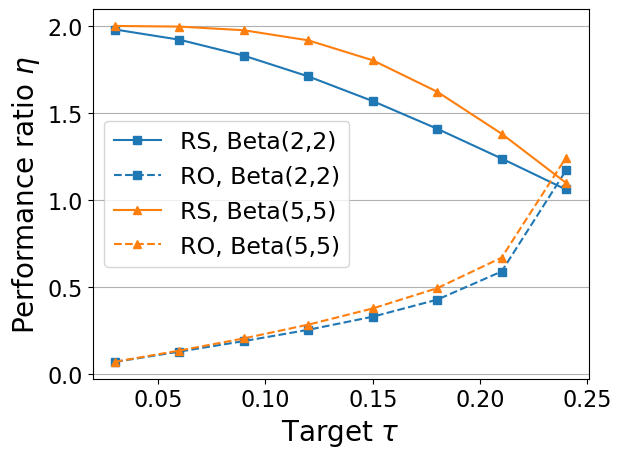} 
        \includegraphics[width=0.33\textwidth]{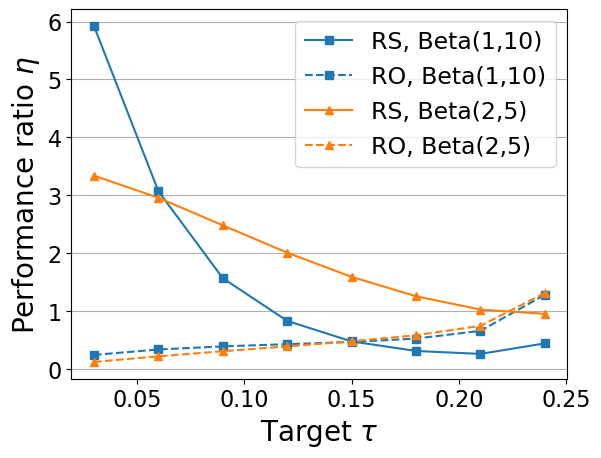}
    }
    {Effectiveness of the PP Mechanism: Prices and Out-of-Sample Performance Ratio under RS vs RO \label{Fig_ComparisonStatistics}}
    {}
    \vspace{-0.2in}
\end{figure}

Interestingly, the out-of-sample performance of the PP mechanism is notably more effective within the RS framework at lower targets, while it performs better in the RO framework at higher targets. In the RO framework, a higher target reduces the radius of the ambiguity set, as it aligns with the optimal worst-case revenue, which decreases with the size of the ambiguity (i.e., $\tau=\Pi_{RO}^*(r)$ decreases as $r$, as shown in Figure \ref{Fig_comparison_r}). This increase in target heightens the chance of excluding the true valuation distribution $P$ from the ambiguity set, thereby reducing the out-of-sample performance of the optimal randomized mechanism derived in the RO framework. As a result, the PP mechanism with a deterministic posted price becomes more effective than the randomized mechanism in this context. These findings indicate that the RO and RS frameworks can complement each other when applying the simple PP mechanism: the RS framework is more advantageous at lower targets, representing higher ambiguity and less information, whereas the RO framework thrives at higher targets, signifying lower ambiguity and sufficient information.


Finally, we compare the \emph{out-of-sample} expected revenues from the optimal RO and RS mechanisms, derived from a uniform reference distribution and evaluated under true valuation distributions following Beta($\alpha$, $\beta$). We remark that the target is set at the maximum worst-case revenue, which is clearly lower than the equivalent target specified in \eqref{equ-equivalence-condition}, namely $\tau = \Pi_{RO}^*(r) \le \tau^{\sf Equiv}(r)$ for every $r\ge 0$. With a regular reference distribution, as outlined in Theorem \ref{Prop_optimalPi} and Proposition \ref{Prop_ContinuousSol}, the RS mechanism's in-sample expected revenue increases with the target, suggesting that the RS mechanism achieves lower in-sample revenue compared to the RO mechanism. However, for out-of-sample revenue, as shown in Figure \ref{Fig_heatmap_RS_RO} and Figure \ref{Fig_RevenueImprovement}, the RS framework is more favorable when the target is high and valuations exhibit positive skewness. Conversely, the RO framework performs better at lower targets or with less skewness. The intuition behind this is that a higher target reduces the RO framework's effectiveness due to the increased risk of the true distribution falling outside the ambiguity set, while higher positive skewness enhances the RS framework's performance by protecting the increased mass of low-valuation buyers. Out-of-sample comparison of the PP mechanisms under the RS and RO mechanisms is illustrated in Figures \ref{Fig_improvement_PP} and \ref{Fig_heatma_PP}. It once again highlights that when the true valuation distribution is skewed toward lower valuations, the RS framework outperforms in seller's revenue.

\begin{figure}[!ht]
\vspace{-0.1in}
    \centering
    \FIGURE
    {
        \subfigure[~$\beta =5$]{\includegraphics[width=0.33\textwidth]{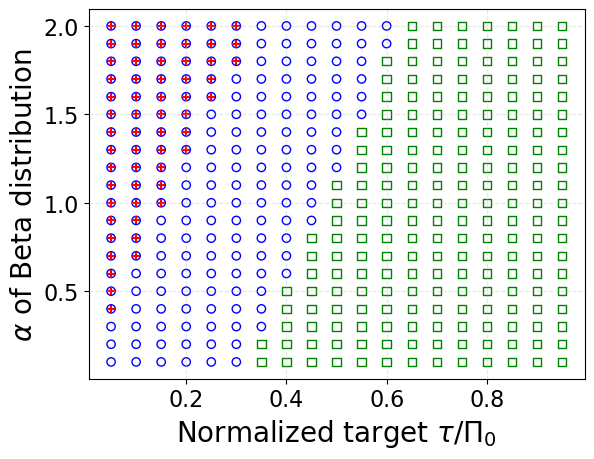}}
        \subfigure[~Low target $\tau/\Pi_0 =0.1$]{\includegraphics[width=0.33\textwidth]{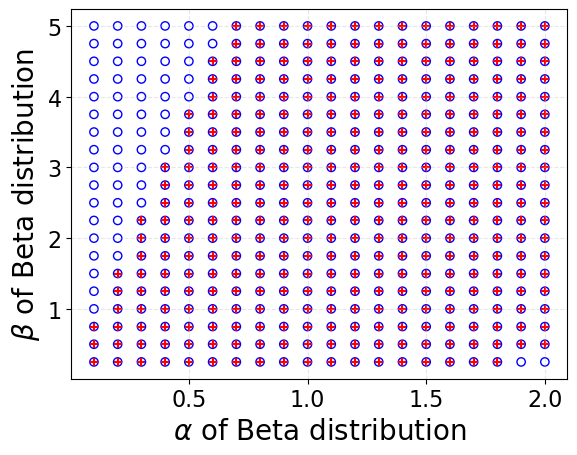}}
        \subfigure[~High target $\tau/\Pi_0 =0.9$]{\includegraphics[width=0.33\textwidth]{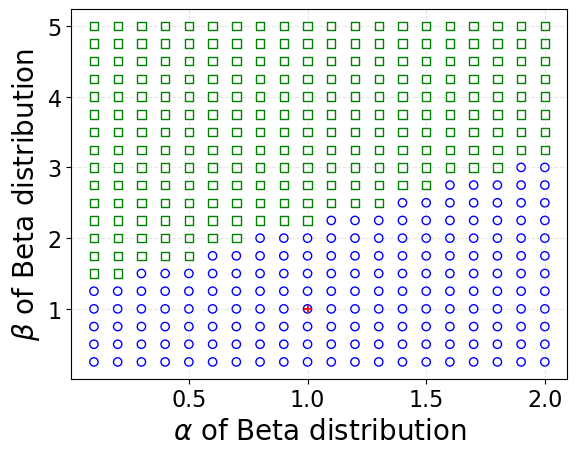}}
    }
    {Seller's Out-of-Sample Preference with Beta$(\alpha,\beta)$ Valuation Distribution: RS vs RO Mechanisms \label{Fig_heatmap_RS_RO}}
    {This figure compares the seller's out-of-sample expected revenue for the RS and RO mechanisms, based on a uniform reference distribution and evaluated using the true Beta$(\alpha,\beta)$ valuation distribution. For any target $\tau$, the ambiguity size $r$ is chosen as the unique solution to $\tau = \Pi_{RO}^*(r)$. Blue circles ``{\color{blue}$\circ$}" represent higher revenue for the RO mechanism, green squares ``{\color{green}$\square$}" for the RS mechanism, and red pluses ``{\color{red}$+$}" indicate that the Beta distribution falls within the ambiguity set of the RO problem.}
\vspace{-0.1in}
\end{figure}

\section{Concluding Remarks}\label{sec-conclude}
This study examines mechanism design under distributional uncertainty through the robust satisficing (RS) framework. Unlike the traditional robust optimization (RO) paradigm, which concentrates on worst-case scenarios within a predefined ambiguity set, the RS approach focuses on minimizing the worst-case shortfall relative to a target revenue level. This shift not only facilitates analytically tractable solutions but also leads to markedly different economic outcomes.

Our findings indicate that the optimal RS mechanism can be implemented as a randomized pricing scheme with a distinctive logarithmic structure. This approach reallocates resources toward lower-valuation buyers, contrasting with RO mechanisms that typically prioritize higher-valuation buyers to mitigate worst-case risks. Consequently, the choice between RS and RO frameworks significantly influences allocation patterns and surplus distribution among buyers.
In comparing the optimal RS mechanism with deterministic posted pricing, we quantify the benefits of randomization. While deterministic pricing is appreciated for its simplicity and transparency, it typically lacks the robustness, flexibility, and precision in allocation that the RS framework and its randomized mechanisms provide. However, in scenarios where targets are low and the valuation distribution exhibits positive skewness, posted pricing can still perform effectively in out-of-sample situations, highlighting a trade-off between ease of implementation and performance effectiveness.

Our research provides valuable managerial insights for designing selling mechanisms under model uncertainties. From a buyer's perspective, the RS mechanism often results in higher surplus compared to the RO mechanism, even across varied environments. While conditions such as increasing hazard rates support this superiority, evidence indicates that RS has distinct benefits in more complex scenarios, including those with irregular distributions.
From a seller's perspective, the comparison of revenue performance between RS and RO is context-dependent. Numerical experiments show that RS can outperform RO in out-of-sample revenue, particularly when valuation distributions are positively skewed or deviate from predefined ambiguity sets. This underscores a trade-off between protection and adaptability: RO offers robustness against known distributions, while RS is more responsive to actual demand conditions, focusing on meeting revenue targets.
For practitioners and policymakers, these insights suggest that the RS framework is particularly well-suited to environments with ample demand data, targeted performance objectives, and diverse customer segments. By avoiding the need for precise ambiguity-set calibration in favor of clear revenue targets, the RS framework provides a practical and transparent strategy for decision-making under uncertainty.

Future research could explore several important avenues: extending the RS framework to multi-buyer or multi-item contexts to enhance market design insights, integrating dynamic learning or data-driven updates for the reference distribution to blend robustness with adaptability, and applying RS principles to other operational areas like inventory management or platform design to expand its relevance and utility.



{
\footnotesize
\theendnotes
}




\bibliographystyle{informs2014}
\bibliography{robust-mechanism}

%
%
%

\newpage
\begin{APPENDICES}

%

\newpage
\setcounter{page}{1}

\setcounter{section}{0}
\renewcommand{\thesection}{\Alph{section}}

\renewcommand{\theequation}{A-\arabic{equation}}
\setcounter{equation}{0}

\renewcommand{\thethm}{A-\arabic{thm}}
\renewcommand{\thelem}{A-\arabic{lem}}
\renewcommand{\theprop}{A-\arabic{prop}}
\renewcommand{\thecor}{A-\arabic{cor}}
\renewcommand{\thedefi}{A-\arabic{defi}}
\renewcommand{\thefigure}{A-\arabic{figure}}
\renewcommand{\thetable}{A-\arabic{table}}
\setcounter{thm}{0}
\setcounter{lem}{0}
\setcounter{prop}{0}
\setcounter{cor}{0}
\setcounter{defi}{0}
\setcounter{figure}{0}
\setcounter{table}{0}

\begin{center}
	{\large Online Appendix \\
		\vspace{.1in}
		\Large \PaperTitle
		\vspace{.1in}
	}
\end{center}



This appendix first provides supplemental examples and numerical analyses in Section \ref{appendix:numerical}, followed by all technical proofs in Section \ref{sec-appendix-proofs}.


\section{Supplemental Examples and Results}\label{appendix:numerical}

\subsection{Characteristics of the Fragility-Adjusted Revenue Functions}

Figure \ref{Fig_fragility_revenue} illustrates that, for any given $k$ in the constraint of the satisficing problem \eqref{Eq_optimizationPi}, the fragility-adjusted revenue function $\rho(\pi, k)$ is convex in $\pi$ with a unique minimum $\pi^*(k)$. And $\rho^*(k)$ is a monotonically increasing function of $k$. 

Figure \ref{fig-fragility-adj-rev-pp-empirical} characterizes a similar pattern of $\rho^{PP*}(k)$ as an increasing function of $k$ in solving the optimal posted pricing mechanism under the empirical reference distribution introduced in Example \ref{exap-empirical-RS-opt}. 

\begin{figure}[!htb]
    \vspace{-0.1in}
    \FIGURE
    {
	\begin{minipage}{0.96\linewidth}
        \includegraphics[width=  0.48\linewidth]{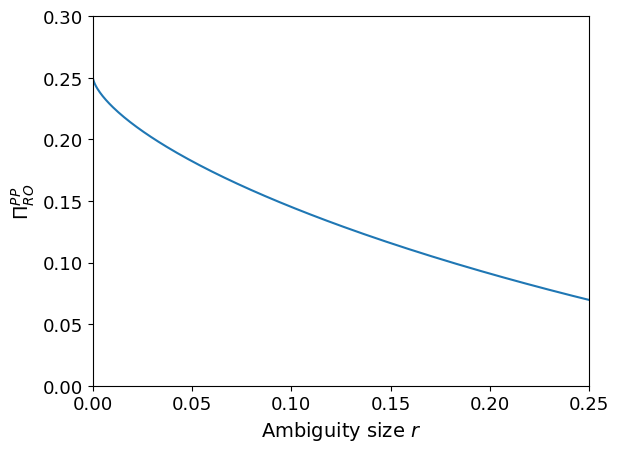}
        \includegraphics[width=  0.48\linewidth]{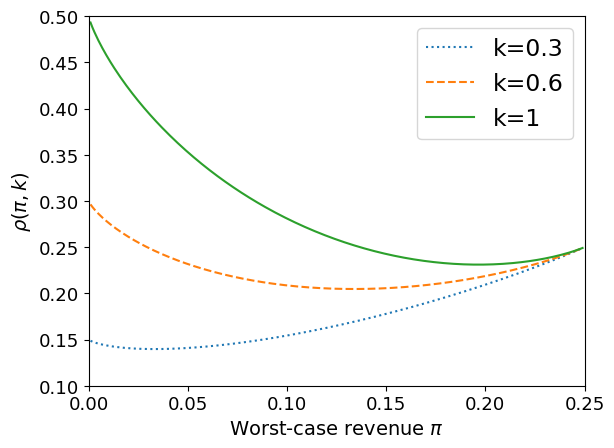}
        \vfill
        \includegraphics[width=  0.48\linewidth]{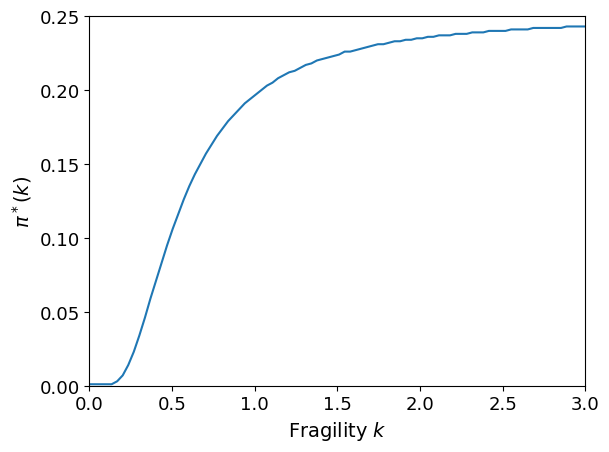}
        \includegraphics[width=  0.48\linewidth]{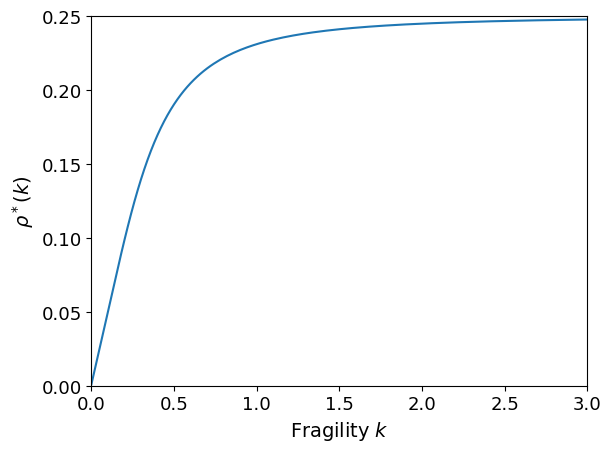}
    \end{minipage}  
    }
    {Fragility-adjusted Revenue and Its Optimal Solution under a Uniform Reference Distribution \label{Fig_fragility_revenue}}
    {}
    \vspace{-0.2in}
\end{figure}

\begin{figure}[!htb]
    \vspace{-0.1in}
    \FIGURE
    {
        \includegraphics[width=0.45\textwidth]{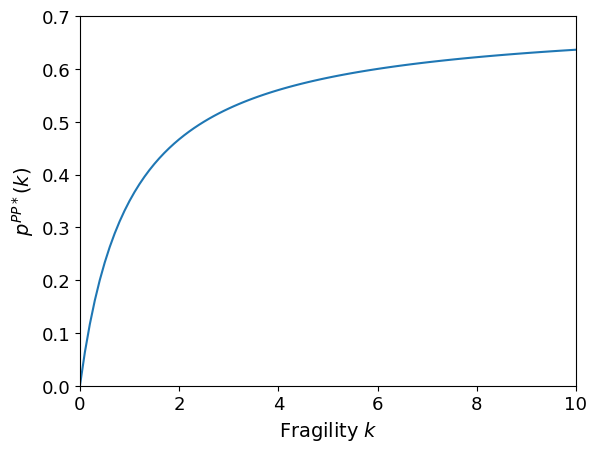} 
        \includegraphics[width=0.45\textwidth]{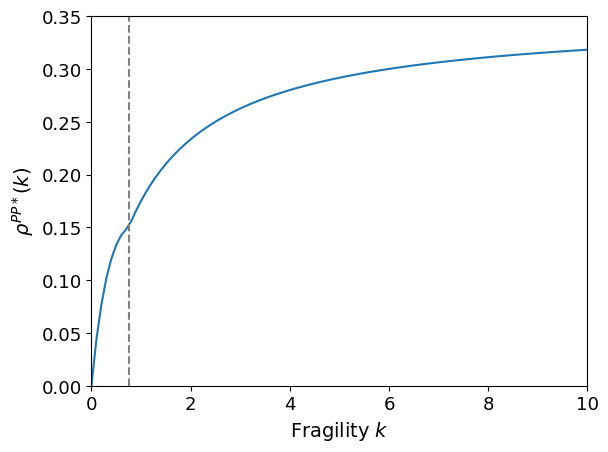}
    }
    {Illustration of $p^{PP*}(k)$ And $\rho^{PP*}(k)$ under an Empirical Reference Distribution $\{(0.3,0.5),(0.7,0.5)\}$ \label{fig-fragility-adj-rev-pp-empirical}}
    {}
    \vspace{-0.2in}
\end{figure}

\subsection{Out-of-Sample Performance under the RS Framework: PP vs Opt Mechanisms}

Figure \ref{Fig_Beta_PDF} shows the probability distribution functions of the beta distributions that are used in the main paper. 
Figure \ref{Fig_optvspp_rs_comparison} further analyzes how key factors influence the seller's preference between PP and optimal mechanisms. Consistent with previous discussions, the posted pricing mechanism generates higher seller revenue when the seller targets lower, less aggressive revenue goals. Conversely, with higher, more aggressive revenue targets, the optimal mechanism proves more beneficial, particularly when the shape parameter $\alpha$ is small or $\beta$ is large. This is because a smaller $\alpha$ or larger $\beta$ skews the valuation distribution positively, increasing the probability mass of low valuations. This substantially reduces the seller's revenue under the posted pricing mechanism by decreasing the likelihood of sales, whereas the randomized pricing strategy of the optimal mechanism is less impacted. These findings suggest the deterministic PP mechanism is most effective for low revenue targets with low variability or positive skewness in valuations but may lead to significant losses compared to the randomized mechanism at high revenue targets with positive skewness.

\begin{figure}[!ht]
\centering
\vspace{-0.1in}
    \caption{The PDF of Beta distribution: \label{Fig_Beta_PDF}}
    \includegraphics[width=  0.4\linewidth]{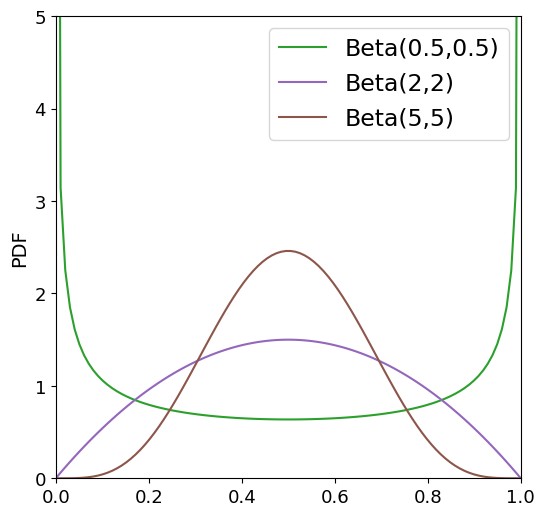}
    \includegraphics[width=  0.4\linewidth]{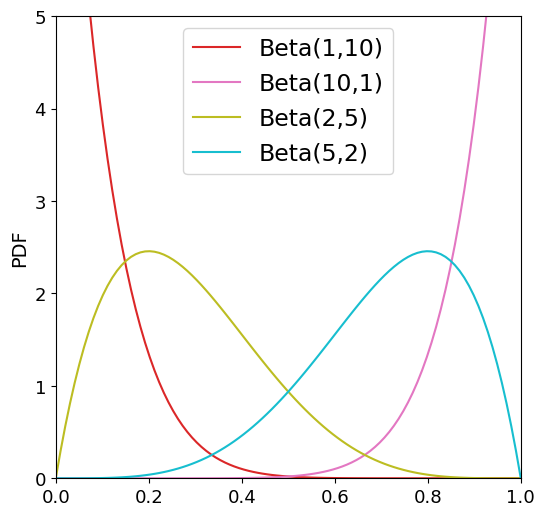}
\vspace{-0.2in}
\end{figure}

\begin{figure}[!ht]
\vspace{-0.1in}
    \centering
    \FIGURE
    {
        \subfigure[~$\beta =5$]{\includegraphics[width=0.33\textwidth]{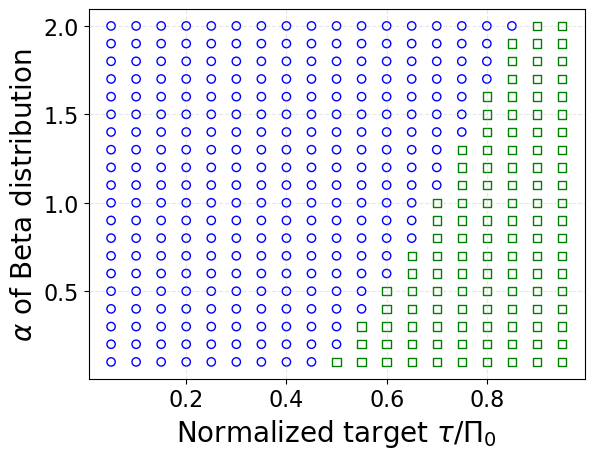}}
        \subfigure[~Low target $\tau/\Pi_0 =0.1$]{\includegraphics[width=0.33\textwidth]{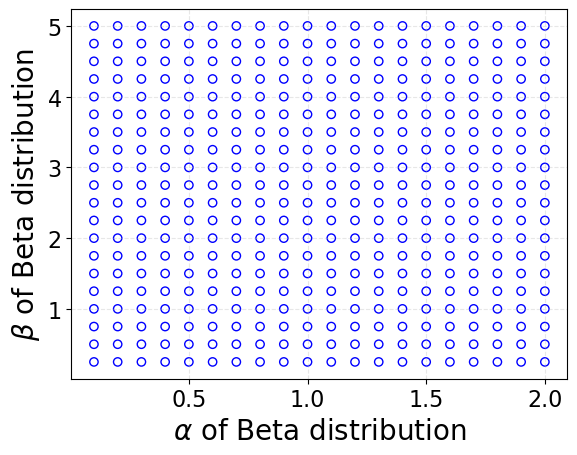}}
        \subfigure[~High target $\tau/\Pi_0 =0.9$]{\includegraphics[width=0.33\textwidth]{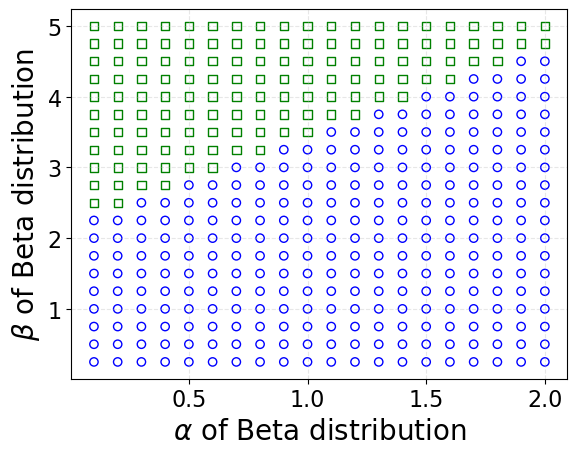}}
    }
    {Seller's Out-of-Sample Preference with Beta$(\alpha,\beta)$ Valuation Distribution under the RS Framework: PP vs Opt Mechanisms \label{Fig_optvspp_rs_comparison}}
    {This figure compares the seller's out-of-sample expected revenue between the PP and Opt mechanisms, derived from a uniform reference distribution and evaluated using the Beta$(\alpha,\beta)$ valuation distribution. Blue circles ``{\color{blue}$\circ$}" indicate higher revenue for the PP mechanism, while green squares ``{\color{green}$\square$}" indicate higher revenue for the Opt mechanism.  
    }
\end{figure}

\subsection{Supplemental Numerical Analysis: Ambiguity Radius and Conditional Payment}

Figure~\ref{Fig_comparison_r} compares the mechanisms induced by RS and RO under a uniform reference distribution. Panel (a) plots the ambiguity radius required to achieve a given target $\tau$ under the two frameworks. For the same target level, the RS framework corresponds to a larger ambiguity radius, indicating that it achieves the target performance while accommodating a wider range of distributions. 
Panel (b) compares the payment conditional on winning across valuations. The RS mechanism charges uniformly lower prices for all buyers who receive the good with positive probability, highlighting its less aggressive extraction relative to the RO mechanism.

\begin{figure}[!ht]
\centering
\vspace{-0.1in}
    \caption{Comparison between the RS and RO Frameworks with a Uniform Reference Distribution \label{Fig_comparison_r}}
    \subfigure[~Radius vs Target]{\includegraphics[width=0.45\linewidth]{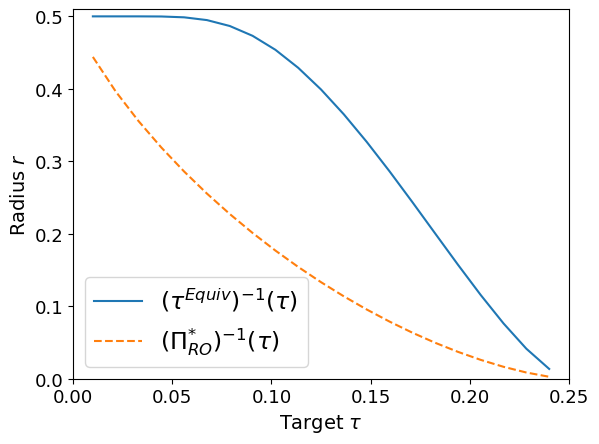}}
    \subfigure[~Payment conditional on winning]{\includegraphics[width=  0.45\linewidth]{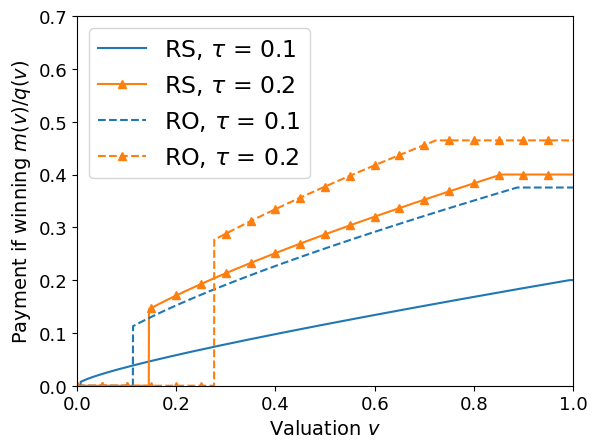}}
\vspace{-0.1in}
\end{figure}

\subsection{Optimal RS vs RO Mechanisms with a Regular, Non-MHR Reference Distribution}

Figure \ref{Fig_opt_comparison_pareto} shows the comparison of the RS and RO mechanisms with a truncated Pareto distribution that has CDF $\frac87\left(1-\left(1+x\right)^{-3}\right)$ on $[0,1]$. The distribution has an increasing virtual value function $\frac{(1+x)^4+16x-8}{24}$, confirming its regularity, and a non-monotone hazard rate $\frac{24}{8(1+x)-(1+x)^4}$. The figure reveals that the single-crossing property of the expected payment still holds, and the RS mechanism dominates in terms of the buyer surplus.

\begin{figure}[!htb]
    \vspace{-0.1in}
    \FIGURE
    {
        \includegraphics[width=0.33\textwidth]{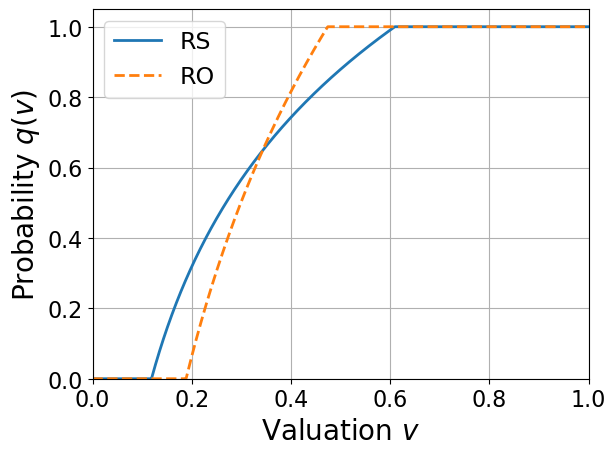}
        \includegraphics[width=0.33\textwidth]{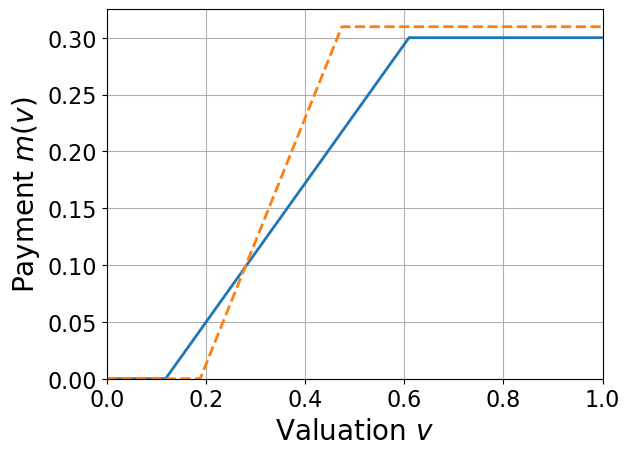} 
        \includegraphics[width=0.33\textwidth]{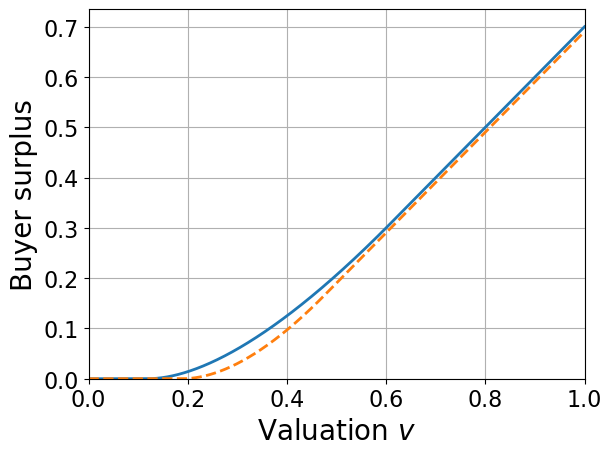}
    }
    {RS vs RO Mechanisms under a Pareto Reference Distribution \label{Fig_opt_comparison_pareto}}
    {The reference is a Pareto distribution $P_0(x) = \frac87\left(1-\left(1+x\right)^{-3}\right)$ for all $x\in[0,1]$, which is regular but with a non-monotone hazard rate. The revenue target is $\tau = 0.1$.}
    \vspace{-0.2in}
\end{figure}


\subsection{Optimal RS vs RO Mechanisms with an Irregular Reference Distribution}

Figure \ref{Fig_opt_comparison_irregular} compares the RS and RO mechanisms for the irregular reference distribution defined as a mixture $0.9\cdot\text{Beta}(10,2)+0.1\cdot\text{Beta}(2,10)$, with a revenue target $\tau=0.24$. The figure comprises the allocation probability $q(v)$, the expected payment $m(v)$ and the buyer surplus $v q(v)-m(v)$ as a function of valuation $v$. Across all panels, the RS mechanism (solid line) tends to offer higher allocation probabilities and payments for low-valuation buyers, while the RO mechanism (dashed line) favors high-valuation buyers. The buyer surplus under the RS mechanism remains dominant, demonstrating the robustness of the RS approach even when the reference distribution is irregular. 

\begin{figure}[!htb]
    \vspace{-0.1in}
    \FIGURE
    {
        \includegraphics[width=0.33\textwidth]{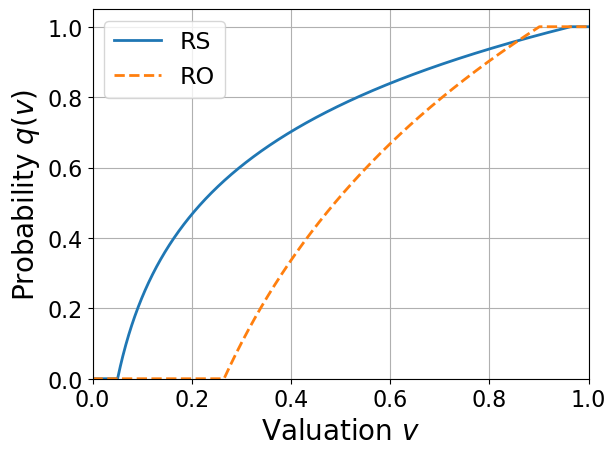}
        \includegraphics[width=0.33\textwidth]{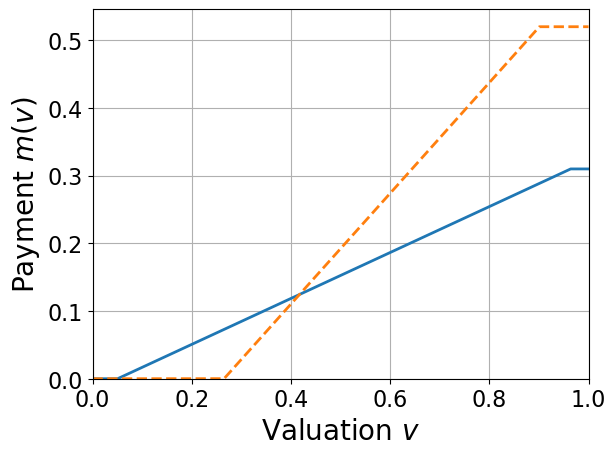} 
        \includegraphics[width=0.33\textwidth]{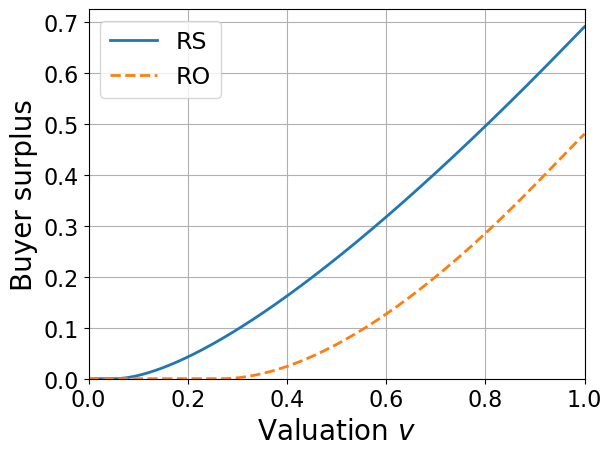}
    }
    {RS vs RO Mechanisms under a Mixture of Beta Distributions \label{Fig_opt_comparison_irregular}}
    {The reference distribution is $P_0 = 0.9\cdot\text{~Beta}(10,2)+0.1\cdot\text{Beta}(2,10)$, which is irregular, and $\tau = 0.24$.}
\end{figure}

\subsection{Out-of-Sample Performance of the Opt Mechanisms: RS vs RO Framework}

Figure \ref{Fig_RevenueImprovement} compares the seller's out-of-sample expected revenue for the optimal mechanisms under the RS and RO frameworks, based on a uniform reference distribution and evaluated using the true Beta$(\alpha,\beta)$ valuation distribution. For any target $\tau$, the ambiguity size $r$ is chosen as the unique solution to $\tau = \Pi_{RO}^*(r)$.

\begin{figure}[!ht]
\vspace{-0.1in}
    \centering
    \FIGURE
    {
        \includegraphics[width=  0.45\linewidth]{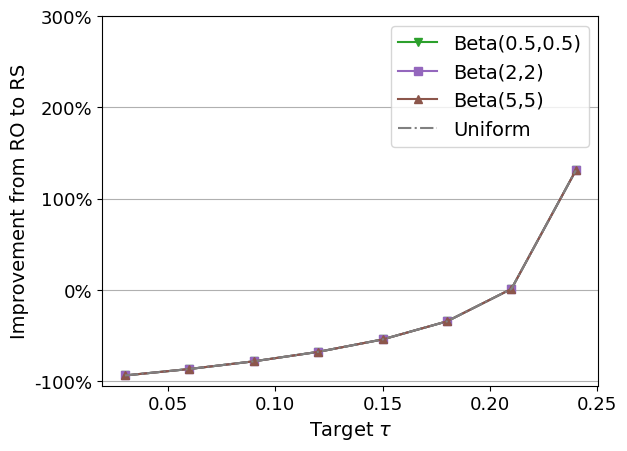}
        \includegraphics[width=  0.44\linewidth]{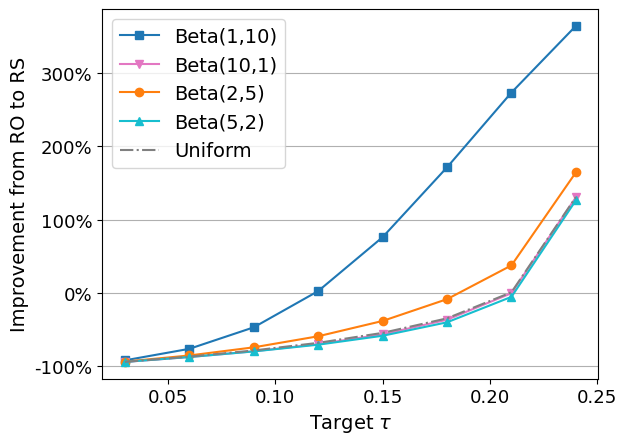}
    }
    {Out-of-Sample Revenue Improvement from RO to RS with Uniform Reference Distribution\label{Fig_RevenueImprovement} 
    }
    {The out-of-sample improvement(\%) from RO to RS is defined as: $\frac{\mathbb{E}_P[m_{RS}^*(\tilde{v})]-\mathbb{E}_P[m_{RO}^*(\tilde{v})]}{\mathbb{E}_P[m_{RO}^*(\tilde{v})]} \times 100$.}
\vspace{-0.2in}
\end{figure}

\subsection{Out-of-Sample Performance of the PP Mechanisms: RS vs RO Frameworks}

Figure \ref{Fig_improvement_PP} and \ref{Fig_heatma_PP} compare the seller's out-of-sample expected revenue for PP mechanisms under the RS and RO frameworks, based on a uniform reference distribution and evaluated using the true Beta$(\alpha,\beta)$ valuation distribution. For any target $\tau$, the ambiguity size $r$ is chosen as the unique solution to $\tau = \Pi_{RO}^*(r)$.

In Figure \ref{Fig_improvement_PP}, the improvement is measured as the percentage increase in expected revenue when switching from the RO-based PP mechanism to the RS-based PP mechanism.
The left panel considers symmetric distributions, while the right panel considers asymmetric Beta distributions. The relative performance of RS depends on the location of demand mass and the revenue target. When the distribution is skewed toward lower valuations (e.g., Beta$(10,1)$ or Beta$(5,2)$), the RS-based PP mechanism yields positive improvements when the revenue target is high. In contrast, when the distribution is skewed toward higher valuations (e.g., Beta$(1,10)$), the RO-based mechanism can outperform RS. 

\begin{figure}[!ht]
\vspace{-0.1in}
    \centering
    \FIGURE
    {
        \includegraphics[width=  0.48\linewidth]{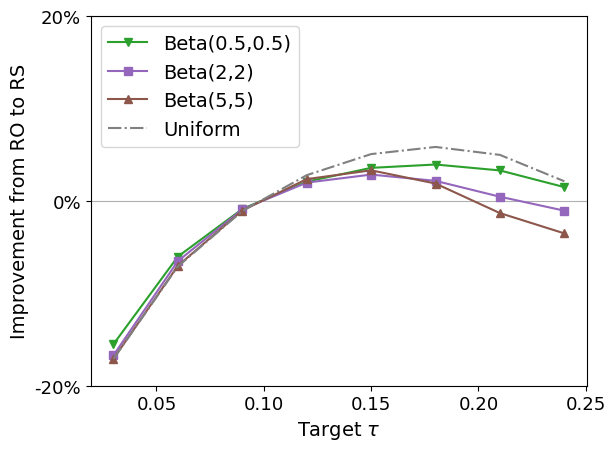}
        \includegraphics[width=  0.47\linewidth]{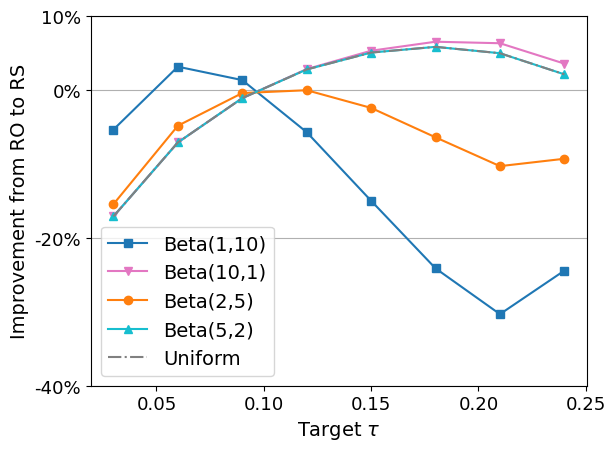}
    }
    {Out-of-Sample Revenue Improvement of PP Mechanism from RO to RS \label{Fig_improvement_PP}
    }
    {The PP mechanism is derived using a uniform reference distribution, and the out-of-sample improvement(\%) of PP mechanism from RO to RS is defined as: $\frac{\mathbb{E}_{P}[p_{RS}^{PP}\mathbbm{1}(\tilde{v}\geq p_{RS}^{PP})]-\mathbb{E}_{P}[p_{RO}^{PP}\mathbbm{1}(\tilde{v}\geq p_{RO}^{PP})]}{\mathbb{E}_{P}[p_{RO}^{PP}\mathbbm{1}(\tilde{v}\geq p_{RO}^{PP})]} \times 100$.}
\vspace{-0.2in}
\end{figure}

Figure \ref{Fig_heatma_PP} shows that the relative performance of RS and RO depends jointly on the target level and the location of the true demand distribution. RS tends to outperform RO when the demand is skewed toward lower valuations or when the target is high, while RO remains advantageous when demand is concentrated at higher valuations or when the target is low.

\begin{figure}[!ht]
\vspace{-0.1in}
    \centering
    \FIGURE
    {
        \subfigure[~$\beta =5$]{\includegraphics[width=0.33\textwidth]{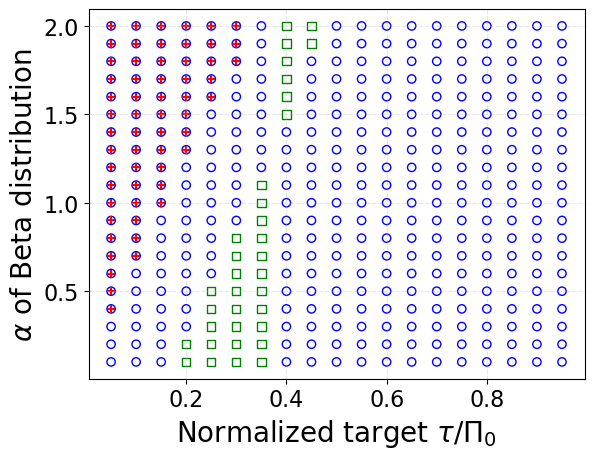}}
        \subfigure[~Low target $\tau/\Pi_0 =0.1$]{\includegraphics[width=0.33\textwidth]{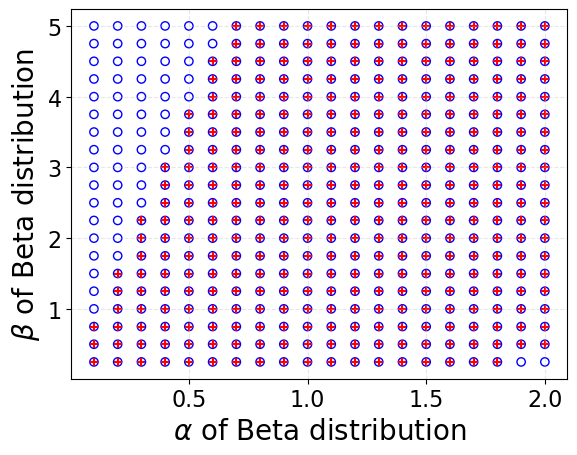}}
        \subfigure[~High target $\tau/\Pi_0 =0.9$]{\includegraphics[width=0.33\textwidth]{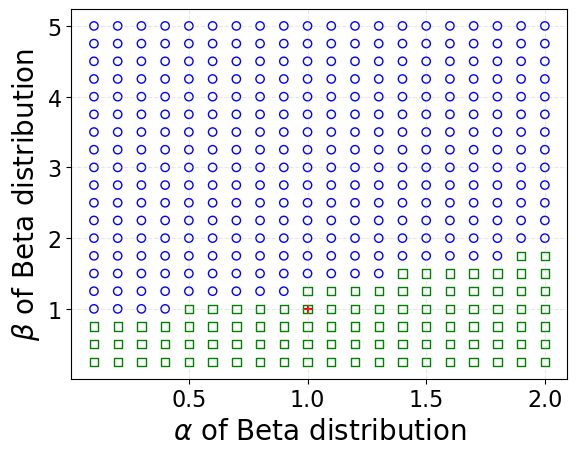}}
    }
    {Seller's Out-of-Sample Preference: PP Mechanisms under RS vs RO Frameworks \label{Fig_heatma_PP}}
    {Blue circles ``{\color{blue}$\circ$}" represent higher revenue for the RO mechanism, green squares ``{\color{green}$\square$}" for the RS mechanism, and red pluses ``{\color{red}$+$}" indicate that the Beta distribution falls within the ambiguity set of the RO problem.}
\vspace{-0.1in}
\end{figure}

\newpage
\section{Technical Proof}\label{sec-appendix-proofs}

\subsection{Proof of Proposition \ref{Lm_pricing}}
\begin{proof}
First, we establish \eqref{Prob_pricing} by applying the functional version of von Neumann's minimax theorem in Lemma \ref{Lm_minimax} (see \citealt{borwein1986fan}). Notice that, with change of variables, we have $\int_0^1 \int_0^v x\mathrm{d}q(x) \mathrm{d}P(v) = \int_0^1 x \int_x^1 \mathrm{d}P(v)\mathrm{d}q(x)$. Thus, the objective function on the right-hand side of expression \eqref{Eq_constraint} is linear in $q$. On the other hand, one can check that the Wasserstein-1 distance $d(P,P_0)$ is convex in $P$. Together with the fact that the double integral is linear in $P$, we have that the objective function is convex in $P$. In addition, it is upper semi-continuous for each fixed $P\in \mathcal{P}$ (see \citealt{carrasco2018optimal}). With Lemma \ref{Lm_minimax}, we now obtain
\beq
\begin{aligned}
\sup_{(q,m)\in\mathcal{M}}\inf_{P \in \mathcal{P}}\mathbb{E}_P[ \int_0^{\tilde{v}} x\mathrm{d}q(x)]+ k d(P,P_0) & = \max_{(q,m)\in\mathcal{M}}\inf_{P \in \mathcal{P}}\mathbb{E}_P[ \int_0^{\tilde{v}} x\mathrm{d}q(x)]+ k d(P,P_0) \\
& = \inf_{P \in \mathcal{P}}\max_{(q,m)\in\mathcal{M}}\mathbb{E}_P[ \int_0^{\tilde{v}} x\mathrm{d}q(x)]+ k d(P,P_0)
\end{aligned}
\eeq

For any given $P\in \mathcal{P}$, \cite{manelli2007multidimensional} showed that the solutions for problem $\max_{(q,m)\in\mathcal{M}}\mathbb{E}_P[ \int_0^{\tilde{v}} x\mathrm{d}q(x)]$ are posted pricing mechanisms, implying $\inf_{P \in \mathcal{P}}\max_{(q,m)\in\mathcal{M}}\mathbb{E}_P[ \int_0^{\tilde{v}} x\mathrm{d}q(x)]+ k d(P,P_0) = \inf_{P \in \mathcal{P}}\max_{x\in[0,1]}x\bar{P}_{-}(x)+ k d(P,P_0)$. This completes the proof for \eqref{Prob_pricing}.

We can rank all the feasible distributions $P\in \mathcal{P}$ according to $d(P,P_0)$. Thus, $\inf_{P \in \mathcal{P}}\max_{x\in[0,1]}x\bar{P}_{-}(x)+ k d(P,P_0)=\inf_{r\geq 0} \inf_{P\in \mathcal{P} : d(P,P_0)=r}\max_{x\in[0,1]}x\bar{P}_{-}(x)+ k r$. Next, we show that for any given $r$, $\inf_{P\in \mathcal{P} : d(P,P_0)=r}\max_{x\in[0,1]}x\bar{P}_{-}(x)= \inf_{P\in \mathcal{W}(P_0,r)}\max_{x\in[0,1]}x\bar{P}_{-}(x)$.
When $r\le \mu_0$, it's sufficient to show that for a given ambiguity size $r$, the optimum of the minimax pricing problem is achieved when the worst-case distribution lies on the boundary of the ambiguity set $\mathcal{W}(P_0,r)$. This has been verified in Lemma \ref{Lm_chen}. When $r> \mu_0$, $\inf_{P: d(P,P_0)=r} \max_{x\in [0,1]} x\bar{P}_{-}(x) = 0 = \inf_{P \in \mathcal{W}(P_0,r)} \max_{x\in [0,1]} x\bar{P}_{-}(x)$ and this completes the proof for \eqref{Eq_robustpricing}.

\begin{lem}\label{Lm_minimax}
(von Neumann-Fan minimax theorem; \citealt{borwein1986fan}) Let $g(x,y)$ be a real-valued function defined on $\mathcal{X}\times \mathcal{Y}$, where $\mathcal{X}$ and $\mathcal{Y}$ are nonempty convex sets and $\mathcal{X}$ is compact. Suppose $g(\cdot,y)$ is concave and upper semi-continuous for each fixed $y\in \mathcal{Y}$ and $g(x,\cdot)$ is convex for each fixed $x\in \mathcal{X}$. Then
$$\sup_{x\in \mathcal{X}} \inf_{y \in \mathcal{Y}}g(x,y) = \max_{x\in \mathcal{X}} \inf_{y \in \mathcal{Y}}g(x,y) = \inf_{y \in \mathcal{Y}}\max_{x\in \mathcal{X}} g(x,y) = \inf_{y \in \mathcal{Y}}\sup_{x\in \mathcal{X}} g(x,y)$$
\end{lem}

Next, we establish \eqref{Eq_robustpricing}. Clearly, 
\beq
\inf_{P \in \mathcal{P}}\left[\max_{x\in[0,1]}x\bar{P}_{-}(x)+ k d(P,P_0)\right] 
 & = &  \inf_{r\ge 0} \left\{ \inf_{P\in \mathcal{P}: d(P,P_0)=r} \max_{x\in [0,1]} x\bar{P}_{-}(x) + kr \right\} \\
 & \ge & \inf_{r \geq 0} \left\{\inf_{P \in \mathcal{W}(P_0,r)} \max_{x\in [0,1]} x\bar{P}_{-}(x) + kr \right\},
 \eeq since $\{P\in \mathcal{P}: d(P,P_0)=r\} \subseteq \mathcal{W}(P_0,r)$. On the other hand, for any $r\ge 0$, $\mathcal{W}(P_0,r) \subseteq \mathcal{P}$, thus, 
 \beq
 \inf_{P \in \mathcal{P}}\left[\max_{x\in[0,1]}x\bar{P}_{-}(x)+ k d(P,P_0)\right] \leq \inf_{P \in \mathcal{W}(P_0,r)} \max_{x\in [0,1]} x\bar{P}_{-}(x) + kr, \quad \forall r\ge 0.
 \eeq  \qed
\end{proof}

\subsection{Proof of Proposition \ref{Prop_robustpricing}}
\begin{proof}
We show that $\min_{r\in[0,\mu_0]} \left( \Pi_{RO}^{PP}(r) + kr \right) = \min_{\pi \in [0,\Pi_0]} \left[ \pi +  k\int_0^1 [\bar{P}_0-\pi/x]^+ \mathrm{d}x \right]$ with $\Pi_0:=\sup_{x\in [0,1]} x\bar{P}_0(x)$. This is because, by \eqref{equ-rev-ro-pp}, $\Pi_{RO}^{PP}(r)$ is monotonically decreasing in $r$ for any $r<\mu_0$. The one-on-one correspondence between $\pi$ and $r$ establishes the equivalence between these two optimizations by transforming the variables from the Wasserstein distance $r$ to the worst-case revenue $\pi$.
\qed
\end{proof}

\subsection{Proof of Theorem \ref{Prop_optimalPi}}
\begin{proof}
One can check that, in general, $u_j(\pi)$ is a non-decreasing function of $\pi$ and $w_j(\pi)$ a non-increasing function of $\pi$. The first-order derivative of $\rho(\pi,k)$ over $\pi$ gives $\frac{\partial}{\partial\pi} \rho(\pi,k) =\frac{\partial}{\partial\pi} \{k\sum_{j \in [J(\pi)]} \int_{u_j(\pi)}^{w_j(\pi)} (\bar{P}_0(x)-\pi/x) \mathrm{d}x + \pi\} =  k\sum_{j \in [J(\pi)]} \int_{u_j(\pi)}^{w_j(\pi)} (-1/x) \mathrm{d}x +1 = -k\sum_{j \in [J(\pi)]} \ln \frac{w_j(\pi)}{u_j(\pi)}+1$. As $\pi$ increases, $\frac{w_j(\pi)}{u_j(\pi)}$ decreases, making $\rho(\pi,k)$ a convex function of $\pi$.

The minimum $\rho^*(k)$ is achieved with $\pi^*(k)$ being the unique solution of $\pi$ to $k\sum_{j \in [J(\pi)]} \ln \frac{w_j(\pi)}{u_j(\pi)}= 1$. We observe that as $k$ increases, $\pi^*(k)$ monotonically increases. Thus, $\rho^*(k) = k\int_0^1 [\bar{P}_0(x)-\pi^*(k)/x]^+ \mathrm{d}x +\pi^*(k)=k \sum_{j \in [J(\pi^*(k))]} \int_{u_j(\pi^*(k))}^{w_j(\pi^*(k))} \bar{P}_0(x)\mathrm{d}x -\pi^*(k) + \pi^*(k) =  k \sum_{j \in [J(\pi^*(k))]} \int_{u_j(\pi^*(k))}^{w_j(\pi^*(k))} \bar{P}_0(x)\mathrm{d}x$, which is also equal to $\frac{\sum_{j \in [J(\pi^*(k))]} \int_{u_j(\pi^*(k))}^{w_j(\pi^*(k))} \bar{P}_0(x)\mathrm{d}x}{\sum_{j \in [J(\pi^*(k))]} \int_{u_j(\pi^*(k))}^{w_j(\pi^*(k))} \frac{1}{x}\mathrm{d}x}$.

For the monotonicity analysis of $\rho^*(k)$, it would be sufficient to analyze the property of $\frac{\sum_{j \in [J(\pi^*)]} \int_{u_j(\pi^*)}^{w_j(\pi^*)} \bar{P}_0(x)\mathrm{d}x}{\sum_{j \in [J(\pi^*)]} \int_{u_j(\pi^*)}^{w_j(\pi^*)} \frac{1}{x}\mathrm{d}x}$ w.r.t. $\pi^*$ because $\pi^*(k)$ is a monotonically increasing function of $k$. For simplicity of notations, let's define $A:=\int_{u_j(\pi^*)}^{w_j(\pi^*)} \bar{P}_0(x)\mathrm{d}x$ and $B:=\int_{u_j(\pi^*)}^{w_j(\pi^*)} \frac{1}{x}\mathrm{d}x$. By the definition of $u_j$ and $v_j$, we have $A>B$ and $\bar{P}_0(w_j) = \pi^*/w_j$, $\bar{P}_0(u_j) = \pi^*/u_j$. Then, $\frac{\mathrm{d}}{\mathrm{d} \pi^*} \frac{\sum_{j \in [J(\pi^*)]} \int_{u_j(\pi^*)}^{w_j(\pi^*)} \bar{P}_0(x)\mathrm{d}x}{\sum_{j \in [J(\pi^*)]} \int_{u_j(\pi^*)}^{w_j(\pi^*)} \frac{1}{x}\mathrm{d}x}= \frac{1}{B^2}\big[[\bar{P}_0(w_j(\pi^*))w_j'(\pi^*)-\bar{P}_0(u_j(\pi^*))u_j'(\pi^*)]B-(\frac{w_j'(\pi^*)}{w_j(\pi^*)}-\frac{u_j'(\pi^*)}{u_j(\pi^*)})A\big] = \frac{1}{B^2}[\frac{w_j'(\pi^*)}{w_j(\pi^*)}-\frac{u_j'(\pi^*)}{u_j(\pi^*)}](B-A)$. Since $w_j'(\pi^*)\leq 0$ and $u_j'(\pi^*)\geq 0$, together with $B<A$, we conclude that $\rho^*(k)$ is a monotonically increasing function of $\pi^*$, which in turn is an increasing function of $k$. \qed
\end{proof}

\subsection{Proof of Theorem \ref{Thm_optimalmechanism}}
\begin{proof}
The optimal satisficing mechanism should be the same as the optimal mechanism to the minimax monopoly pricing problem $\inf_{P \in \mathcal{W}(P_0,r)} \max_{x\in [0,1]} x\bar{P}_{-}(x)$ with $r$ satisfies that $\Pi_{RO}^{PP}(r) = \pi^*$. \cite{chen2024screening} gives the optimal mechanism to the minimax monopoly pricing problem as 
\beq
\left(q_{RO}^*\left(v\right),m_{RO}^*\left(v\right)\right) = \left\{\begin{array}{ll}
    \alpha^* \cdot\left(\sum_{i\in[j-1]} \ln \left(\frac{w_i}{u_i}\right) + \ln (\frac{v}{u_j}), \; \sum_{i\in[j-1]}(w_i-u_i) +(v-u_j) \right), & v \in [u_j,w_j)\\
    \alpha^* \cdot \left(\sum_{i\in[j]} \ln \left(\frac{w_i}{u_i}\right),\; \sum_{i\in[j]}(w_i-u_i) \right), & v \in[w_j,u_{j+1}),   
\end{array}\right.
\eeq
where $\alpha^* = 1/\left(\sum_{i\in [J]\ln(w_i/u_i)}\right)$. Since $k_{RS}^* = 1/\left(\sum_{i\in [J]}\ln(w_i^*/u_i^*)\right)$, this completes the proof for \eqref{equ-opt-rs-mechanism}.\qed
\end{proof}

\subsection{Proof of Proposition \ref{Prop_ContinuousPP}}

\begin{proof}
If we restrict the allocation and payment rule to $q(v) = \mathbbm{1}(v\geq p)$ and $m(v)=p \mathbbm{1}(v\geq p)$, then under a continuous reference distribution $P_0$, the satisficing constraint is equivalent to:
\beq
\begin{aligned}
&\tau - \mathbb{E}_P[m(\tilde{v})] \leq kd(P,P_0),\quad \forall P \in \mathcal{P} \\
\Longleftrightarrow \quad
&\tau \leq \inf_{P\in \mathcal{P}} \mathbb{E}_P[m(\tilde{v})]+kd(P, \mathrm{d}P_0(x)),\\
\Longleftrightarrow \quad
&\tau \leq \int_0^1 \inf_{v \in [0,1]} m(v)+k|v-x| \mathrm{d}P_0(x),\\
\Longleftrightarrow \quad
&\tau \leq \int_0^1 \inf_{v \in [0,1]} p \mathbbm{1}(v\geq p)+k|v-x| \mathrm{d}P_0(x),\\
\Longleftrightarrow \quad
&\tau \leq \int_p^1 \min(p, k(x-p)) \mathrm{d}P_0(x),\\
\Longleftrightarrow \quad
&\tau \leq \int_p^{(1+\frac{1}{k})p} k(x-p) \mathrm{d}P_0(x) +\int_{(1+\frac{1}{k})p} ^1 p \mathrm{d}P_0(x).
\end{aligned}
\eeq
Since we only need to find one price $p$ such that the constraint holds, the last inequality is equivalent to $\tau \leq \max_{p\in[0,1]} \int_p^{(1+\frac{1}{k})p} k(x-p) \mathrm{d}P_0(x) +\int_{(1+\frac{1}{k})p} ^1 p \mathrm{d}P_0(x)$
\qed
\end{proof}

\subsection{Proof of Proposition \ref{Prop_ContinuousSol}}
\begin{proof}
First, we prove that when $P_0$ is regular, then there can be at most two intersection points between $\bar{P}_0(x)$ and the level curve $c/x$. Note that $\bar{P}_0(x) = c/x$ is equivalent to $x\left(1-P_0\left(x\right)\right)=c$. Thus, if $P_0$ is regular, which means that the virtual value function $x-\frac{1-P_0(x)}{\mathrm{d}P_0(x)/\mathrm{d}x}$ is monotonically increasing, then the derivative of $x(1-P_0(x))$ (i.e., $1-P_0(x)-x\frac{\mathrm{d}P_0(x)}{\mathrm{d}x}$) crosses the x-axis for at most once. And if the derivative crosses the x-axis, it must cross the axis from above to below the axis as $x$ increases. Thus, $x(1-P_0(x))$ increases first and then decreases. For any given constant $c$, there can be at most two intersection points to the equation $x\left(1-P_0\left(x\right)\right)=c$.

Next, we show that $p^{PP*}(k)=u$. 
$\rho^{PP}(p,k)$ can be simplified to be $p\left(1-P_0\left(\frac{k+1}{k}p\right)\right)-kp\left(P_0(\frac{k+1}{k}p)-P_0(p)\right)+\int_p^{\frac{k+1}{k}p} kx \mathrm{d}P_0(x)$.
The partial derivative $\frac{\partial}{\partial p} \rho(p,k) = (k+1)p \mathrm{d}P_0(\frac{k+1}{k}p)\cdot \frac{k+1}{k}-kp\mathrm{d}P_0(p)+ 1 - P_0(\frac{k+1}{k}p) - \frac{k+1}{k}p\mathrm{d}P_0(\frac{k+1}{k}p) - k(P_0(\frac{k+1}{k}p)-P_0(p))- kp(\frac{k+1}{k}\mathrm{d}P_0(\frac{k+1}{k}p)-\mathrm{d}P_0(p)) = 1-P_0(\frac{k+1}{k}p)-k(P_0(\frac{k+1}{k}p)-P_0(p))$. Thus, $p^{PP*}(k)$ satisfies the first-order condition, i.e., it should be the solution to the equation $1-P_0(\frac{k+1}{k}p) = k(P_0(\frac{k+1}{k}p)-P_0(p))$, or equivalently, $\bar{P}_0(\frac{k+1}{k}p) = k (\bar{P}_0(p)-\bar{P}_0(\frac{k+1}{k}p))$. It is further equivalent to $\frac{k+1}{k}p \bar{P}_0(\frac{k+1}{k}p) = p \bar{P}_0(p)$.  
Thus, $p^{PP*}(k) = u$.
\qed
\end{proof}

\subsection{Proof of Proposition \ref{Prop_DiscretePP}}
\begin{proof}
When the empirical distribution $P_0$ is constructed by $N$ observations, the last RS constraint is equivalent to:
\beq
\begin{aligned}
    &\tau - \mathbb{E}_P[m(\tilde{v})] \leq kd(P,P_0),\quad \forall P \in \mathcal{P} \\
    \Longleftrightarrow \quad
    &\tau - \mathbb{E}_P[m(\tilde{v})]\leq k \cdot \mathop{\textup{inf}} \limits_{\gamma \in \Pi(P,P_0)}  \sum_{n=1}^N \alpha_n \mathbb{E}_{\tilde{v}\sim P_n}[|\tilde{v}-\hat{v}_n|],\quad \forall P \in \mathcal{P}\\
    \Longleftrightarrow \quad
    &\tau \leq \mathop{\textup{inf}} \limits_{\gamma \in \Pi(P,P_0)} \sum_{n=1}^N \alpha_n \left(\mathbb{E}_{P_n}[m(\tilde{v})]+ k \mathbb{E}_{ P_n}[|\tilde{v}-\hat{v}_n|]\right),\quad \forall P \in \mathcal{P} \\
    \Longleftrightarrow \quad
    &\tau \leq \mathop{\textup{inf}} \limits_{\gamma \in \Pi(P,P_0)} \sum_{n=1}^N \alpha_n \int_0^1 m(v)+k|v-\hat{v}_n|\mathrm{d}P_n(v),\quad \forall P \in \mathcal{P} \\
    \Longleftrightarrow \quad
    &\tau \leq \inf_{P \in \mathcal{P}} \mathop{\textup{inf}} \limits_{\gamma \in \Pi(P,P_0)} \sum_{n=1}^N \alpha_n \int_0^1 m(v)+k|v-\hat{v}_n|\mathrm{d}P_n(v)\\
    \Longleftrightarrow \quad
    &\tau \leq \sum_{n=1}^N \alpha_n \inf_{P \in \mathcal{P}}\int_0^1 m(v)+k|v-\hat{v}_n|\mathrm{d}P(v)\\
    \Longleftrightarrow \quad
    &\tau \leq \sum_{n=1}^N \alpha_n \inf_{v \in [0,1]} \left(m(v)+k|v-\hat{v}_n|\right),
\end{aligned}
\eeq
where $P_n(v):=\gamma(v|\tilde{u}=\hat{v}_n)$. \qed
\end{proof}

\subsection{Proof of Proposition \ref{Prop_DiscreteSol}}
\begin{proof}
Notice that given $p$, $k$ and $n$,
\beq
\rho_n(p,k) = \inf_{v \in [0,1]} p\mathbbm{1}(v\geq p)+k\vert \hat{v}_n-v\vert =
    \min (p,k(\hat{v}_n-p)^+). 
\eeq
When $\hat{v}_n\ge p$, $k(\hat{v}_n-p) \le p$ if and only if $p \ge  \frac{k}{k+1}\hat{v}_n$. Thus, we can reformulate the above infimum as
\beq
\rho_n(p,k) = \left\{\begin{array}{ll}
    p, \quad &\text{if } p\leq \frac{k}{k+1} \hat{v}_n\\
    k(\hat{v}_n-p)^+, \quad &\text{if } p > \frac{k}{k+1} \hat{v}_n.  
\end{array}\right.
\eeq

Under $k_{RS}^{PP}$, it is sufficient to find one feasible $p$ such that the satisficing constraint holds. Specifically when $N=2$, the constraint in problem \eqref{Prob_RS2} reduces to $\tau \leq \sup_p \sum_{n=1}^2 \alpha_n \rho_n(p,k)$. Geometric analysis shows that the supreme can only be obtained when $p = \frac{k_{RS}^{PP}}{k_{RS}^{PP}+1} \hat{v}_1$ or when $p=\frac{k_{RS}^{PP}}{k_{RS}^{PP}+1} \hat{v}_2$.

When $p = \frac{k_{RS}^{PP}}{k_{RS}^{PP}+1} \hat{v}_1$, $\sum_{n=1}^2 \alpha_n \inf_{v \in [0,1]} p\mathbbm{1}(v\geq p)+k_{RS}^{PP}\vert \hat{v}_n-v\vert = p = \frac{k_{RS}^{PP}}{k_{RS}^{PP}+1} \hat{v}_1$.

When $p = \frac{k_{RS}^{PP}}{k_{RS}^{PP}+1} \hat{v}_2$, $\sum_{n=1}^2 \alpha_n \inf_{v \in [0,1]} p\mathbbm{1}(v\geq p)+k_{RS}^{PP}\vert \hat{v}_n-v\vert = \max((1-\alpha_1)p,(1-\alpha_1) p+\alpha_1 k(\hat{v}_1-p)) = \max((1-\alpha_1)\frac{k_{RS}^{PP}}{k_{RS}^{PP}+1} \hat{v}_2,\alpha_1 k_{RS}^{PP}\hat{v}_1+(\frac{1}{k_{RS}^{PP}+1} -\alpha_1)k_{RS}^{PP}\hat{v}_2)$.

Thus, $\sup_p \sum_{n=1}^2 \alpha_n \inf_{v \in [0,1]} p\mathbbm{1}(v\geq p)+k_{RS}^{PP}\vert \hat{v}_n-v\vert = \frac{k_{RS}^{PP}}{k_{RS}^{PP}+1} \max (\hat{v}_1, (1-\alpha_1)\hat{v}_2,\alpha_1(k_{RS}^{PP}+1)\hat{v}_1+(1-\alpha_1-\alpha_1 k_{RS}^{PP}) \hat{v}_2)$. And if the maximum is attained by the first term in the maximization operator, then $p_{RS}^{PP} =  \frac{k_{RS}^{PP}}{k_{RS}^{PP}+1} \hat{v}_1$; otherwise, $p_{RS}^{PP} =  \frac{k_{RS}^{PP}}{k_{RS}^{PP}+1} \hat{v}_2$. Furthermore, the derivative of $\alpha_1 k_{RS}^{PP}\hat{v}_1+(\frac{1}{k_{RS}^{PP}+1} -\alpha_1)k_{RS}^{PP}\hat{v}_2$ w.r.t. $k_{RS}^{PP}$ is $\alpha_1 (\hat{v}_1-\hat{v}_2)+\frac{1}{(1+k_{RS}^{PP})^2}\hat{v}_2$.

When $\hat{v}_1\leq (1-\alpha_1)\hat{v}_2$, the maximum cannot be $\frac{k_{RS}^{PP}}{k_{RS}^{PP}+1}\hat{v}_1$, and thus $p_{RS}^{PP} =  \frac{k_{RS}^{PP}}{k_{RS}^{PP}+1} \hat{v}_2$. The maximum is $(1-\alpha_1)\frac{k_{RS}^{PP}}{k_{RS}^{PP}+1} \hat{v}_2$ when $k_{RS}^{PP} \geq \frac{\hat{v}_1}{\hat{v}_2-\hat{v}_1}$, and it is $\alpha_1 k_{RS}^{PP}\hat{v}_1+(\frac{1}{k_{RS}^{PP}+1} -\alpha_1)k_{RS}^{PP}\hat{v}_2$ when $k_{RS}^{PP} < \frac{\hat{v}_1}{\hat{v}_2-\hat{v}_1}$. Specifically, when $k_{RS}^{PP} = \frac{\hat{v}_1}{\hat{v}_2-\hat{v}_1}$, the maximum is $(1-\alpha_1)\hat{v}_1$. One can check that under the condition $\hat{v}_1 \leq (1-\alpha_1)\hat{v}_2$, $\frac{k_{RS}^{PP}}{k_{RS}^{PP}+1} \max ((1-\alpha_1)\hat{v}_2,\alpha_1(k_{RS}^{PP}+1)\hat{v}_1+(1-\alpha_1-\alpha_1 k_{RS}^{PP}) \hat{v}_2)$ is a monotonically increasing function of $k_{RS}^{PP}$. Thus, if $(1-\alpha_1)\hat{v}_1 < \tau <(1-\alpha_1) \hat{v}_2$, the optimal $k_{RS}^{PP}$ is the unique solution to the equation $(1-\alpha_1)\frac{k_{RS}^{PP}}{k_{RS}^{PP}+1} \hat{v}_2=\tau$. If $\tau \leq (1-\alpha_1)\hat{v}_1$, the optimal $k_{RS}^{PP}$ is the unique solution to the equation $\alpha_1 k_{RS}^{PP}\hat{v}_1+\frac{k_{RS}^{PP}(1-\alpha_1-\alpha_1 k_{RS}^{PP})}{k_{RS}^{PP}+1} \hat{v}_2=\tau$ that is smaller than $\frac{\hat{v}_1}{\hat{v}_2-\hat{v}_1}$, which is $\frac{\mu_0 - \tau - \sqrt{(\mu_0  - \tau)^2-4\alpha_1\tau(\hat{v}_2-\hat{v}_1)}}{2\alpha_1(\hat{v}_2-\hat{v}_1)}$.

Otherwise, when $\hat{v}_1 > (1-\alpha_1)\hat{v}_2$, the maximum is $\frac{k_{RS}^{PP}}{k_{RS}^{PP}+1} \hat{v}_1$ with the corresponding optimal price $p_{RS}^{PP} = \frac{k_{RS}^{PP}}{k_{RS}^{PP}+1} \hat{v}_1$ if $k_{RS}^{PP} \geq \frac{1}{\alpha_1}-1$, and it is $\alpha_1 k_{RS}^{PP}\hat{v}_1+\frac{k_{RS}^{PP}(1-\alpha_1-\alpha_1 k_{RS}^{PP})}{k_{RS}^{PP}+1} \hat{v}_2$ with the corresponding optimal price $p_{RS}^{PP} = \frac{k_{RS}^{PP}}{k_{RS}^{PP}+1} \hat{v}_2$ if $k_{RS}^{PP} < \frac{1}{\alpha_1}-1$. Specifically, when $k_{RS}^{PP} =\frac{1}{\alpha_1}-1$, the maximum is $(1-\alpha_1)\hat{v}_1$. One can check that under the condition $\hat{v}_1 > (1-\alpha_1)\hat{v}_2$, $\frac{k_{RS}^{PP}}{k_{RS}^{PP}+1} \max (\hat{v}_1,\alpha_1(k_{RS}^{PP}+1)\hat{v}_1+(1-\alpha_1-\alpha_1 k_{RS}^{PP}) \hat{v}_2)$ is a monotonically increasing function of $k_{RS}^{PP}$. Thus, if $(1-\alpha_1)\hat{v}_1 < \tau < \hat{v}_1$, the optimal $k_{RS}^{PP}$ is the unique solution to the equation $\frac{k_{RS}^{PP}}{k_{RS}^{PP}+1} \hat{v}_1=\tau$. If $\tau \leq (1-\alpha_1)\hat{v}_1$, the optimal $k_{RS}^{PP}$ is the unique solution on the interval $(0,\frac{1}{\alpha_1}-1)$ to the equation $\alpha_1 k_{RS}^{PP}\hat{v}_1+\frac{k_{RS}^{PP}(1-\alpha_1-\alpha_1 k_{RS}^{PP})}{k_{RS}^{PP}+1} \hat{v}_2=\tau$, which is $\frac{\mu_0 - \tau - \sqrt{(\mu_0  - \tau)^2-4\alpha_1\tau(\hat{v}_2-\hat{v}_1)}}{2\alpha_1(\hat{v}_2-\hat{v}_1)}$. \qed
\end{proof}

\subsection{Proof of Proposition \ref{prop-statistics-optRS}}
\begin{proof}
When $P_0$ is uniform, we have $u_1+w_1 = 1$, $w_1-u_1 = \frac{2\tau}{k_{RS}^*}$ and $u_1*w_1 = \frac{1}{4}-\frac{\tau^2}{k_{RS}^{*2}}$. Thus, 
\beq
\begin{aligned}  \mathbb{E}_{q_{RS}^*}[\tilde{p}] &= \int_{u_1}^{w_1} v\mathrm{d}q_{RS}^*(v) = k_{RS}^*(w_1-u_1) = 2\tau,\\
Var(\tilde{p}) &= \int_{u_1}^{w_1}(v-2\tau) \mathrm{d} q_{RS}^*(v) \\
&= \int_{u_1}^{w_1} k_{RS}^*(v-4\tau+\frac{4\tau^2}{v})\mathrm{d}v \\
&= k_{RS}^*\left(\frac{w_1-u_1}{2}+4\tau^2\ln\left(\frac{w_1}{u_1}\right)-4\tau \left(w_1-u_1\right)\right) \\
&= \tau(1-4\tau),\\
Skew(\tilde{p}) &= \int_{u_1}^{w_1}\left(\frac{v-2\tau}{\sqrt{\tau\left(1-4\tau\right)}}\right)^3 \mathrm{d}q_{RS}^*(v) \\
&=\frac{k_{RS}^*}{\sqrt{\tau^3(1-4\tau)^3}} \int_{u_1}^{w_1} (v^2-6\tau v+12\tau^2-8\tau^3/v)\mathrm{d}v \\
&= \frac{k_{RS}^*}{\sqrt{\tau^3(1-4\tau)^3}} \left(\frac{(w_1-u_1)(w_1^2 + w_1 u_1 + u_1^2)}{3} - 3\tau (w_1^2-u_1^2)-8\tau^3\ln\left(\frac{w_1}{u_1}\right)+12\tau^2\left(w_1-u_1\right)\right) \\
&= \frac{1}{\sqrt{\tau (1-4\tau)^3}}(\frac{2\tau^2}{3 k_{RS}^{*2}}+\frac{1}{2}-6\tau+16\tau^2).
\end{aligned}
\eeq
\qed
\end{proof}

\subsection{Proof of Theorem \ref{Thm_comparison_opt_pp_bs}}
\begin{proof}
For any given $v$, we can simplify $q_{RS}^*(v)v-m_{RS}^*(v)-\left(q_{RS}^{PP}\left(v\right)v-m_{RS}^{PP}\left(v\right)\right)$ to $ \int_0^v q_{RS}^*(x)\mathrm{d}x - \mathbbm{1}_{v\geq p_{RS}^{PP}} \left(v - p_{RS}^{PP}\right)$. Thus, when $v < p_{RS}^{PP}$, the expression is nonnegative. When $v \geq p_{RS}^{PP}$, $q_{RS}^*(v)v-m_{RS}^*(v)-\left(q_{RS}^{PP}\left(v\right)v-m_{RS}^{PP}\left(v\right)\right) = p_{RS}^{PP}-v+\int_0^v q_{RS}^*(x)\mathrm{d}x = p_{RS}^{PP}-\int_0^v(1-q_{RS}^*(x))\mathrm{d}x$, which is a monotonically decreasing function of $v$. Thus, there exists $\underline{v}_{RS}\in [p_{RS}^{PP},1]$ such that the OPT mechanism dominates the PP mechanism in buyer surplus when $v 
\leq \underline{v}_{RS}$. 

Next, consider the case when $P_0$ is a power distribution with $P_0(x) = x^{\alpha}$ ($\alpha>0$). According to Theorem \ref{Thm_optimalmechanism}, we have $q_{RS}^*(v) = k_{RS}^*\ln\left(\frac{v}{u_{RS}^*}\right)$ when $v\in [u_{RS}^*, w_{RS}^*)$. Otherwise, $q_{RS}^*(v)=0$ if $v < u_{RS}^*$ and $q_{RS}^*(v)=1$ if $v\geq w_{RS}^*$. Here, $u_{RS}^*$ and $w_{RS}^*$ satisfy that $\frac{w_{RS}^*}{u_{RS}^*} = \exp(1/k_{RS}^*)$, $u_{RS}^* \bar{P}_0(u_{RS}^*) = w_{RS}^*\bar{P}_0(w_{RS}^*)$, and $\tau = k_{RS}^*\int_{u_{RS}^*}^{w_{RS}^*}\bar{P}_0(x)\mathrm{d}x$. Thus, $\int_0^1 (1-q_{RS}^*(x))\mathrm{d}x = 1 - (\int_{u_{RS}^*}^{w_{RS}^*} q_{RS}^*(x)\mathrm{d}x+1-w_{RS}^*) = w_{RS}^*-(x q_{RS}^*(x)|_{x=u_{RS}^*}^{w_{RS}^*}-\int_{u_{RS}^*}^{w_{RS}^*} x \mathrm{d}q_{RS}^*(x)) = \int_{u_{RS}^*}^{w_{RS}^*} x \mathrm{d}q_{RS}^*(x) = k_{RS}^*(w_{RS}^*-u_{RS}^*) = \frac{w_{RS}^*-u_{RS}^*}{\ln(w_{RS}^*)-\ln(u_{RS}^*)}$. Since $\tau = \frac{\int_{u_{RS}^*}^{w_{RS}^*}\bar{P}_0(x)\mathrm{d}x}{\ln w_{RS}^* - \ln u_{RS}^*} = \frac{\int_{u_{RS}^{PP}}^{w_{RS}^{PP}}\bar{P}_0(x)\mathrm{d}x}{(w_{RS}^{PP}-u_{RS}^{PP})/u_{RS}^{PP}}$, thus $\frac{p_{RS}^{PP}}{\int_0^1 (1-q_{RS}^*(x))\mathrm{d}x} = \frac{\int_{u_{RS}^*}^{w_{RS}^*}\bar{P}_0(x)\mathrm{d}x}{w_{RS}^*-u_{RS}^*}/\frac{\int_{u_{RS}^{PP}}^{w_{RS}^{PP}}\bar{P}_0(x)\mathrm{d}x}{w_{RS}^{PP}-u_{RS}^{PP}}$. Given $\tau$, define $r^*:=\frac{w_{RS}^*}{u_{RS}^*}$, $r^{PP}:=\frac{w_{RS}^{PP}}{u_{RS}^{PP}}$, $c^*:=u_{RS}^* \bar{P}_0(u_{RS}^*)$ and $c^{PP}:=u_{RS}^{PP} \bar{P}_0(u_{RS}^{PP})$ then we have $r^*>r^{PP}$ (since $\ln(1+x)<x$ when $x>0$) and $c^*<c^{PP}$. Since $c^*$ ($c^{PP}$ respectively) uniquely determines $u_{RS}^*$ and $w_{RS}^*$ ($u_{RS}^{PP}$ and $w_{RS}^{PP}$ respectively), it will be sufficient to analyze $\frac{\int_{u_{RS}^*}^{w_{RS}^*}\bar{P}_0(x)\mathrm{d}x}{w_{RS}^*-u_{RS}^*}$ as a function of $c^*$. $\frac{\int_{u_{RS}^*}^{w_{RS}^*}\bar{P}_0(x)\mathrm{d}x}{w_{RS}^*-u_{RS}^*} = \frac{1}{\alpha+1}\cdot \frac{(w_{RS}^*)^{\alpha+1}-(u_{RS}^*)^{\alpha+1}}{w_{RS}^*-u_{RS}^*} = \frac{1}{\alpha+1}$. The last equality comes from the fact that $c^* = u_{RS}^*-(u_{RS}^*)^{\alpha+1}= w_{RS}^*-(w_{RS}^*)^{\alpha+1}$. Thus, $p_{RS}^{PP}= \int_0^1 (1-q_{RS}^*(x))\mathrm{d}x \geq \int_0^v (1-q_{RS}^*(x))\mathrm{d}x$, $\forall \, v\in[0,1]$.
\qed
\end{proof}

\subsection{Proof of Theorem \ref{Prop_Connection}}
\begin{proof}
This is a direct result of Theorem \ref{Prop_optimalPi}.
\qed
\end{proof}

\subsection{Proof of Proposition \ref{Prop_ComparisonU}}

\begin{proof}
Theorem \ref{Prop_optimalPi} shows that $$k_{RS}^* d(\pi^*(k_{RS}^*))+\pi^*(k_{RS}^*) = \tau = \Pi_{RO}^*(r).$$

Since $d(\pi^*(k_{RS}^*))$ must be positive, thus $\Pi_{RO}^*(r)>\pi^*(k_{RS}^*)$. Given that $u_1$ is a decreasing function, we can directly have that $u_1(\pi^*(k_{RS}^*)) < u_1(\Pi_{RO}^*(r))$. \qed
\end{proof}

\subsection{Proof of Theorem \ref{thm_universal-dominance-rs-ro}}
Before presenting the proof, we first establish the following lemma, which demonstrates that an increasing hazard rate guarantees specific conditions related to the sensitivity of revenue for two posted prices yielding the same revenue. In particular, proposition \ref{Prop_ContinuousSol} states that for a regular reference distribution, revenue is quasi-concave, meaning a given revenue corresponds to at most two posted prices. The subsequent lemma demonstrates that with an increasing hazard rate -- indicative of regularity -- the revenue function's sensitivity is greater to changes in the higher posted price than in the lower one between the two iso-revenue prices.
\begin{lem}\label{lem-ifr-sensitivity}
Assume the revenue target satisfies $\tau=\Pi^*_{RO}(r)$ and that the reference distribution $P_0$ has an increasing hazard rate. Then, for each $c\in[0,\Pi_0]$, the following condition holds:
\beq
\mathcal{R}_{P_0}'(u(c)) \leq \vartheta(c)\cdot\mathcal{R}_{P_0}'(w(c)), \mbox{ where } \vartheta(c) := \frac{\kappa(c)-\ln(\kappa(c))-1}{1/\kappa(c)+\ln(\kappa(c))-1}, \ \kappa(c) := \frac{w(c)}{u(c)}, \label{equ-ref-der-assumption}
\eeq
 and $u(c)\le w(c)$ are two posted prices yielding the same revenue $c$ (i.e., $\mathcal{R}_{P_0}(p) = c$).
\end{lem}

\begin{proof}
We start with the proof of Lemma \ref{lem-ifr-sensitivity}. 
Define the hazard rate
$h(x):=\frac{p_0(x)}{\bar P_0(x)}$, where $p_0$ is the density of $P_0$. For any $c\in[0,\Pi_0]$, we have $R'_{P_0}(x)=\bar P_0(x)-xp_0(x)=\bar P_0(x)\bigl(1-xh(x)\bigr),
$ and therefore $\frac{R'_{P_0}(u)}{R'_{P_0}(w)}=\kappa \frac{1-uh(u)}{1-wh(w)}$. Next, observe that $\ln\kappa
=\ln\frac{\bar P_0(u)}{\bar P_0(w)}=
\int_u^w h(x)\,dx$.
Hence
$\frac{\ln\kappa}{w-u}$
is the average value of the hazard rate on $[u,w]$. If $h$ is increasing, then
$h(u)\le \frac{\ln\kappa}{w-u}\le h(w)$.
Multiplying the left inequality by $u$ and the right inequality by $w$ yields
$uh(u)\le \frac{u\ln\kappa}{w-u}
=\frac{\ln\kappa}{\kappa-1}$,
and
$wh(w)\ge \frac{w\ln\kappa}{w-u}
=\frac{\kappa\ln\kappa}{\kappa-1}$.
Therefore,
$1-uh(u)\ge \frac{\kappa-1-\ln\kappa}{\kappa-1}$,
$1-wh(w)\le -\,\frac{\kappa\ln\kappa-\kappa+1}{\kappa-1}$.
Noting that $1-wh(w)<0$, we obtain
$\frac{R'_{P_0}(u)}{R'_{P_0}(w)}
\ge
\kappa\,
\frac{\frac{\kappa-1-\ln\kappa}{\kappa-1}}
{-\,\frac{\kappa\ln\kappa-\kappa+1}{\kappa-1}}=
-\frac{\kappa-\ln\kappa-1}{1/\kappa+\ln\kappa-1}
=\vartheta(c)$.
Since $R'_{P_0}(w)<0$, multiplying both sides by $R'_{P_0}(w)$ reverses the inequality and gives
$R'_{P_0}(u)\le \vartheta(c)\,R'_{P_0}(w)$, which completes the proof for Lemma \ref{lem-ifr-sensitivity}. \qed

Now, we are ready to establish the proof of Theorem \ref{thm_universal-dominance-rs-ro}. Let $u_{RS}^*$ and $w_{RS}^*$ denote the two solutions to the equation $\mathcal{R}_{P_0}(x) = \pi^*(k_{RS}^*)$, and $u_{RO}^*$ and $w_{RO}^*$ for the solutions to the equation $\mathcal{R}_{P_0}(x)=\tau$. The analysis in Proposition \ref{Prop_ComparisonU} shows that $u_{RS}^*<u_{RO}^*$ and $w_{RS}*>w_{RO}^*$. Since $q_{RS}*(0)=q{RS}^*(0)=0$ and $q_{RS}*(1)=q_{RS}^*(1)=1$, thus there exists a threshold $\underline{v}_q$ such that $q_{RS}^*(v)\geq q_{RO}^*(v)$ when $v\leq \underline{v}_q$.

The piecewise linear structure of $m^*(v)$ implies that if $m_{RO}^*(1)\geq m_{RS}^*(1)$, then there exists a threshold $\underline{v}_m$ such that $m_{RS}^*(v)\geq m_{RO}^*(v)$ when $v\leq \underline{v}_m$ and $m_{RS}^*(v)\leq m_{RO}^*(v)$ otherwise. Thus, to show the single-crossing property of the payment function, it is sufficient to show that $m_{RO}^*(1)\geq m_{RS}^*(1)$, or equivalently, $\int_0^1 q_{RS}^*(x)\mathrm{d}x \geq \int_0^1 q_{RO}^*(x)\mathrm{d}x$.

For the comparison over the buyer surplus $q(v)v-m(v)$, we start with the observation that $q(v)v-m(v) = \int_0^v q(x)\mathrm{d}x$. Thus, $(q_{RS}^*(v)v-m_{RS}^*(v))-(q_{RO}^*(v)v-m_{RO}^*(v)) = \int_0^v (q_{RS}^*(x)-q_{RO}^*(x))\mathrm{d}x$. If $\int_0^v (q_{RS}^*(x)-q_{RO}^*(x))\mathrm{d}x$ is positive for all $v\in [0,1]$, then the optimal RS mechanism dominate the optimal  For the buyers with valuation $v \leq \underline{v}_q$ and thus $q_{RS}^*(v)\geq q_{RO}^*(v)$, the integral $\int_0^v (q_{RS}^*(x)-q_{RO}^*(x))$ is positive. Further, $\int_0^v (q_{RS}^*(x)-q_{RO}^*(x))$ increases with $v$ when $v\leq \underline{v}_q$. When $v > \underline{v}_q$, the integral $\int_0^v (q_{RS}^*(x)-q_{RO}^*(x)) \mathrm{d}x$ decreases with $v$. Thus, if $\int_0^1 (q_{RS}^*(x)-q_{RO}^*(x))\mathrm{d}x$ is positive, then we can conclude that the single-crossing property of the payment function holds and the RS mechanism dominates the RO mechanism over the buyer surplus. 

Notice that $\int_0^1 q_{RS}^*(x)\mathrm{d}x = \int_{u_{RS}^*}^{w_{RS}^*} q_{RS}^*(x) \mathrm{d}x + 1-w_{RS}^* = 1-k_{RS}^*(w_{RS}^*-u_{RS}^*) = 1-\frac{w_{RS}^*-u_{RS}^*}{\ln(w_{RS}^*)-\ln(u_{RS}^*)}$, thus $\int_0^1 (q_{RS}^*(x)-q_{RO}^*(x))\mathrm{d}x = \frac{w_{RO}^*-u_{RO}^*}{\ln(w_{RO}^*)-\ln(u_{RO}^*)}-\frac{w_{RS}^*-u_{RS}^*}{\ln(w_{RS}^*)-\ln(u_{RS}^*)}$. It is thus sufficient to show that $\frac{\ln(w(c))-\ln(u(c))}{w(c)-u(c)}$ decreases in $c$ with $u(c)<w(c)$ being the intersection between $\bar{P}_0(x)$ and the level curve $c/x$. The derivative of $\frac{\ln(w(c))-\ln(u(c))}{w(c)-u(c)}$ is $\frac{1}{(w-u)^2}\cdot \left(\left(w-u\right)\left(\frac{w'}{w}-\frac{u'}{u}\right)-\left(w'-u'\right) \ln \frac{w}{u} \right) = \frac{1}{(w-u)^2}\cdot \left(w'\left(\frac{w-u}{w}-\ln \frac{w}{u} \right)+u'\left(-\frac{w-u}{u}+\ln \frac{w}{u}\right)\right) = \frac{1}{(w-u)^2}\cdot \left(\left(\frac{w-u}{w}-\ln \frac{w}{u} \right)/\mathcal{R}_{P_0}'(w)+\left(-\frac{w-u}{u}+\ln\left(\frac{w}{u}\right)\right)/\mathcal{R}_{P_0}'(u)\right)$. The last equality is obtained through the implicit differentiation of the equation $\mathcal{R}_{P_0}(u(c))=\mathcal{R}_{P_0}(w(c))=c$. Notice that $\mathcal{R}_{P_0}'(u)>0>\mathcal{R}_{P_0}'(w)$. Thus, the sign of the derivative of $\frac{\ln(w(c))-\ln(u(c))}{w(c)-u(c)}$ is the same as the sign of $(1-\frac{1}{\kappa(c)}-\ln \kappa(c) )\frac{\mathcal{R}_{P_0}'(u)}{\mathcal{R}_{P_0}'(w)}+(\ln \kappa(c)+1-\kappa(c))$. Since $1-\frac{1}{r}-\ln r < 0$ when $r>1$, $\left(1-\frac{1}{\kappa(c)}-\ln \kappa(c)\right)\frac{\mathcal{R}_{P_0}'(u)}{\mathcal{R}_{P_0}'(w)}+\left(1-\kappa(c) +\ln \kappa(c)\right) \leq 0$, which completes the proof.
\qed
\end{proof}



\subsection{Proof of Proposition \ref{Prop_ComparisonStatistics}}

\begin{proof}
We can see that $\mathbb{E}_{q_{RO}^*}[\tilde{p}] = \alpha(w_{RO}-u_{RO}) =\frac{\sqrt{1-4\tau}}{\ln(1+\sqrt{1-4\tau})-\ln(1-\sqrt{1-4\tau})} >\frac{\sqrt{1-4\Lambda(k_{RS}^*)}}{\ln(1+\sqrt{1-4\Lambda(k_{RS}^*)})-\ln(1-\sqrt{1-4\Lambda(k_{RS}^*)})} =\mathbb{E}_{q_{RS}^*}[\tilde{p}] $. 

Next, we show that $\mathbb{E}_{q_{RO}^*}[\tilde{p}] = \alpha(w_{RO}-u_{RO}) = \frac{w_{RO}-u_{RO}}{\ln w_{RO}-\ln u_{RO}}>p_{RO}^{PP} = \frac{1-\sqrt{2r}}{2}$. That is equivalent to show $2r>(1-2\frac{w_{RO}-u_{RO}}{\ln w_{RO}-\ln u_{RO}})^2$. Notice that $2r = w_{RO}-u_{RO}-2w_{RO}u_{RO}\ln \frac{w_{RO}}{u_{RO}}$ and $w_{RO}+u_{RO} = 1$. Thus, 
\beq
\begin{aligned}
&2r>(1-2\frac{w_{RO}-u_{RO}}{\ln w_{RO}-\ln u_{RO}})^2 \\
\Longleftrightarrow \quad
&\frac{w_{RO}-u_{RO}}{w_{RO}+u_{RO}}-2\frac{w_{RO}u_{RO}}{(w_{RO}+u_{RO})^2}\ln \frac{w_{RO}}{u_{RO}}>(1-\frac{w_{RO}-u_{RO}}{w_{RO}+u_{RO}}\frac{2}{\ln w_{RO}-\ln u_{RO}})^2 \\
\Longleftrightarrow \quad
&\frac{t-1}{t+1}-\frac{2t}{(t+1)^2}\ln t>(1-\frac{2(t-1)}{(t+1)\ln t})^2 \quad (t:=\frac{w_{RO}}{u_{RO}}> 1)\\
\Longleftrightarrow \quad
&2t\ln^3 t+2(t+1)\ln^2 t -4(t^2-1)\ln t+4(t-1)^2<0.
\end{aligned}
\eeq
Define $f(t):=2t\ln^3 t+2(t+1)\ln^2 t -4(t^2-1)\ln t+4(t-1)^2$. Then one can check that $f(t)$ decreases in $t\in (1,\infty)$. Given that $f(1)=0$, we confirm that $\mathbb{E}_{q_{RO}^*}[\tilde{p}] > p_{RO}^{PP}$.
\qed
\end{proof}

\end{APPENDICES}

\end{document}